\documentclass[12pt]{article}
\usepackage{lipsum}
\usepackage{pifont}
\usepackage{pinlabel}
\usepackage{rotating}
\usepackage{fancyhdr}
\usepackage[all]{xy}
\usepackage{amsmath, amsthm, latexsym, amssymb, graphicx,color,cite,mathabx,enumerate,mathrsfs}
\usepackage[left=2.0 cm,right=2.0cm,top=3cm,bottom=3cm]{geometry}

\usepackage{pinlabel} %For putting labels on figures

\usepackage{dsfont}
\usepackage{authblk}

\newtheorem{theorem}{Theorem}[section]
\newtheorem{lemma}[theorem]{Lemma}
\newtheorem{definition}[theorem]{Definition}
\newtheorem{corollary}[theorem]{Corollary}
\newtheorem{proposition}[theorem]{Proposition}

\newtheorem{remark}[theorem]{Remark}

\newtheorem{question}[theorem]{Question}
\newtheorem{example}[theorem]{Example}

\renewcommand{\b}{\mathfrak{b}}

\newcommand{\D}{\mathcal{D}}

\newcommand{\T}{\overline{T}}
\newcommand{\eps}{\varepsilon}
\newcommand{\Pcal}{\mathcal{P}}
\newcommand{\Lcal}{\mathcal{L}}
\newcommand{\ra}{\rightarrow}

\newcommand{\mc}[1]{\mathcal{#1}}
\newcommand{\defn}[1]{\textbf{#1}}
\newcommand{\boundary}{\partial}

\usepackage{amsmath}
\usepackage[svgnames]{xcolor}
\usepackage[framemethod=tikz]{mdframed}

\title{Bounding the Kirby-Thompson invariant of spun knots}

\author[Aranda, Pongtanapaisan, Taylor and Zhang]{Rom\'an Aranda, Puttipong Pongtanapaisan,\\ Scott A. Taylor, and Suixin (Cindy) Zhang}

\begin{document}
\pagenumbering{arabic}
\maketitle
\begin{abstract}
%Every smooth closed surface in $S^4$ admits a bridge trisection, similar to the bridge positions of classical knots. 
A bridge trisection of a smooth surface in $S^4$ is a decomposition analogous to a bridge splitting of a link in $S^3$. The Kirby-Thompson invariant of a bridge trisection measures its complexity in terms of distances between disc sets in the pants complex of the trisection surface. We give the first significant bounds for the Kirby-Thompson invariant of spun knots. In particular, we show that the Kirby-Thompson invariant of the spun trefoil is 15.
\end{abstract}

\section{Introduction}
Every smooth surface in the 4-sphere $S^4$ (or indeed any 4-manifold) admits a certain kind of decomposition known as a \emph{bridge trisection}. These bridge trisections are analogous to bridge positions of classical knots in $S^3$. They give rise to the fundamental notion of the \defn{bridge number} $\b(S)$ of a knotted smooth surface $S$. Bridge trisections and bridge number were defined by Meier and Zupan \cite{meier2017bridge} and are closely related to Gay and Kirby's trisections of smooth 4-manifolds \cite{gay2016trisecting}. The major advantage of both bridge trisections and trisections of 4-manifolds is that the handle structure of the knotted surface or 4-manifold is captured using 2-dimensional data on the trisection surface $\Sigma$. They also give rise to certain diagrammatic representations of knotted surfaces. In recent years, many authors have connected (bridge) trisections to major open problems in the theory of 2-knots and 4-manifolds \cite{PLC_thom, weinstein_trisections, gay2018doubly}.

One pressing problem has been to develop new 2-knot or 4-manifold invariants using trisections. In \cite{kirby2018new}, Kirby and Thompson defined a non-negative integer-valued 4-manifold invariant $\mc{L}(M)$ using the cut-complex of $\Sigma$. In \cite{blair2020kirby}, the third author and collaborators adapted Kirby and Thompson's definition to create an non-negative integer valued invariant $\mc{L}(S)$ of a smooth surface in $S^4$. They showed that for orientable $S$, if $\mc{L}(S) = 0$ then $S$ is an unlink. They also showed that for a connected, irreducible surface $S$, $\mc{L}(S) > \b(S) - g(S) - 2$, where $g(S)$ is the genus of $S$. Using spun knots, Meier and Zupan show that $\b(S)$ can be arbitrarily large for 2-knots $S$; consequently $\mc{L}(S)$ can be as well. However, for spun 2-bridge knots, the only previously known lower bound is that $\mc{L}(S)$ is nonzero. Calculating $\mc{L}(S)$ for specific surfaces remains a challenging problem, as does showing that for fixed bridge number $\mc{L}(S)$ can be arbitrarily large. In this paper, we take steps toward those questions by showing:
\begin{theorem}\label{thm_main}
Let $K\subset S^3$ be a 2-bridge knot with Conway number $p/q$. We have 
\[15\leq \Lcal(S(K)) \leq \min \big\{ 6d(p/q,0) + 6,  6d(p/q,\infty) + 9\big\}.\] 
In particular, if $K$ is a trefoil knot $3/1$, then $\Lcal(S(K))=15$.
\end{theorem}

\begin{proof}
The lower bound and upper bounds are proven in Corollaries \ref{cor_lower_2b} and \ref{cor_upper_2bridges}, respectively. 
\end{proof}

More generally, we construct estimates for any spun knot.  For a trivial $N$-tangle $T$, we define $\Pcal_{comp}(T)$ and $\Pcal_{c}(T)$ to be the sets of pants decompositions in the pants complex $p\in \Pcal(\Sigma_{2N})$ such that all loops in $p$ bound compressing disks and c-disks, respectively. 

\begin{theorem}
Let $K=T^+_K \cup T^-_K$ be a knot in $b$-bridge position. % and let $\mathcal{T}_{MZ}$ be the $(3b-2,b)$-bridge trisection for the spun 2-knot $S(K)\subset S^4$. 
Let $d\geq 0$ be the distance in $\Pcal(\Sigma_{2b})$ between the sets $\Pcal_c(T^+_K)$ and $\Pcal_{comp}(T^-_K)$. Then 
\[ 6b-8\leq \Lcal(S(K))\leq 6(d+b-1).\]
\end{theorem}

\begin{proof}
The upper bound is proven in Theorem \ref{thm_upper_bbridges} for a particular minimal bridge trisection of $S(K)$. Since $\Lcal(S(K))$ is the minimum value of $\Lcal(\mc{T})$ along all minimal bridge trisections of $S(K)$ (see Section \ref{section_L_invariant}), the upper bound holds. The lower bound is Theorem 6.3 of \cite{blair2020kirby}. 
\end{proof}

The invariant $\mc{L}(\mc{T})$ for a bridge trisection $\mc{T}$ with trisection surface $\Sigma$ is defined using the pants complex of $\mc{T}$ and the associated disc complexes (see Section \ref{section_L_invariant}). Most of the delicate combinatorial work in this paper consists of a careful analysis of paths in the pants complex. Our techniques may, therefore, also be of interest to those working on surface dynamics. In fact, most of our work in Section \ref{section_combinatorics} focuses in understanding the combinatorics of (4,2)-bridge trisections. We show

\newtheorem*{thm_lower_bound}{Theorem \ref{thm_lower_bound}}
\begin{thm_lower_bound}
Let $\mc{T}$ be a $(4,2)$-bridge trisection for a knotted connected surface $F$ in $S^4$. Then $$L(\mc{T})\geq 15.$$ %Moreover if $\mc{T}$ is of type 1 or 2 from Lemma \ref{lem_combinatorics}, then $\L(\mc{T})\geq 15$. 
\end{thm_lower_bound}

In \cite{meier2017bridge}, Meier and Zupan described bridge trisection diagrams $\mc{T}_{MZ}$ for twist spun knots. Even though $(\pm 1)$-twist 2-bridge knots are unknotted, it is unclear whether their bridge trisections $\mc{T}_{MZ}$ are stabilized. They form a family of candidates of non-stabilized non-minimal bridge trisections. In order to disprove this, one could try to build upper bounds for $\Lcal(\mc{T}_{MZ})$ of $(\pm 1)$-twist spun knots and use Theorem \ref{thm_lower_bound} to see they are stabilized. 

\subsection*{Acknowledgements}
Taylor was partially supported by NSF Grant DMS-2104022. Aranda and Pongtanapaisan were partially supported by a grant from the Berger Fund and Colby College. The second author acknowledges the Pacific Institute for the Mathematical Sciences for the support. We are grateful to Nathaniel Ferguson for helpful conversations and to Jeffrey Meier for suggesting this project.
%%%%

%——————————————————————————————————————————————————————————————————————————————————————————————————
\section{Preliminaries}

In this section, we introduce terminology and recall the definitions of the pants complex, a genus-0 trisection of $S^4$ and bridge trisections, and the invariant $\mathcal{L}$. For more detailed explanations please refer to \cite{meier2017bridge,blair2020kirby}

\subsection{The pants complex}
Suppose that $\Sigma$ is a compact surface with punctures. A simple closed curve $\gamma\subset\Sigma$ is called \textbf{essential} if it is disjoint from the punctures, does not bound an unpunctured or once-punctured disk in $\Sigma$, and does not cobound an unpunctured annulus in $\Sigma$ with $\partial \Sigma$. If $\Sigma$ is a sphere, we define the \textbf{inside} of a simple closed curve in $\Sigma$ to be the sides with the least punctures punctures and the \textbf{outside} to be a side that is not an inside. Some curves  have two inside regions and no outside region. We say that a simple closed curve in a sphere $\Sigma$ is an \defn{odd curve} if the number of punctures on each side is odd and an \defn{even curve} otherwise.

A \textbf{pair-of-pants} is a sphere with three punctures, an annulus with one puncture, or a disk with two punctures. A \textbf{pants decomposition} of $\Sigma$ is a collection of pairwise disjoint essential curves cutting $\Sigma$ into pairs of pants. Pants decompositions are considered up to isotopy. If $\Sigma$ is a sphere with $2b \geq 4$ punctures, then each pants decomposition of $\Sigma$ has $2b-3$ curves. Define $P(\Sigma)$, the \textbf{pants complex}\footnote{It is possible to define higher dimensional simplices of $\mc{P}(\Sigma)$, but we will not make use of them.} of $\Sigma$, as follows. Each pants decomposition of $\Sigma$ is a vertex of $P(\Sigma)$. Two vertices are connected by an edge if the two corresponding pants decompositions have all but one (isotopy class of) curve in common and the two curves where they differ (have representatives that) intersect minimally in exactly two points. We say that the two endpoints of an edge differ by an \textbf{A-move}. The distance $d(x, y)$ between two collections of vertices $x$ and $y$ in $P(\Sigma)$ is the minimum number of edges in a path in $P(\Sigma)$ between a vertex of $x$ and a vertex of $y$. For a path $\alpha$ in $\mc{P}(\Sigma)$, we say that a curve $\gamma \subset \Sigma$ is \defn{unmoved} on $\alpha$ if it (up to isotopy) belongs to every vertex of $\alpha$. On the other hand, if we have a path from vertex $a$ to vertex $b$ and if $c$ is a curve in a pants decomposition $x$ that is a vertex of the path, then if the edge of the path leaving $x$ corresponds to an $A$-move replacing $c$ with $c'$, we say that $c$ is \defn{moved} by the path and write $c \mapsto c'$. Clearly, the length of the path is at least the number of curves moved by the path. Some curves may be moved multiple times so it need not be equal to the number of curves that are moved.

A \textbf{trivial tangle} $(B_\delta, \delta)$ is a 3-ball $B_\delta$ containing properly embedded arcs $\delta$ such that, fixing the endpoints of $\delta$, we may isotope $\delta$ into $\partial B_\delta$. We consider the endpoints of $\delta$ on $\Sigma = \boundary B_\delta$ to be punctures on $\Sigma$. A \defn{c-disc} in $(B_\delta, \delta)$ is a properly embedded disc $D \subset B_\delta$ transverse to $\delta$, with $\boundary D$ essential in the (punctured) surface $\Sigma$, and with $|D \cap \delta| \leq 1$. The c-disc $D$ is a \defn{compressing disc} if $|D \cap \delta| = 0$ and a \defn{cut-disc} otherwise. The \defn{disc set} $\mc{D}(B_\delta, \delta)$ for $(B_\delta, \delta)$ consists of the vertices $v$ of $\mc{P}(\Sigma)$ such that each curve in the pants decomposition $v$ bounds a c-disc in $(B_\delta, \delta)$. 

Each arc $\delta_0$ of a trivial tangle $(B_\delta, \delta)$ admits a disc $D$ such that $\boundary D$ is the endpoint union of $\delta_0$ with an arc on $\boundary B_\delta$ and with interior disjoint from $\delta$. Such a disc is called a \defn{bridge disc} and the arc on $\boundary B_\delta$ is a \defn{shadow arc}. There are a collection of pairwise disjoint bridge discs so that each arc of $\delta$ belongs to a bridge disc. The union of all the shadow arcs for such a collection of bridge discs is a \defn{complete shadow arc collection}.

For a link $L \subset S^3$, a decomposition $(S^3, L) = (B_\lambda, \lambda) \cup_\Sigma (B_\tau, \tau)$, where each pair $(B_\delta, \delta)$ is a trivial tangle, is called a \defn{bridge splitting}. The surface $\Sigma=\partial B_i$ for $i=\lambda, \tau$ is the \defn{bridge sphere} of the splitting.  An \textbf{efficient defining pair} is a pair of pants decomposition $(\D_\kappa, \D_\lambda)$ with $x\in  \D_\kappa$ and $y \in \D_\lambda$ such that $d(x, y) = d(\D_\kappa, \D_\lambda)$. Zupan \cite{zupan2013bridge} uses this distance to define a knot invariant for knots in $S^3$. We need the following well-known result (see \cite{bachman2005distance,zupan2013bridge}):
\begin{lemma}\label{std unlink surface}
Suppose that $\Sigma$ is a bridge sphere for an unlink $L \subset S^3$, then:
\begin{enumerate}
    \item If $|L| \geq 2$, there is a sphere $P \subset S^3$ intersecting $\Sigma$ in a single essential simple closed curve and separating components of $L$. Such a sphere is called a \defn{reducing sphere} for $\Sigma$.
    \item If $L_0$ is a component of $L$ such that $|L_0 \cap \Sigma| = 2$, then there is a disc with boundary equal to $L_0$ and interior disjoint from $L$ such that $L_0 \cap \Sigma$ is a single arc. Furthermore, given a collection of pairwise disjoint reducing spheres, there is such a disc disjoint from them.
    \item If $L_0$ is a component of $L$ such that $|L_0 \cap \Sigma| \geq 4$, then there exist discs $D_1$ and $D_2$ on opposite sides of $\Sigma$ such that:
    \begin{enumerate}
        \item For $i = 1,2$, $\boundary D_i$ is the endpoint union of a strand of $L \setminus \Sigma$ and an arc on $\Sigma$;
        \item For $i = 1,2$, the interior of $D_i$ is disjoint from $L \cup \Sigma$;
        \item $D_1 \cap D_2$ is a single point (necessarily a puncture of $\Sigma$).
    \end{enumerate}
    In this case, we say that $L$ is \textbf{perturbed} and call the discs $D_1$ and $D_2$ a \defn{perturbing pair}. Furthermore, given a collection of pairwise disjoint reducing spheres, there exists a perturbing pair disjoint from them.
\end{enumerate}
\end{lemma}

\begin{definition}\label{def_curve_types}
For a link $L$ in $S^3$ with bridge sphere $\Sigma$, the intersection of a reducing sphere with $\Sigma$ is called a \defn{reducing curve} for $(S^3, L)$ on $\Sigma$. Notice that an essential curve is a reducing curve if and only if it bounds compressing discs for $\Sigma$ in both of the trivial tangles on either side of $\Sigma$. Similarly, if $\gamma \subset \Sigma$ is a curve bounding cut discs on both sides of $\Sigma$, then $\gamma$ is a \defn{cut-reducing curve} for $(S^3, L)$ on $\Sigma$.
\end{definition}

\subsection{Bridge trisections}

Suppose that $S$ is a smooth, closed surface in $S^4$. A \defn{bridge trisection} $\mc{T}$ with trisection surface $\Sigma$ (a sphere) is defined as follows\footnote{It is possible to define higher genus bridge trisections \cite{meier2018bridge}, but we will not need them in this paper.}. Suppose that $W_1$, $W_2$, and $W_3$ are 4-balls in $S^4$ such that $W_{i} \cap W_{j}$ is a 3-ball $B_{ij}$ (for $i \neq j$) and that 
\[W_1 \cap W_2 \cap W_3 = B_{12} \cap B_{23} \cap B_{13}\]
is a smooth 2-sphere $\Sigma$. Then we say that $S^4 = W_1 \cup W_2 \cup W_3$ is a \defn{0-trisection} of $S^4$ \cite{gay2016trisecting}. Suppose also that each of $B_{12}$, $B_{23}$, and $B_{13}$ are transverse to $S$ and that $\Sigma$ and $S$ intersect transversally in $2b$ points and that:
\begin{enumerate}
    \item For each $i \in \{1,2,3\}$, $S \cap W_i$ is a trivial disk system;
    \item For each $\{i,j,k\} = \{1,2,3\}$, in $B_{ij} \cup B_{jk}$, the sphere $\Sigma$ is a bridge surface for the link $S \cap (B_{ij} \cup B_{jk})$;
    \item For each $\{i,j,k\} = \{1,2,3\}$, the link $S \cap (B_{ij} \cup B_{jk})$ is an unlink of $c_j$ components.
\end{enumerate}

We call $\mathcal{S}=(B_{12},T_{12})\cup (B_{23},T_{23})\cup(B_{31},T_{31})$ the \textbf{spine} of the bridge trisection and $\Sigma$ the \textbf{bridge surface} of $S$. The numbers $c_1, c_2, c_3$ are the \defn{patch numbers} of the bridge trisection. The \defn{bridge number} $\b(\mc{T})$ of the trisection is $\b(\mc{T}) = |S \cap \Sigma|/2$ and the \defn{bridge number} $\b(S)$ of $S$ is the minimum of $\b(\mc{T})$ over all bridge trisections $\mc{T}$ for $S$. We say that a trisection $\mc{T}$ with bridge number $b$ and patch numbers $c_1, c_2, c_3$ is a $(b; c_1, c_2, c_3)$-bridge trisection. As we mentioned, the definitions of bridge trisection and bridge number are due to Meier and Zupan, who also prove that every smooth surface admits a bridge trisection.  We let $\mathcal{D}_{ij} \subset \mc{P}(\Sigma)$ be the disk set of the tangle $(B_{ij},T_{ij})$. 

Meier and Zupan also introduce in \cite{meier2017bridge} the notion of a \textbf{tri-plane diagram}: a triple of planar tangle diagrams whose pairwise unions are unlinks. Since a bridge trisection is determined by its spine consisting of a triple of 3-balls $B_{12}, B_{23}, B_{31}$ with trivial tangles $T_{12}, T_{23}, T_{31}$, one can project the tangle $T_ij$ onto a vertical disk in $B_ij$ respectively and obtain three planar tangle diagrams. In particular, every knotted surface in $S^4$ can be represented by a tri-plane diagram which is unique up to interior Reidemeister moves, bridge sphere braiding, and perturbation and deperturbation. See Section 2 in \cite{meier2017bridge} for details.
%Finally, we note that if a surface $S$ has a $(b; c_1, c_2, c_3)$ trisection, then $\chi(S) = c_1 + c_2 + c_3 - b$.

\begin{lemma}\label{all 2}
Suppose that $S \subset S^4$ is a topologically knotted sphere with a $(4; c_1, c_2, c_3)$-trisection and $4 = \b(S)$. Then $c_i = 2$ for all $i$.
\end{lemma}
\begin{proof}
Since $S$ is topologically knotted, by \cite[Corollary 1.12]{meier2017bridge}, $c_i \geq 2$ for all $i$. The result follows since $2=\chi(S) = c_1 + c_2 + c_3 -4$.
%By the aforementioned relationship between the bridge number, patch numbers, and Euler characteristic, each $c_i \leq 2$.
\end{proof}

Henceforth, we abbreviate the phrase ``$(4; 2,2,2)$-trisection'' to $(4,2)$-trisection.

\subsection{Spun knots}\label{section_spun_knots}

%Define spun knot. 
We now recall a construction of spun knots from a knot $K\subset S^3$ due to Artin \cite{artin1925isotopie}. Let $(B^3, K^\circ)$ be the result of removing a small, open ball centered on a point in $K$, so that $K$ is a knotted arc with endpoints on the north and south poles, labeled $n$ and $s$ respectively. Then, the spin $S(K)$ of $K$ is the knotted surface given by \[(S^4,S(K))=\left((B^3, K^\circ)\times S^1\right)\cup \left((S^2,\{n,s\})\times D^2\right).\]

Meier and Zupan also show that every spun $b$-bridge knot $S(K)\in S^4$ has bridge number at most $3b -2$ by providing an explicit $(3b-2,b)$-bridge trisection, whose corresponding tri-plane diagram is shown below in Figure \ref{fig_spun_knot}. From now on, we will denote this particular bridge trisection by $\mathcal{T}_{MZ}$ and, for that trisection, define $T_{ij}$ as indicated for $i,j \in \{1,2,3\}$ with $i \neq j$. 

\begin{remark}
For this particular trisection $\mathcal{T}_{MZ}$ for a spun $b$-bridge knot, since $\b(\mathcal{T}_{MZ})=3b-2$ and $c_i=b$ for all $i\in\{1,2,3\}$, the corresponding bridge sphere is $2\b(\mathcal{T}_{MZ})$-punctured, and each pants decomposition $p_{ij}^i$ has exactly $2\b(\mathcal{T}_{MZ})-3=2(3b-2)-3=6b-7$ curves. Thus, it follows from Lemma \ref{lem_efficient_pairs} that there exist $p_{ij}^i\in \mathcal{D}_{ij}$ and $p_{ki}^i\in \mathcal{D}_{ik}$ with $d(p_{ij}^i, p_{ki}^i) = \b(\mathcal{T})-c_i=(3b-2)-b=2b-2$.
\end{remark}

\begin{figure}[h]
\centering
\includegraphics[width=1\textwidth]{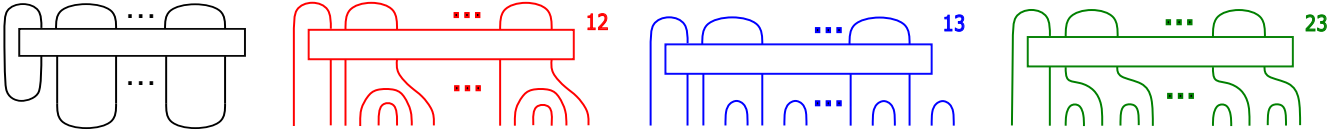}
\caption{A $(3b-2,b)$-bridge tri-plane diagram for the spin $\mathcal{S}(K)$ of the $b$-bridge knot $K$ given in bridge position (left). We will denote the tangles by $T_{12}, T_{13}$, and $T_{23}$ from left to right.}
\label{fig_spun_knot}
\end{figure}

We note the following:

\begin{theorem}[Meier-Zupan \cite{meier2017bridge}]\label{thm_spun_are_42}
If $K \subset S^3$ has $\b(K) = 2$, then $\b(\mathcal{S}(K)) = 4$. Consequently, if $\mc{T}$ is a $(4; c_1, c_2, c_3)$-trisection for a spun 2-bridge knot, then each $c_i = 2$.
\end{theorem}

\begin{proof}
We defer to \cite[Section 5]{meier2017bridge} for details. Let $\mc{T}$ be a $(b; c_1, c_2, c_3)$ bridge trisection of a spun 2-bridge knot $\mc{S}(K)$. By Corollary 5.3 and Theorem 5.5 of \cite{meier2017bridge}:
\[
\min(c_1, c_2, c_3) \geq \operatorname{mrk}(\mc{S}(K)) = \operatorname{mrk}(K),
\]
where $\operatorname{mrk}$ is the ``meridional rank'' of the 2-knot or knot. By \cite{boileau1989pi}, $\operatorname{mrk}(K) = 2$, so $c_i \geq 2$ for all $i$. Also, \[2 = \chi(\mc{S}(K)) = c_1 + c_2 + c_3 - b \geq 6 - b.\]
Thus, $b \geq 4$. Since Meier and Zupan have constructed trisections of spun 2-bridge knots of bridge number 4, $\b(\mc{S}(K)) = 4$. Since the meridional rank of $\mc{S}(K) = 2$, $\mc{S}(K)$ is topologically knotted. The result follows from Lemma \ref{all 2}.
\end{proof}

\subsection{The Kirby-Thompson Invariant}\label{section_L_invariant}

We now define the Kirby-Thompson invariant of a bridge trisection. For a schematic diagram of the efficient defining pairs for a trisection, see Figure \ref{fig_L_invariant}.

\begin{definition}[Kirby-Thompson Invariant $\mathcal{L}$]\label{def_L_inv}
Suppose that $S\subset S^4$ is knotted surface with bridge trisection $\mathcal{T}$ having trisection surface $\Sigma$ and spine $\mathcal{S}=(B_{12},T_{12})\cup (B_{23},T_{23})\cup(B_{31},T_{31})$. For $\{i, j, k\} = \{1, 2, 3\},$ let $(p^j_{ij}, p^j_{jk})$ be an efficient defining pair for $(B_{ij} , T_{ij} )\cup_{\Sigma} (B_{jk}, T_{jk}).$ If $\Sigma$ is a sphere with strictly less than 4 punctures, define $\mathcal{L}(\mathcal{T}) = 0.$ Otherwise, define the \defn{Kirby-Thompson invariant} $\mathcal{L}(\mc{T})$ to be the minimum of 
\[
d(p^1_{12}, p^2_{12}) + d(p^2_{23}, p^3_{23}) + d(p^1_{31}, p^3_{31})
\] 
over all such choices of efficient defining pairs. Define the \defn{Kirby-Thompson invariant} $\mathcal{L}(S)$ to be the minimum of $\mathcal{L}(\mathcal{T} )$ over all trisections $\mathcal{T}$ of $S$ with $\b(\mathcal{T}) = \b(S)$.
\end{definition}

Moreover, the distance between an efficient defining pair in the setting of Definition \ref{def_L_inv} is determined. 
\begin{lemma}[Lemma 5.6 of \cite{blair2020kirby}]\label{lem_efficient_pairs}
If $\mc{T}$ is a $(\b(\mc{T}),c_1, c_2, c_3)$-bridge trisection, then every efficient defining pair satisfies 
\[d(p^i_{ij}, p^i_{ik}) = \b(\mc{T})-c_i.\] 
\end{lemma} 

{\color{blue}
\begin{figure}[h!] \label{fig_L-invarant}
\labellist \small\hair 2pt  
\pinlabel {$T_{12}$}  at 100 295 
\pinlabel {{\color{white}$T_{13}$}}  at 450 505 
\pinlabel {{\color{white}$T_{23}$}}  at 510 105 
\pinlabel \scriptsize$p_{12}^1$  at 166 376 
\pinlabel \scriptsize$p_{12}^2$  at 162 200 
\pinlabel \scriptsize$p_{23}^2$  at 386 90 
\pinlabel \scriptsize$p_{13}^1$  at 356 490 
\pinlabel \scriptsize$p_{13}^3$  at 519 401 
\pinlabel \scriptsize$p_{23}^3$  at 522 216 
\pinlabel {
\begin{rotate}{-24}
{\scriptsize
$\b(\mathcal{T})-c_2$
}
\end{rotate}
}
at 170 70
\pinlabel {
\begin{rotate}{30}
{\scriptsize
$\b(\mathcal{T})-c_1$
}
\end{rotate}
}
at 150 480
\pinlabel 
\scriptsize
$\b(\mathcal{T})-c_3$
at 700 295 
\endlabellist \centering 
\includegraphics[width=0.35\textwidth]{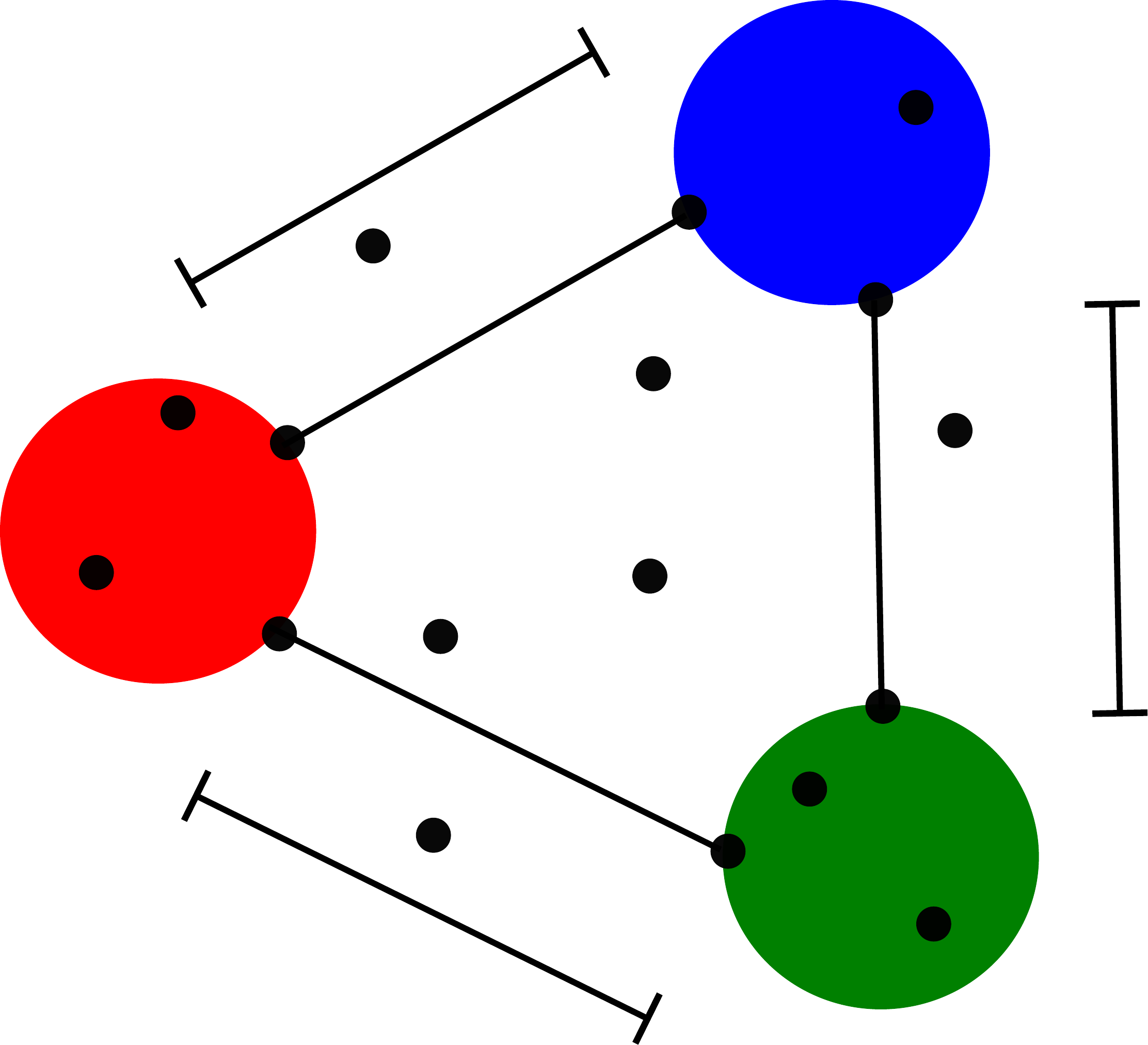}
\caption{Defining $\mathcal{L}(T)$ via efficient defining pairs. The ellipses represent the disk sets. The line joining $p_{ij}^{i}$ to $p_{ij}^{j}$ represents a geodesic path in the pants complex, which has length $\b(\mathcal{T})-c_i$ for a $(\b(\mathcal{T}),c_1, c_2, c_2)$-bridge trisection.}
\label{fig_L_invariant}
\end{figure}
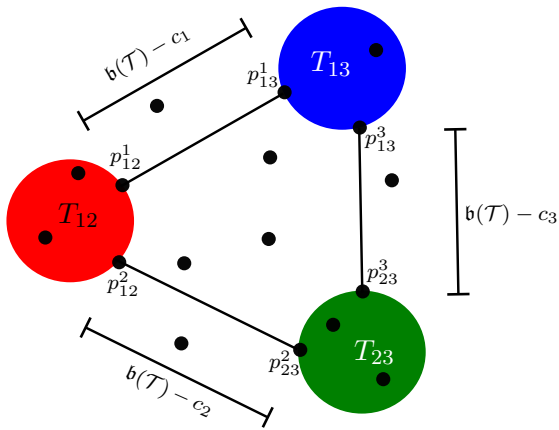}

%%%%%%%%%%%%%%%%%%%%%%%%%%%%%%%%%
\subsection{Reducibility and Stabilization of Bridge Trisection}

We provide two related ways in which a bridge trisection may have higher bridge number than necessary: reducibility and stabilization. 

\begin{definition}
Given two trisections $\mc{T}_i$ for surfaces $S_i$ ($i = 1,2$) in distinct copies of $S^4$, their \defn{distant sum} is the trisection obtained by taking the connected sum of the two copies of $S^4$ using a point on each trisection surface disjoint from the surfaces. Their \defn{connected sum} is the trisection obtained by taking the connected sum of the two copies of $S^4$ using punctures on the two trisection surfaces. For more details see \cite{meier2017bridge}. A trisection with trisection surface $\Sigma$ is \defn{reducible} if there exists an essential simple closed curve in $\Sigma$ bounding a c-disk in each tangle forming the spine. 
\end{definition}

\begin{lemma}\label{reduciblenonmin}
If $S$ is a knotted 2-sphere with $\b(S) \leq 7$, then no bridge trisection of minimal bridge number is reducible.
\end{lemma}
\begin{proof}
As explained in \cite{blair2020kirby}, if a trisection $\mc{T}$ were a reducible (4,2)-bridge trisection for $S$, then it would be the connected sum of two other trisections $\mc{T}_1$ and $\mc{T}_2$, such that $\mathfrak{b}(\mc{T}_1) + \mathfrak{b}(\mc{T}_1) = \b(\mc{T}) + 1 \leq 7$ and each has bridge number at least 2.  In particular, either $\mc{T}_1$ or $\mc{T}_2$ would have bridge number at most 3, implying that the corresponding surface is unknotted by \cite[Theorem 1.8]{meier2017bridge}. In which case, the other trisection is a trisection for $S$ of smaller bridge number than $\mc{T}$.
\end{proof}

\begin{lemma}%[Key Lemma]
\label{lem_1}
Suppose that $\mc{T}$ is a bridge trisection with spine $\bigcup\limits_{i \neq j}(B_{ij}, T_{ij})$. Then $\mc{T}$ is reducible or stabilized if and only if there is an essential curve $\gamma$ bounding a c-disk in each $(B_{ij}, T_{ij})$. Furthermore, such a curve is a reducing or cut-reducing curve (respectively) for each link $L_j=T_{ij}\cup\T_{jk}$.
\end{lemma}
\begin{proof}
This follows easily from Lemma \ref{std unlink surface}. 
\end{proof}

In \cite[Section 6]{meier2017bridge}, Meier and Zupan define what it means for a bridge trisection to be \defn{stabilized}. This is the analogous to a ``perturbed bridge surface'' for knots in 3-manifolds or to ``stabilized Heegaard splittings'' of 3-manifolds. While we do not need the precise definition of stabilization, we need the following two results, both from \cite{meier2017bridge}.

\begin{lemma}%[{\cite[Section 6]{meier2017bridge}}] 
\label{stabnonmin}
If $S \subset S^4$, then no stabilized bridge trisection of $S$ has minimal bridge number.
\end{lemma}

\begin{lemma}[{Stabilization Criterion \cite[Lemma 6.2]{meier2017bridge}}]\label{gen_criterion_destab}
Let $\mc{T}$ be a bridge trisection with spine \[(B_{12},T_{12})\cup (B_{23},T_{23})\cup(B_{31},T_{31}).\] If for some $\{i,j,k\} = \{1,2,3\}$, there exists a collection of shadow arcs $\alpha$ for $(B_{ij},T_{ij})$ and $\beta$ for $(B_{jk},T_{jk})$ and a single shadow arc $\gamma$ for $(B_{ik},T_{ik})$ such that the interiors of all the shadow arcs are disjoint and the following two conditions hold, then $\mc{T}$ is stabilized:
\begin{enumerate}
    \item The union $\alpha \cup \beta$ is a simple closed curve (ignoring the punctures)
    \item Exactly one endpoint of $\gamma$ lies on $\alpha \cup \beta$.
\end{enumerate}
\end{lemma}

Noting that the union of an arc with an isotopic copy having interior disjoint from the original is a circle, produces the following criterion we'll use repeatedly.

\begin{lemma}\label{criterion_destab}
Let $\mc{T}$ be a bridge trisection with spine \[(B_{12},T_{12})\cup (B_{23},T_{23})\cup(B_{31},T_{31}).\] Suppose that there exist $\{i,j,k\} = \{1,2,3\}$ so that there is a shadow arc $\alpha$ for both $(B_{ij},T_{ij})$ and $(B_{jk},T_{jk})$ and a shadow arc $\gamma$ for $(B_{ik},T_{ik})$ sharing  exactly one endpoint with $\alpha$ and with interior disjoint from $\alpha$. Then $\mc{T}$ is stabilized.
\end{lemma}
\begin{figure}[h]
\centering
\includegraphics[width=.3\textwidth]{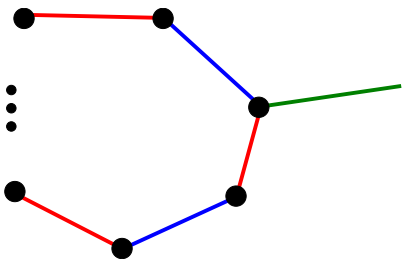}
\caption{The arrangement of arcs from Lemma \ref{gen_criterion_destab}.}
\end{figure}

We note that in \cite{blair2020kirby}, the authors show that if a $(b;c_1, c_2, c_3)$-bridge trisection $\mc{T}$ of a knotted surface $S$ is not reducible, then
\[
\mc{L}(\mc{T}) \geq 2(c_1+c_2+c_3)-8.
\]
If $\mc{T}$ is a (4,2)-bridge trisection, this inequality translates to $\Lcal (\mc{T})\geq 2\cdot6-8 = 4$. The goal of Section \ref{section_combinatorics} is to further improve this estimate in Theorem \ref{thm_lower_bound}.

%%%%%%%%%%%%%%%%%%%%%%%%%%%%%%%
\section{Combinatorics of $(4,2)$-bridge trisections}\label{section_combinatorics}

This section studies relations among pairs of pants decompositions of a trisection surface $\Sigma$ having 8 punctures. For each $\{i,j,k\}=\{1,2,3\}$, the link $L_i=T_{ij}\cup \T_{ik}$ is a 2-component unlink in 4-bridge position. %Lemma \ref{lemma_pants_structure} describes the type of curves belonging in an efficient pair and Lemma \ref{lem_combinatorics} determines what punctures are bounded by the reducing curves. Lemmas \ref{lemma_red_intersect} and \ref{lemma_psi_f_distinct} give constraints to how the curves in $p^i_{ij}$ and $p^j_{ij}$ interact. Lemma \ref{remark_reverse_lem} gives a necessary condition for path between $p^i_{ij}$ and $p^j_{ij}$ to be geodesic {\color{red} Update this summary}. 
%Let $(p^i_{ij}, p^i_{ik})$ and $(p^j_{ij}, p^j_{jk})$ be two efficient defining pairs for an irreducible, unstabilized $(4,2)$-bridge trisection with trisection surface $\Sigma$. In order to estimate $\Lcal(\mc{T})$, we are interested in understanding the distance $d(p^i_{ij}, p^j_{ij})$. 
%
We define an \textbf{inside} of a simple closed curve in $\Sigma$ to be a side with $\leq 4$ punctures and an \textbf{outside} to be a side with $>4$ punctures. Note that curves with four punctures on each side have two inside regions and no outside region. We say that a puncture or set of punctures is \defn{enclosed} by such a curve if the curve does not separate them and they are all inside the curve. Analyzing which curves in a pants decomposition can enclose which others, produces the next lemma:

\begin{lemma}\label{lemma_pants_structure}
Let $(p_{ij}^{i},p_{ik}^{i})$ be an efficient defining pair for $L_i$. Then, we may choose notation $p_{ij}^{i}=\{\gamma_1,\gamma_2,\gamma_3,f_1, f_2\}$ and  $p_{ik}^{i}=\{\gamma_1,\gamma_2,\gamma_3,f'_1, f'_2\}$ so that all of the following hold:
\begin{itemize}
    \item $\gamma_1$ is a reducing curve for $L_i$
    \item Both $\gamma_2$ and $\gamma_3$ are cut-reducing curves for $L_i$.
    \item $f_1, f_2$ bound compressing discs for $T_{ij}$ and $f'_1, f'_2$ bound compressing discs for $T_{ik}$
    \item Every geodesic from $p_{ij}^i$ to $p_{ik}^i$ moves $f_1$ to $f'_1$ and $f_2$ to $f'_2$ and $\gamma_1$, $\gamma_2$, and $\gamma_3$ are unmoved.
\end{itemize}
\end{lemma}

%{\color{red} The proof of this lemma doesn't use the particular M-Z trisection for $S(K)$.}

\begin{proof}
Recall that $\Sigma$ has 8 punctures, so each pants decomposition has 5 curves. Let $(p_{ij}^{i},p_{ik}^{i})$ be an efficient defining pair. By Lemma \ref{lem_efficient_pairs}, the distance from $p_{ij}^i$ to $p_{ik}^i$ is equal to $\b(\mc{T}) - c_i = 2$. Thus, at least 3 curves are unmoved by any geodesic in the pants complex joining $p_{ij}^i$ to $p_{ik}^i$. Let $\gamma_1, \gamma_2, \gamma_3$ be three such curves, and let $f_1, f_2$ be the other two. Curves in $\Sigma$ bounding cut discs in one of the tangles in the spine, enclose an odd number of punctures in $\Sigma$, while those bounding compressing discs enclose an even number of punctures. Thus, each of $\gamma_1, \gamma_2, \gamma_3$ is either a reducing curve or a cut-reducing curve for $L_i$. 

It is impossible for $\gamma_1$, $\gamma_2$ and $\gamma_3$ to all bound cut disks to both sides, because there are only 8 punctures and the three curves are pairwise nonparallel. Thus, at least one is a reducing curve. Without loss of generality, we may assume it is $\gamma_1$. Since $c_i = 2$, all reducing curves for $L_i$ enclose the same punctures. Thus, $\gamma_2$ and $\gamma_3$ must be cut-reducing curves. Each encloses exactly 3 punctures. Since $p_{ij}^i$ is a pants decomposition, all other curves of $p_{ij}^i$ enclose an even number of punctures. Consequently, both $f_1$ and $f_2$ must be moved by every geodesic between $p_{ij}^i$ and $p_{ik}^i$. Thus, each geodesic moves the pair $f_1, f_2$ to the pair $f'_1, f'_2$, which are the curves of $p_{ik}^i$ that are not $\gamma_1$, $\gamma_2$, or $\gamma_3$. 

Furthermore, one of $\gamma_2$ or $\gamma_3$ encloses three punctures as well as either $f_1$ or $f_2$. Since no geodesic between $p_{ij}^i$ and $p_{ik}^i$ moves $\gamma_2$ or $\gamma_3$, there are not two geodesics one of which moves $f_1$ to $f'_1$ and other of which moves it to $f'_2$. Thus, we may assume the notation was chosen so that every such geodesic moves $f_1$ to $f'_1$ and $f_2$ to $f'_2$.
\end{proof}

\begin{remark}
We will often consider efficient defining pairs $(p_{ij}^{i},p_{ik}^{i})$ and $(p_{ij}^{j},p_{jk}^{j})$. In which case, we choose notation $p_{ij}^i = \{\gamma_1, \gamma_2, \gamma_3, f_1, f_2\}$ and $p_{ij}^j
 = \{\psi_1, \psi_2, \psi_3, h_1, h_2\}$ as in Lemma \ref{lemma_pants_structure}. We refer to any of $\gamma_1, \gamma_2, \gamma_3$ as a \defn{$\gamma_n$-loop} and any of $\psi_1, \psi_2, \psi_3$ as a \textbf{$\psi_n$-loop}.
\end{remark}

A \textbf{configuration} of either $T_{ij}$, $T_{ik}$ or $L_i$ is the partition $\Delta_{ij}$, $\Delta_{jk}$ or $\Delta_i$ (respectively) of the set of the labeled punctures $L= \{1,2,3,4,5,6,7,8\}$ on $\Sigma$ built as follows: two punctures are related if they belong to the same connected component of $T_{ij}$, $T_{jk}$, or $L_i$ respectively.  We will often abbreviate the string $`3, 4, 5, 6, 7, 8'$ as $3-8$, and so forth. An element of a configuration with exactly $n$ elements is called an \defn{$n$-cycle}.

We are interested in the triplet of configurations $(\Delta_1, \Delta_2, \Delta_3)$ for $L_1$, $L_2$, and $L_3$. Up to relabeling, $(4,2)$-bridge trisection has essentially three options for such triplets. This is formalized in Lemma \ref{lem_combinatorics}. 

\begin{lemma}\label{lem_combinatorics}
Let $S$ be a connected surface in $S^4$ with a $(4,2)$-bridge trisection $T$. Up to permutation of $L$ and choice $\{i,j,k\}=\{1,2,3\}$, there are three possible configurations for $L_i$, $L_j$, and $L_k$: 
\begin{enumerate}
\item  $\Delta_{i}=\{\{1,2\},\{3 - 8\}\}$, 
$\Delta_{j}=\{\{1-5,8\},\{6,7\}\}$,
$\Delta_{k}=\{\{3,4\},\{1,2,5-8\}\}$.

\item $\Delta_{i}=\{\{1,2\},\{3-8\}\}$, 
$\Delta_{j}=\{\{1,2,6,7\},\{3,4,5,8\}\}$,
$\Delta_{k}=\{\{3,4\},\{1,2,5-8\}\}$.

\item $\Delta_{i}=\{\{1 - 4\},\{5 - 8\}\}$, 
$\Delta_{j}=\{\{1,4,5,8\},\{2,3,6,7\}\}$,
$\Delta_{k}=\{\{1,2,7,8\},\{3 - 6\}\}$. 
\end{enumerate}
\end{lemma}

\begin{proof}[Proof of Lemma \ref{lem_combinatorics}]
The fact that $T$ is a $(4,2)$-bridge trisection implies that $\Delta_1$, $\Delta_2$, and $\Delta_3$ each have either one 2-cycle and one 6-cycle or exactly two 4-cycles. 

\textbf{Case 1}: Suppose first that $\Delta_{j}$ has one 2-cycle. After relabeling, we can assume that $\Delta_{ij}=\{\{1,2\},\{3,4\},\{5,6\},\{7,8\}\}$ and $\Delta_{jk}=\{\{1,2\},\{3,8\},\{4,5\},\{6,7\}\}$. By connectivity of $F$ we have that $\{1,2\}\not \in \Delta_{ik}$. We have two cases: either $\Delta_{ik}$ shares a common 2-cycle with $\Delta_{ij}$ (or $\Delta_{jk}$) or not. 

\textbf{Subcase 1a}: $\Delta_{ij}$ and $\Delta_{ik}$ have a common 2-cycle, say $\{3,4\}\in \Delta_{ij} \cap \Delta_{ik}$. 

Suppose $\{6,7\}\in \Delta_{ik}$. Since $|\Delta_{k}| = 2$, the labels 5 and 8 must lie in the same component of $\Delta_{ik}$ as 1 and 2. This yields option 1 of the statement. Suppose now that $\{6,7\}\not \in \Delta_{ik}$, in particular $\Delta_{ik}$ and $\Delta_{jk}$ have no common 2-cycle. Focusing in $\Delta_{k}$, observe that if $\{5,8\}\not \in \Delta_{ik}$, then $\Delta_{ik}$ must contain one of $\{1,2\}$ or $\{6,7\}$, which is a contradiction to the previous sentence. Thus we have $\{5,8\}\in \Delta_{ik}$, concluding that $\Delta_{ik}$ must relate the labels 1 and 2 to 6 and 7 somehow. This yields the configuration in option 2 of the statement.  

\textbf{Subcase 1b}: $\Delta_{ik}$ has no common 2-cycle with either $\Delta_{ij}$ and $\Delta_{jk}$. 

We will see that this case cannot occur. Here, $\Delta_{ik}$ is forced to relate 1 and 2 to labels in $\{3-8\}$. After relabeling, we can assume that $\{2,3\}\in \Delta_{ik}$. We have five remaining options for $x$ such that $\{1,x\}\in \Delta_{ik}$. If $x = 4$, in order to  have $|\Delta_k| = 2$, it must be that contains $\{7,8\}\in \Delta_{jk}$. Thus $\Delta_{jk}$ and $\Delta_{ik}$ have a common 2-cycle, a contradiction. Similarly, we rule out $x = 5,6,7$. If $x = 8$, then as $\Delta_{ik}$ does not share a 2-cycle with $\Delta_{jk}$, it must be the case that $\Delta_{ik}$ contains either $\{4,6\}$ or  $\{4,7\}$. The first possibility implies $\Delta_i$ is a single 8-cycle, while the second implies $\\Delta_{ik}$ and $\Delta_{ij}$ share a 2-cycle. Both are impossibilities in this subcase.

\textbf{Case 2:} Suppose now that $\Delta_{j}$ contains two 4-cycles. 

Without loss of generality, we can assume that $\Delta_{ij}=\{\{1,2\},\{3,4\},\{5,6\},\{7,8\}\}$ and $\Delta_{jk}=\{\{1,4\},\{2,3\},\{5,8\},\{6,7\}\}$. Observe that if $\Delta_{i}$ or $\Delta_{j}$ have one 2-cycle, then we can permute the symbols $\{i,j,k\}$ and continue as in Case 1; yielding the configurations 1 and 2 in the statement. In particular, if $\{x,y\}\in \Delta_{ik}$, then we must have $\{a,b\},\{c,d\}\in \Delta_{ik}$ where $\{x,a\}, \{y,b\}\in \Delta_{ij}$ and $\{x,c\}, \{y,d\}\in \Delta_{jk}$. 

\textbf{Subcase 2a:} $\Delta_{ik}$ relates 1 and 2 to 3 and 4. 

By the previous paragraph, we are forced to have $\Delta_{ik} = \{ \{1,3\}, \{2,4\}, \{5,7\}, \{6,8\}\}$. Thus \[\Delta_{j} = \Delta_{k}= \Delta_{i}=\{\{1-4\}, \{5 - 8\}\}\] which contradicts the fact that $F$ is connected. 

\textbf{Subcase 2b:} $\Delta_{ij}$ does not relate 1 and 2 to 3 and 4. 

After relabeling, we can assume that $\{4,5\}\in \Delta_{ik}$. The fact that $|\Delta_{k}| = |\Delta_{i}| = 2$ forces $\Delta_{ik}=\{ \{4,5\}, \{3,6\}, \{2,7\}, \{1,8\}\}$. This yields configuration 3 in the statement. 
\end{proof}

It is easy to see that (MZ)-bridge trisections for (twist) spun 2-bridge knots have configurations as in Case 2 of Lemma \ref{lem_combinatorics}. 

\begin{question}
Are there nonstabilized $(4,2)$-bridge trisections of the other types? 
\end{question}

\begin{remark}\label{remark_combinatorial_property}
The following combinatorial properties of reducing curves are direct consequences of Lemma \ref{lem_combinatorics}: Let $\psi_1$ and $\gamma_1$ be reducing curves in $\Delta_{j}$ and $\Delta_{i}$, respectively.
\begin{itemize}
    \item If $\{x,y\}$ are punctures enclosed by $\gamma_1$ and if one of them is also enclosed by $\psi_1$, then both are enclosed by $\psi_1$.
    \item Suppose $\psi_1$ and $\gamma_1$ both bound four punctures, and that $\gamma_1$ bounds $\{x,y,z,w\}$. Then, after relabeling, $\psi_1$ separates $\{x,y\}$ from $\{z,w\}$.
\end{itemize}
\end{remark}

%%%%%%%%%%%%%%%%%%%%%%%%%%%%%%%%%%%%%%%%%%%%%%%%%%%%%%%%%%%5
\subsection{Reducing curves}
Reducing curves play a special role in trisections. In the case of (4,2)-bridge trisections, they restrict the pants decompositions near $p^i_{ij}$ in $\Pcal(\Sigma)$. Lemmas \ref{lemma_red_intersect} and \ref{lem_experimental_2} show that in certain circumstances reducing curves for different links must intersect at least four times. Lemma \ref{lemma_psi_f_distinct} compares the $\gamma_n$-curves in $p^i_{ij}$ with the ones (called $\psi_n$-curves, for convenience) in $p^j_{ij}$. Lemmas \ref{lem_gamma1_first} and \ref{lem_gamma1_last} imply that A-moves of the form $\gamma_1\mapsto \psi_n$ and $\gamma_n \mapsto \psi_1$ cannot occur near $p^i_{ij}$. We rely heavily on theorems of Lee \cite{lee2017reduction}, governing the relationship between perturbations of a bridge position with bridge disks.

\begin{lemma}\label{lemma_red_intersect}% Scott's Lemma
Suppose $L_i$ has one component intersecting $\Sigma$ exactly twice and $L_j$ has no such component. Let $\gamma$ in $\Sigma$ be a reducing curve for $L_i$ and suppose $\psi \subset \Sigma$ is either a reducing curve or cut-reducing curve for $L_j$. Then the following hold:
\begin{enumerate}
    \item If $\psi$ is a reducing curve, then $|\gamma \cap \psi| \geq 4$.
    \item If $\psi$ is a cut-reducing curve, and $\psi$ and $\gamma$ are disjoint, then $\gamma$ lies inside a 3-punctured disk bounded by $\psi$. %is on the side of $\psi$ that contains 5 punctures.
\end{enumerate}
\end{lemma}

\begin{proof}
Let $\gamma$ and $\psi$ be as in the statement and assume that they have been isotoped so as to intersect minimally.  Let $Q$ be a sphere separating the components of $L_j$ such that $Q \cap \Sigma = \psi$. Let $L_i(1)$ and $L_i(3)$ be the 1-bridge and 3-bridge components of $L_i$ and let $L'_j$ and $L''_j$ be the two components of $L_j$. 

Since $\gamma$ is a reducing curve for $L_i$, it is isotopic to the boundary of a regular neighborhood of an arc $\alpha \subset \Sigma$ joining the punctures $L_i(1) \cap \Sigma$. The arc $\alpha$ is the intersection $D \cap \Sigma$ of a disc $D$ such that $\partial D = L_i(1)$ and the interior of $D$ is disjoint from $L_i$. Observe that there is a shadow arc $\alpha'$ for $(\bar{B_{ik}}, \bar{T}_{ik})$ that is a copy of $\alpha.$

Suppose that $\gamma \cap \psi = \emptyset$. We may, therefore, assume that $D$ is disjoint from  $Q \cap B_{ij}$. 

Observe that $E_1 = D \cap B_{ij}$ is a bridge disc for an arc of $T_{ij}$. Let $K_j \subset B_{ij} \cup \bar{B_{jk}}$ be the link that results from isotoping this arc along $E_1$ and across $\Sigma$. The link $K_j$ is isotopic to $L_j$, and is, therefore, an unlink of two components. One component is equal to a component of $L_j$. The result of $\boundary$-reducing $(B_{jk}, T_{jk})$ along the c-disk $E = Q \cap B_{jk}$ is the disjoint union of two trivial tangles, call them $(U_1, \tau_1)$ and $(U_2, \tau_2)$. The result of $\boundary$-reducing $(B_{jk}, K_j \cap B_{jk})$ along $E$ is two tangles, one of which is either $(U_1, \tau_1)$ or $(U_2, \tau_2)$. Without loss of generality, we may assume it is $(U_2, \tau_2)$. Call the other one $(U'_1, \tau'_1)$. If $(U'_1, \tau'_1)$ is a trivial tangle, then so is  $(B_{jk}, K_j \cap B_{jk})$. %(** INSERT REFERENCE??**\textcolor{blue}{not sure which paper to cite}) 
If $\psi$ is a reducing-curve, then $\tau'_1$ is a single strand; it must be unknotted, as $K_j$ is an unlink. Otherwise, $\psi$ separates the punctures of $\Sigma$ into one set with 3 punctures and the other with 5 punctures. If $\gamma$ is on the side with 5 punctures, we have our theorem, so assume $\gamma$ is on the side with 3 punctures. Thus, one of $(U_1, \tau_1)$ has 2 strands, and $(U_2, \tau_2)$ has 3 strands. Thus, $(U'_1, \tau'_1)$ has a single strand and, as before, we see that it is a trivial tangle. Thus, $(B_{jk}, K_j \cap B_{jk})$ is a trivial tangle and $\Sigma$ is a bridge sphere for $K_j$.

By \cite[Theorem 1.1]{lee2017reduction}, there is a bridge disc $E_2$ for a strand of $\bar{T}_{jk}$ in $\bar{B_{jk}}$ such that the arcs $\alpha$ and $\beta = E_2 \cap \Sigma$ intersect in a single point.  The three shadow arcs $\alpha$, $\alpha'$, and $\beta$ show that $\Sigma$ is stabilized as in Lemma \ref{criterion_destab}. This contradicts our assumption on $\Sigma$. Thus, $|\gamma \cap \psi| > 0$ when $\psi$ is a reducing curve and $\gamma$ is on the side with 5 punctures if $\psi$ is a cut-reducing curve and $|\gamma \cap \psi| =\emptyset$. 

Consider the twice punctured disc $D \subset \Sigma$ bounded by $\gamma$. If $|\psi \cap \gamma| > 0$, then $\psi \cap D$ consists of parallel arcs separating the punctures. If $\psi$ is a reducing curve, then it bounds discs in $\Sigma$ each containing an even number of punctures. In which case, $|\psi \cap D|$ is even and $|\psi \cap \gamma|$ is a multiple of 4. Consequently, if $\psi$ is a reducing curve, $|\gamma \cap \psi| \geq 4$.
\end{proof}
\begin{lemma}\label{lem_experimental_2}% Scott's Lemma
Suppose $L_i$ has one component intersecting $\Sigma$ exactly twice. That is, $L_i$ is a 2-component link, where one component is in 1-bridge position and the other component is in 3-bridge position. Let $\gamma \subset \Sigma$ be a reducing curve for $L_i$ and suppose $\psi \subset \Sigma$ is a cut-reducing curve for $L_j$. 
\begin{enumerate}
    \item Suppose that both components of $L_j$ are in 2-bridge position. Then $|\gamma \cap \psi|\neq 2$. 
    \item Suppose $L_j$ has one component in 3-bridge position. If $|\gamma\cap \psi|=2$, then the two punctures corresponding to the 1-bridge component of $L_j$ lie inside a 3-punctured disk bounded by $\psi$. 
\end{enumerate}
\end{lemma}
\begin{figure}[h]
\centering
\labellist \small\hair 2pt 
\pinlabel {(a)} at 0 120
\pinlabel {$\alpha$} at 95 120
\pinlabel {$Q$} at 180 120

\pinlabel {(b)} at 250 120
\pinlabel {$\alpha$} at 340 110
\pinlabel {$Q$} at 410 120

\pinlabel {(c)} at 480 120
\pinlabel {$\alpha$} at 585 105
\pinlabel {$Q$} at 540 125

\endlabellist
\includegraphics[width=1\textwidth]{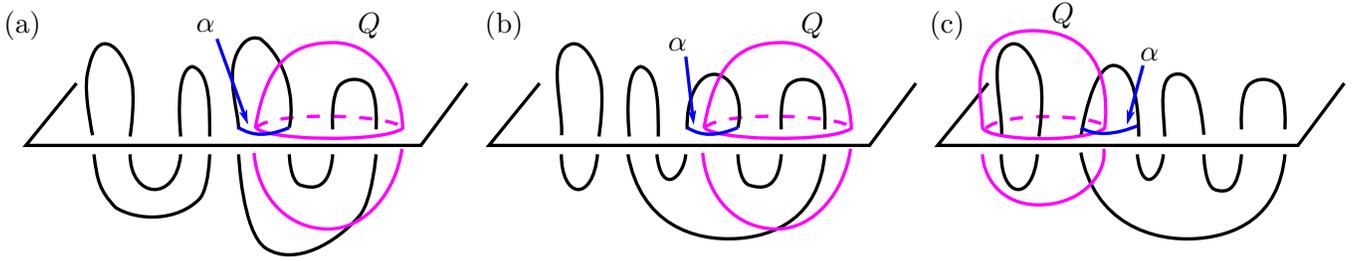}
\caption{The link $L_j=T_{ij}\cup \T_{jk}$ in bridge position. The arc $\alpha$ is a shadow for arcs in $T_{ij}$ and $T_{ik}$.} 
\label{fig_experiment}
\end{figure}
\begin{proof}

 Suppose for the sake of contradiction that $|\gamma \cap \psi| = 2$. Let $Q$ be a cut-reducing sphere such that $Q \cap \Sigma = \psi$. Cut open $(S^3, L_j)$ along $Q$ and glue in (3-ball, unknotted arc) pairs $(B^3, \alpha_1)$ and $(B^3, \alpha_2)$ to obtain $(S^3, \lambda_1)$ and $(S^3, \lambda_2)$. In the 3-balls that we glued in we may find once punctured disks whose boundaries coincide with the images of $\psi$. 
Attach those discs to the remnants of $\Sigma$ to obtain bridge spheres $\Sigma_1$ and $\Sigma_2$ for $(S^3, \lambda_1)$ and $(S^3, \lambda_2)$, respectively. We can recover $(S^3, L_j, \Sigma)$ by taking the connected sum of the triples $(S^3, \lambda_1, \Sigma_1)$ and $(S^3, \lambda_2, \Sigma_2)$. In particular, $\lambda_1$ and $\lambda_2$ are unlinks. Since we are decomposing a 2-component unlink $L_j$ via a cut-reducing sphere, we can assume that $\lambda_1$ has one component and $\lambda_2$ has two components. There are a few cases to consider (see Figure \ref{fig_experiment}). 
In all of these cases, the strategy is the following. Using the same notations as in Lemma \ref{lemma_red_intersect}, there is a shadow arc $\alpha'$ for $(\bar{B_{ik}}, \bar{T}_{ik})$ that is a copy of $\alpha$ for $(B_{ij}, T_{ij})$. We then use a result of Lee's \cite{lee2017reduction} to find a shadow in $(B_{jk},T_{jk})$ intersecting $\alpha$ only in one endpoint (and no interior points). By Lemma \ref{criterion_destab}, this implies that $\mc{T}$ is stabilized, contrary to hypothesis.

Let $D$ as in Lemma \ref{lemma_red_intersect}. The intersection $D \cap \Sigma$ is a shadow $\alpha$ for arcs in both $T_{ij}$ and $T_{ik}$. Since $|\gamma \cap \psi|=2$, the disk $Q_0 = Q \cap B_{ij}$ intersects the disc $E= D \cap B_{ij}$ in a single arc. 
Thus, $E$ persists to bridge discs $E_1$ for $\lambda_1$ and $E_2$ for $\lambda_2$.

\textbf{Case 1:} Each component of $L_j$ is in 2-bridge position, i.e. intersects $\Sigma$ four times.

Only one component of $L_j$ intersects $Q$. Without loss of generality, we may assume it is $L_j'$. Furthermore, all of the punctures $L''_j \cap \Sigma$ must lie in $\Sigma_2$ as $|L'_j \cap \Sigma| = 4$. Thus, $\lambda_1$ is an unknot intersecting $\Sigma_1$ exactly 4 times. 
Recall $E_1$ is a bridge disk for $\lambda_1$. Let $E'_1$ be another bridge disc for $\lambda_1$, on the same side of $\Sigma_1$ as $E_1$, but disjoint from $E_1$. Observe that in the four punctured sphere $\Sigma_1$, the frontiers of the arcs $E_1 \cap \Sigma_1$ and $E'_1 \cap \Sigma_1$ are isotopic. Since a reduction along a bridge disk of the 2-bridge unknot is an unknot in 1-bridge position, a result of Lee \cite[Theorem 1.2]{lee2017reduction} tells us that each arc of $\lambda_1 \setminus \Sigma_1$ on the opposite side of $\Sigma_1$ from $E_1$ and $E'_1$ has a bridge disc intersecting both $E_1$ and $E_2$ only in one endpoint (and no interior points). Let $\epsilon$ be such a disc for the strand of $\lambda_1 \setminus \Sigma_1$ that does not contain $\alpha_1$. Then $\epsilon$ is also a bridge disc for $L_j$ and it intersects $\alpha$ only in one endpoint (and no interior points).

\textbf{Case 2:} A component of $L_j$ is in 1-bridge position, i.e. intersects $\Sigma$ only twice.

If $\lambda_1$ is an unknot intersecting $\Sigma_1$ exactly 4 times, then we have the situation with the schematic shown in Figure \ref{fig_experiment}(b). In this case, the shadow we seek for $(B_{jk},T_{jk})$ is found as in Case 1. That is, there is a shadow arc $\alpha'$ for $(\bar{B_{ik}}, \bar{T}_{ik})$ that is a copy of a shadow arc $\alpha$ for $(B_{ij}, T_{ij})$. Since $\lambda_1$ is a 2-bridge unknot, Lee's result \cite{lee2017reduction} tells us that there is a shadow in $(\bar{B_{jk}},\bar{T_{jk}})$ intersecting $\alpha$ only in one endpoint (and no interior points). On the other hand, if $\lambda_1$ is an unknot intersecting $\Sigma_1$ exactly 6 times, we have the second conclusion of our lemma (see Figure \ref{fig_experiment}(c)). 
%
%
% \textcolor{blue}{\textbf{Subcase 2a:} If $\lambda_1$ is an unknot intersecting $\Sigma_1$ exactly 4 times, then the shadow we seek for $(B_{jk},T_{jk})$ is found through the same as in Case 1. }
%
% {\color{blue}\textbf{Subcase 2b:}
% Suppose that $\lambda_1$ is an unknot intersecting $\Sigma_1$ exactly 6 times. By Lemma 2.4 of \cite{kwonweak} {\color{red} this is not what it says, haven't proved that $E_1$ is reducing}, there is a complete cancelling disk system of such an unknot in 3-bridge position. That is, there are six cancelling bridge disks $\Delta_1,...,\Delta_6$ such that $\Delta_i\cap \Delta_{i+1}$ in one point of the knot and $\Delta_i\cap\Delta_j = \emptyset$ for nonadjacent indices $i,j.$ This gives the desired shadow in $(B_{jk},T_{jk})$ intersecting $\alpha$ only in one endpoint (and no interior points).
% }
\end{proof}
Our proofs of Lemmas \ref{lemma_red_intersect} and \ref{lem_experimental_2} above do not work for higher bridge numbers, as there is a 4-bridge position of the unknot with no complete cancelling disk system (see \cite{lee2017reduction}).

For the remainder of this section, let $p_{ij}^{i}$ and $p_{ij}^{j}$ be pants decompositions belonging to defining pairs for $L_i = T_{ij}\cup\T_{ik}$ and $L_j = T_{kj}\cup \T_{ij}$, respectively. Denote their curves by $p_{ij}^{i}=\{\gamma_1,\gamma_2,\gamma_3,f_1, f_2\}$ and $p_{ij}^{j}=\{\psi_1,\psi_2,\psi_3,h_1, h_2\}$ as in Lemma \ref{lemma_pants_structure}. 

\begin{lemma}\label{lemma_psi_f_distinct}
No $\psi_n$-loop is equal to $f_m$, for any $m \in \{1,2\}$. Similarly, no $\gamma_n$-loop is equal to $h_m$ for any $m \in \{1,2\}$.
\end{lemma}

\begin{proof}
The second statement follows from the first by reversing the roles in the proof below. We prove the first statement.

By Lemma \ref{lemma_pants_structure}, $\psi_2$ and $\psi_3$ bound cut-disks and $f_1$ and $f_2$ bound compressing disks, so the number of punctures they enclose is different modulo 2. Thus $\psi_n \neq f_1,f_2$ for $n=2,3$. 

Suppose now that $\psi_1=f_1$. In particular, $\gamma_1$ and $\psi_1$ are disjoint reducing curves. By Lemma \ref{lemma_red_intersect}, the number of punctures enclosed by $\gamma_1$ and $\psi_1$ must be the same. For if $\gamma_1$ bounds two punctures and $\psi_1$ bounds four punctures, then the two curves will intersect. But $\gamma_1$ and $f_1$ are distinct curves in the pants decomposition $p_{ij}^{i}$, so they cannot both enclose four punctures. We conclude that $\psi_1=f_1$ and $\gamma_1$ enclose two punctures each. 
Let $f'_1$ and $f'_2$ be simple closed curves such that $p_{ik}^i=\{\gamma_1, \gamma_2, \gamma_3, f'_1, f'_2\}$ completes a defining pair $(p_{ij}^i, p_{ik}^i)$ for $T_{ij} \cup \T_{ik}$. Focus our attention of the A-move corresponding to $f_1 \mapsto f'_1$, which happens inside a 4-holed sphere $E$.  The boundaries of $E$ correspond to boundaries of small neighborhoods of punctures or to some $\gamma_n$-curves. Notice that one or two boundaries of $E$ correspond to some $\gamma_n$-curves. 

\textbf{Case 1:} $\boundary E$ has exactly one $\gamma_n$ loop. 

After a surface homeomorphism, we can draw $E$ as in the Figure \ref{fig_various_subsurfaces}(a). Here, after choosing coordinates for the 4-punctured sphere, $f_1$ is depicted as a separating curve of slope $1/0$. The conditions $|f_1\cap f'_1|=2$ and $f'_1 \cap \gamma_n = \emptyset$ imply that $f'_1$ is a separating simple closed curve in $E$ of slope $n/1$ for some $n\in \mathbb{Z}$. In other words, $f_1 =\partial \eta(c)$ and $f'_1=\partial\eta(c')$ for some properly embedded arcs $c, c'$ in $E$ such that $c$ is an arc disjoint from $\gamma_n$, and $c\cap c'=\partial c \cap \partial c'$ is exactly one puncture. We pick $c'$ so that the end disjoint from $c$ corresponds to the puncture $p$ on the same side of $f_1$ as $\gamma_n$ (see Figure \ref{fig_various_subsurfaces}(a)). 
Now, recall that $f'_1$ bounds a compressing disk for $T_{ik}$, and so $c'$ is a shadow for some arc in $T_{ik}$. Similarly, $c$ is a shadow for arcs in both $T_{ij}$ and $T_{kj}$ because $f_1=\psi_1$ is a compressing disk for both tangles. By Remark \ref{criterion_destab}, these three shadow arcs with one common endpoint imply that the bridge trisection is stabilized. This concludes Case 1.
\begin{figure}[h]
\centering
\labellist \small\hair 2pt 
\pinlabel {(a)} at 7 177
\pinlabel {$E$} [r] at 11 140
\pinlabel {$\gamma_n$} at 119 126
\pinlabel {$p$} [t] at 106 47
\pinlabel {$f'_1$} [t] at 111 21
\pinlabel {$f_1$} [l] at 26 73

\pinlabel {(b)} at 193 177
\pinlabel {$f'_1$} [b] at 250 120
\pinlabel {$\gamma_2$} [b] at 280 105
\pinlabel {$\gamma_1$} [l] at 254 58
\pinlabel {$q$} [t] at 291 74
\pinlabel {$p$} [t] at 322 74
\pinlabel {$f_1$} [t] at 356 60
\pinlabel {$\gamma_3$} [t] at 425 47

\pinlabel {(c)} at 525 177
\pinlabel {$D$} at 683 23
\pinlabel {$\gamma_1$} [t] at 615 110
\pinlabel {$q$} [r] at 605 61
\pinlabel {new $c'$} [t] at 634 61
\pinlabel {old $c'$} [l] at 538 99
\pinlabel {$x$} [l] at 701 114
\pinlabel {$f'_1$} at 699 153
\pinlabel {$p$} [t] at 701 61

\endlabellist
\includegraphics[width=.9\textwidth]{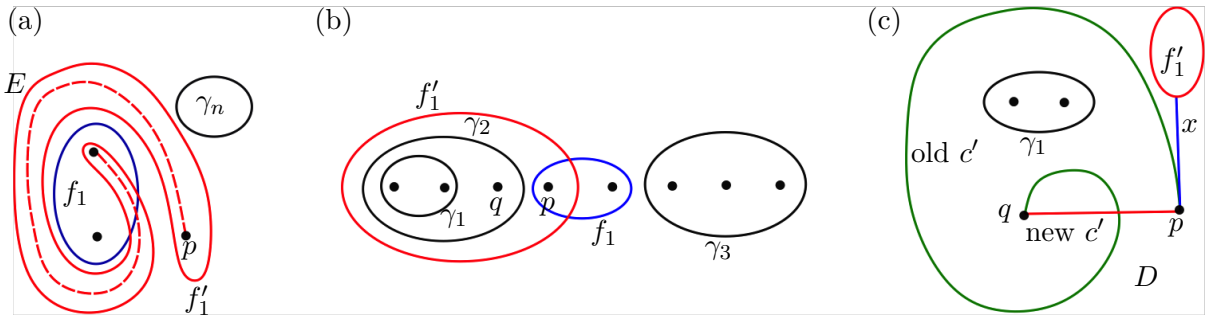}
\caption{Various subsurfaces of $\Sigma$.} 
\label{fig_various_subsurfaces}
\end{figure}

\textbf{Case 2:} $\boundary E$ has two $\gamma_n$-loops. 

Both must bound cut-disks. After a surface homeomorphism, the curves in $p^i_{ij}$ can be depicted as in Figure \ref{fig_various_subsurfaces}(b). Observe here that $f'_1$ must surround four punctures on each side. Let $D$ be the 4-holed sphere inside $\Sigma$ co-bounded by $f'_1$, $\gamma_1$, $\partial\eta(p)$ and $\partial\eta(q)$ (see Figure \ref{fig_various_subsurfaces}(b)-(c)). By construction, there exists an arc $x$ in $D$ with endpoints in $p$ and $E$ such that $x$ is disjoint from $f'_1\cap D$. 
Since $\gamma_1$ and $f'_1$ both bound compressing disks for $T_{ik}$, it follows that there is an arc in $T_{ik}$ connecting $p$ and $q$. Furthermore, such arc has a shadow arc $c'$ in $\Sigma$ with interior disjoint from $f'_1$ and $\gamma_1$. Regarded as a subset of $D$, the arc $c'$ connects $E$ and $\gamma_1$. % which might have common interior with the arc $x$. 
We can slide $c'$ over $\gamma_1$ several times and choose a shadow arc $c$ with interior disjoint from $x$. In particular, $c$ intersects $f_1$ in one point. This, together with the fact that $f_1=\psi_1$ bounds reducing curve for $T_{kj}\cup \T_{ij}$, implies the existence of a shadow arc $c$ for both $T_{kj}$ and $T_{ij}$ with $c\cap c' = \partial c \cap \partial c' = \{p\}$. By Lemma \ref{criterion_destab} we conclude that $\mc{T}$ is stabilized. 
\end{proof}

\begin{lemma}\label{lem_gamma1_first}\label{remark_reverse_lem}
Suppose $e$ is an edge in $\mc{P}(\Sigma)$ with initial endpoint at  $p_{ij}^{i}$ then $e$ does not move $\gamma_1$ to any $\psi_n$-loop in $p^j_{ij}$. Similarly, if $e$ is an edge in $\mc{P}(\Sigma)$ with terminal endpoint at  $p_{ij}^{j}$, then $e$ does not move any $\gamma_n$-loop of $p^i_{ij}$ to $\psi_1$.
\end{lemma}

\begin{proof}[Proof of Lemma \ref{lem_gamma1_first}]
The second statement follows from the first by interchanging the roles of $\gamma_1$ and $\psi_1$, and so we prove only the first statement. Suppose, to establish a contradiction, that $\gamma_1$ is moved to some $\psi_n$-loop by $e$.

First we show that $e$ does not move $\gamma_1$ to $\psi_1$.  Suppose $\gamma_1$ bounds a twice-punctured disc $D$. If $e$ moves $\gamma_1$ to $\psi_1$ then $|\gamma_1 \cap \psi_1| = 2$, so $D \cap \psi_1$ consists of a single arc. It follows that the two punctures of $D$ are on opposite sides of $\psi_1$, contradicting Remark \ref{remark_combinatorial_property}. Similarly, $\psi_1$ does not bound a twice-punctured disc. 

Consequently, if $e$ moves $\gamma_1$ to $\psi_1$, then both $\gamma_1$ and $\psi_1$ enclose four punctures. This sets us in the third configuration of Lemma \ref{lem_combinatorics}. First, observe that $f_1$ and $f_2$ must be separated by $\gamma_1$. This holds since $p^i_{ij}=\{\gamma_1, \gamma_2, \gamma_3, f_1, f_2\}$ is a pants decomposition for $\Sigma$, and only $\gamma_1$, $f_1$ and $f_2$ bound an even number of punctures. Thus, after a surface homeomorphism, we can draw $\Sigma$ and $p^i_{ij}$ as in Figure \ref{fig_gamma1_first}. We see, therefore, that if $e$ moves $\gamma_1$ to $\psi_1$, then $\gamma_1$ and $\psi_1$ will both bound the same three (out of four) punctures, contradicting Lemma \ref{lem_combinatorics}. Hence, $\gamma_1$ cannot be moved first to $\psi_1$. 

We will now see that, due to parity constraints, if $e$ moves $\gamma_1$, then $\gamma_1$ is moved to a curve bounding an even number of punctures. In particular, $\gamma_1$ is not moved to $\psi_n$ for $n=2,3$. In order to do this, we focus on the 4-holed sphere, denoted by $E$, corresponding to the first A-move. The four boundary components of $E$ are loops (or punctures), $\{\partial_1, \partial_2, \partial_3, \partial_4\}$. If $\gamma_1$ bounds four punctures, up to surface homeomorphism, then $\Sigma$ can be depicted as in Figure \ref{fig_gamma1_first} and we see that each component of $\boundary E$ is an odd curve. On the other hand, if $\gamma_1$ encloses exactly two punctures, then two components of $\boundary E$ are single punctures. The other two boundaries, say $\partial_3$ and $\partial_4$, will enclose punctures 1 and 5, 2 and 4, or 3 and 3, respectively. Notice that they cannot enclose punctures 2 and 4, since that will force the existence of a fourth curve in $p^i_{ij}$ bounding even number of punctures. Thus, in any case, all the components of $\boundary E$ are either a single puncture or enclose an odd number of punctures. Consequently, $e$ moves $\gamma_1$ to a curve enclosing an even number of punctures, as desired.
\end{proof}
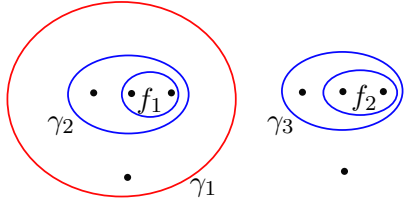
\begin{figure}[h]
\labellist \small\hair 2pt 
\pinlabel {$\gamma_1$} at 158 8
\pinlabel {$\gamma_2$} at 44 58
\pinlabel {$\gamma_3$} at 220 58
\pinlabel {$f_1$} at 115 78
\pinlabel {$f_2$} at 287 80
\endlabellist
\centering
\includegraphics[width=.3\textwidth]{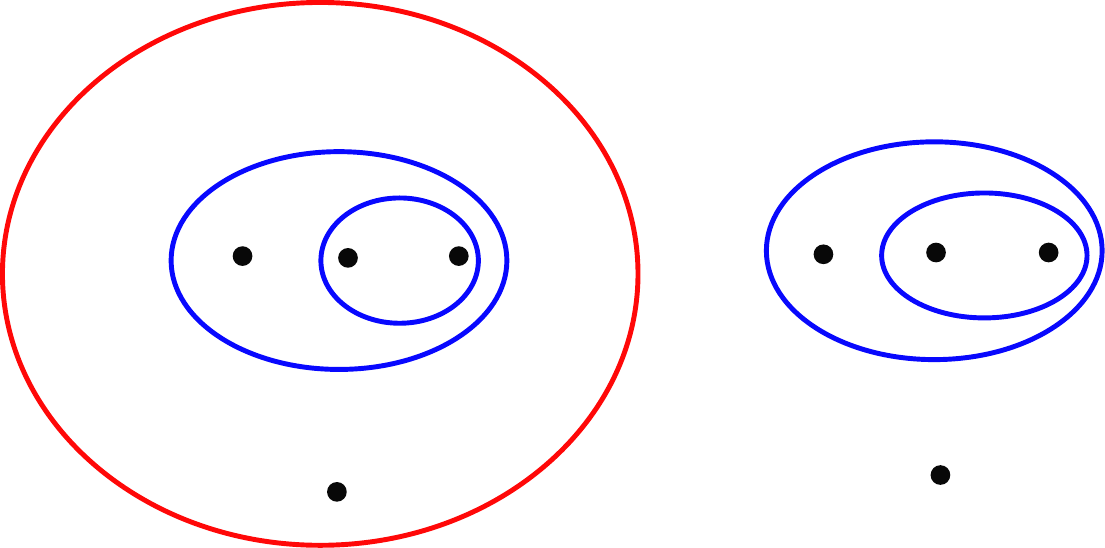}
\caption{When the reducing curve bounds four punctures, the two cut curves lie on distinct sides.}
\label{fig_gamma1_first}
\end{figure}

\begin{lemma}\label{lem_gamma1_last}%2.21.half
Suppose $e$ is an edge in $\mc{P}(\Sigma)$ with initial endpoint at  $p_{ij}^{i}$, then $e$ does not move any $\gamma_n$-loop of $p^i_{ij}$ to $\psi_1$.
Similarly, if $e$ is an edge in $\mc{P}(\Sigma)$ with terminal endpoint at  $p_{ij}^{j}$ then $e$ does not move $\gamma_1$ to any $\psi_n$-loop.  
\end{lemma}
\begin{proof}
As we did in Lemma \ref{lem_gamma1_first}, it is enough to show the first statement. The case $\psi_1 \mapsto \gamma_1$ has been discussed in the proof of Lemma \ref{lem_gamma1_first}. 
%By Lemma \ref{lemma_red_intersect}, if $\gamma_1 \mapsto \psi_1$ we must have that both bound two punctures each. In this case, since $|\gamma_1\cap \psi_1|=2$, the two punctures bounded by $\gamma_1$ must be separated by $\psi_1$, contradicting Remark \ref{remark_combinatorial_property}. Thus, $\gamma_1 \mapsto \psi_1$ is impossible. 
%
%Suppose now that the A-move corresponding to $e$ relates $\gamma_1$ with the cut-curve $\psi_2\in p^j_{ij}$. 

We study the case $\gamma_1 \mapsto \psi_2$. In particular, $\gamma_1$ and $\psi_1$ must be disjoint because the endpoint of $e$ is $p^j_{ij}$. Thus, Lemma \ref{lemma_red_intersect} forces both $\gamma_1$ and $\psi_1$ to bound two punctures each. The 4-holed sphere corresponding to $e$ is drawn in Figure \ref{fig_221_half}(a). Observe that we are forced, by Lemma \ref{lemma_pants_structure}, to have one cut-curve inside $\partial_4$ and one compressing curve $x$. Here, the sets of curves $\{x,\partial_2, \partial_4\}$ and $\{h_1, h_2, \psi_1\}$ agree. Since $\psi_1$ bounds two punctures, we can assume $\partial_4 = h_1$.
Moreover, Part 2 of Lemma \ref{lem_experimental_2} implies that $\psi_1 = \partial_2$, leaving us with $x=h_2$.
% If $\partial_2 = h_2$, then $h'_2\in p^j_{jk}$ bounds two punctures, say $\{p,w\}$, and is forced to intersect $\gamma_1$ in two points. One can see this by focusing in the 4-holed sphere with boundaries corresponding $\psi_2$ and the punctures $p$, $v$ and $w$. In the 4-holed sphere, the conditions $|\gamma_1 \cap \psi_2|=|\partial_2 \cap h'_2|=2$ and $h'_2\cap \psi_2 =\emptyset$ give us that $|\gamma_1 \cap h'_2|=2$. Now, we can take properly embedded arcs $c$ and $c'$ in $\Sigma$ such that $\partial \eta (c) = \gamma_1$ and $\partial \eta(c')=h'_2$. The intersection between $\gamma_1$ and $h'_2$ imply $c\cap c' = \partial c \cap \partial c' = \{p\}$. Moreover, since $\gamma_1$ is a reducing curve in $\D_{ij}\cap \D_{ik}$ and $h'_2 \in \D_{jk}$, we conclude that $c$ is a shadow arc for the tangles $T_{ij}$ and $T_{ik}$, and $c'$ is a shadow for $T_{jk}$. Thus, by Lemma \ref{criterion_destab}, the bridge trisection is stabilized. 
\begin{figure}[h]
\centering
\labellist \small\hair 2pt 
\pinlabel {(a)} at -20 165 
\pinlabel {$p$} at 25 140
\pinlabel {$v$} at 150 155
\pinlabel {$w$} at 175 155
\pinlabel {$q$} at 25 63 
\pinlabel {$s$} at 250 63
\pinlabel {$u$} at 105 60
\pinlabel {$x$} at 170 70
\pinlabel {$\psi_3$} at 220 80
\pinlabel {$\partial_4$} at 60 25
\pinlabel {$\partial_2$} at 120 160
\pinlabel {$\gamma_1$} at 50 160
\pinlabel {$\psi_2$} at 240 120

\pinlabel {(b)} at 320 165
\pinlabel {$q$} at 375 75
\pinlabel {$u$} at 445 85
\pinlabel {$x=h_2$} at 450 105
\pinlabel {$h'_2$} at 520 80
\pinlabel {$\psi_3$} at 545 50
\pinlabel {$c'$} at 400 45
\pinlabel {$\gamma_1$} at 340 140
\pinlabel {$h'_1$} at 575 40

\endlabellist 
\includegraphics[width=.85\textwidth]{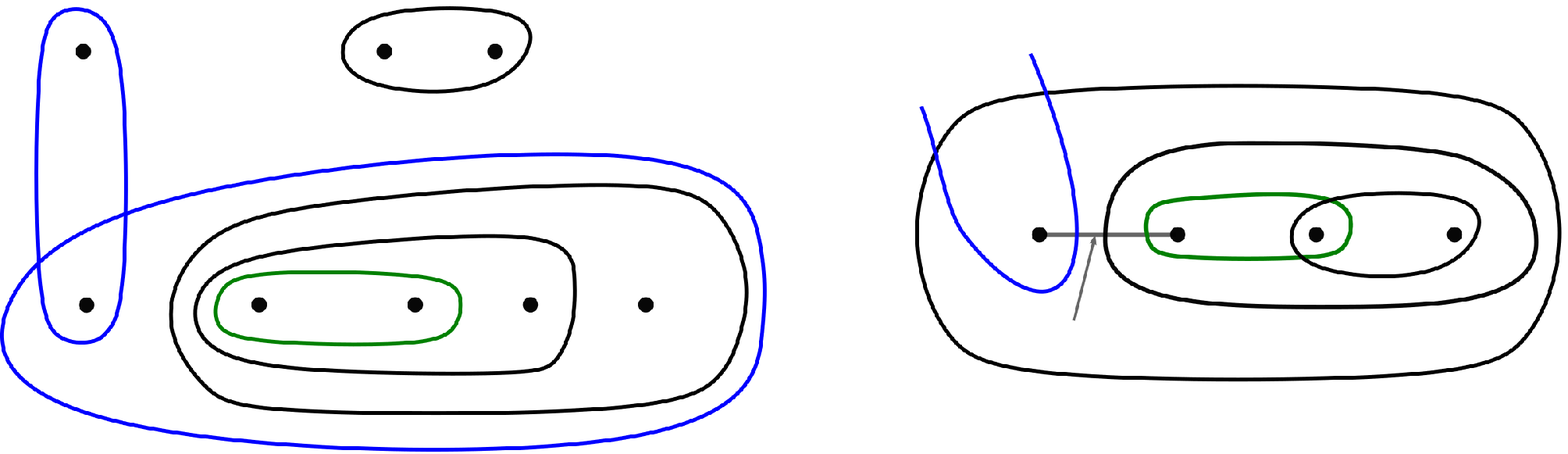}
\caption{A close look at the A-move $\gamma_1 \mapsto \psi_2$.}
\label{fig_221_half}
\end{figure}

% The only option left is $\partial_2 = \psi_1$, which implies $x = h_2$. 
Focus on $h'_1 \in p^j_{jk}$. If $h'_1$ bounds two punctures, we can proceed as in the previous paragraph and conclude that the bridge trisection is stabilized. Thus $h'_1$ must bound four punctures. 
Here, $h'_1$ bounds $q$ and the curve $\psi_3$. By focusing in such disk (see Figure \ref{fig_221_half}(b)), we see that $h'_2$ must be disjoint from $\gamma_1$ because $(h_2\cup h'_2)\cap \psi_3=\emptyset$. This lets us to find a a shadow $c'$ for $T_{jk}$ connecting $q$ and $u$, such that $c'$ is disjoint from $h'_1$ and $h'_2$. %, intersecting $\gamma_1$ in one point. 
We can slide $c'$ over $h'_1$ and $h'_2$ in order to arrange that $c'$ and $\gamma_1$ intersect once. Thus, the bridge trisection is stabilized by Lemma \ref{criterion_destab}.
\end{proof}

%%%%%%%%%%%%%%%%%%%%%%%%%%%%%%%%%%%%%%%%%%%%%%%%%%%%%%%%%%%%%%%%%%%%%%%%%%%5
\subsection{Improved lower bound}
We are ready to prove the lower-bound of Theorem \ref{thm_main}. The main result of this Section is Theorem \ref{thm_lower_bound} which states that the Kirby-Thompson invariant of a (4,2)-bridge trisection of a knotted sphere in $S^4$ is at least $15$. 

As before, let $S$ be a connected surface in $S^4$ with an unstabilized, irreducible $(4,2)$-bridge trisection $\mc{T}$. 
Fix $\{i,j,k\} = \{1,2,3\}$. Let $(p_{ij}^{i},p_{ik}^{i})$ and  $(p_{ij}^{j},p_{jk}^{j})$ be defining pairs. Denote the curves in $p^i_{ij}$ and $p^j_{ij}$ by $p_{ij}^{i}=\{\gamma_1,\gamma_2,\gamma_3,f_1,f_2\}$ and $p_{ij}^{j}=\{\psi_1,\psi_2,\psi_3,h_1,h_2\}$ as in Lemma \ref{lemma_pants_structure}. 
We know that $f_1, f_2, h_1, h_2$ bound compressing disks for $T_{ij}$; also each $\gamma_n$-curve is a reducing or cut-reducing curve for $L_i$ and each $\psi_n$-curve is a reducing or cut-reducing curve for $L_j$; in fact, $\gamma_1$ and $\psi_1$ are reducing curves and the others are cut-reducing curves. Recall that there are essential, simple closed curves $f'_1$ and $f'_2$ such that $p^i_{ik}=\{\gamma_1, \gamma_2, \gamma_3, f'_1, f'_2\}$ completes an efficient defining pair $(p^i_{ij}, p^i_{ik})$. Likewise, there are essential, simple closed curves $h'_1$ and $h'_2$ such that $p^j_{jk}=\{\psi_1, \psi_2, \psi_3, h'_1, h'_2\}$ completes an efficient defining pair $(p^j_{ij}, p^j_{jk})$. 

The proof of Theorem \ref{thm_lower_bound} will be broken into three propositions: \ref{prop_longer_proposition}, \ref{prop_five_22} and \ref{prop_five_44}. Each of them proving that $d(p^i_{ij}, p^j_{ij})\geq 5$ for each pair, depending on the number of punctures bounded by $\gamma_1$ and $\psi_1$. We begin in Proposition \ref{lem_new_long_proposition} showing that such distance is at least 4. 

\begin{proposition}\label{lem_new_long_proposition}
If $\lambda(ij)$ is a path from $p_{ij}^{i}$ to $p_{ij}^{j}$. The length of $\lambda(ij)$ is at least 4. If it is equal to 4, then at least one of $f_1$ and $f_2$ is unmoved by $\lambda(ij)$. 
\end{proposition}

%%%%%%%%%%%%%%%%%%%%%%%%%%%%%%%%%%%
\begin{proof}
By Lemma \ref{lemma_psi_f_distinct}, no $\psi_n$ loop is equal to $f_1$ or $f_2$ and no $\gamma_n$ loop is equal to $h_1$ or $h_2$. Thus, if some $\gamma_n$-loop is unmoved by $\lambda(ij)$, then it is equal to some $\psi_n$-loop. But by Lemma \ref{lem_1}, this implies that $\mc{T}$ is reducible, a contradiction. Thus, $\lambda(ij)$ moves every $\gamma_n$-loop, so the length of $\lambda(ij)$ is at least 3. If it is equal to 3, then $f_1$ and $f_2$ are unmoved by $\lambda(ij)$ and if it is equal to 4, at least one of $f_1, f_2$ is unmoved by $\lambda(ij)$, as desired. Thus, we simply need to show that the length is not 3. 

Assume, for a contradiction, that the length of $\lambda(ij)$ is 3. As $f_1, f_2$ are unmoved, by Lemma \ref{lemma_psi_f_distinct}, $\{f_1, f_2\} = \{h_1, h_2\}$. By Lemma \ref{lem_1}, each of the curves $\{\gamma_1, \gamma_2, \gamma_3\}$ moves exactly once. For each $m = 1,2,3$, let $\gamma'_m$ denote the $\psi_n$-loop to which $\gamma_m$ is moved by $\lambda(ij)$. Lemmas \ref{lem_gamma1_first} and \ref{lem_gamma1_last} imply that the curves $\gamma_1$ and $\psi_1$ are not involved in the first and third A-moves of $\lambda(ij)$. Thus, $\gamma_1 \mapsto \psi_1$ must be the second A-move in $\lambda(ij)$. We can then assume that $\gamma_2$ moves first, $\gamma'_2=\psi_2$ and $\gamma'_3 = \psi_3$.

We focus on the 4-holed sphere $E$ where the A-move $\gamma_2 \mapsto \gamma'_2$ occurs. After a surface homeomorphism, we can draw $E$ as in Figure \ref{fig_case2b_part1}(a) where the parity of punctures one one side of $\partial_n$, is given by the Figure \ref{fig_case2b_part1}(a). Since $\gamma_2$ is a cut disk, one of its sides contains three punctures. Thus, we may assume that $\partial_2$ only bounds the puncture $p$ and $\partial_1$ bounds two punctures. %; forcing $\partial_3=\gamma_3$. 
We get two cases, depending on the number of punctures bounded by $\partial_3$, one or three (see Figure \ref{fig_case2b_part1}). 
\begin{figure}[h]
\centering
\labellist \small\hair 2pt 
\pinlabel {(a)} at 5 256
\pinlabel {even} at 59 199
\pinlabel {odd} at 59 111
\pinlabel {odd} at 193 195
\pinlabel {even} at 193 111
\pinlabel {$\partial_1$} [t] at 50 178
\pinlabel {$\partial_2$} [bl] at 85 127
\pinlabel {$\partial_3$} [b] at 194 221
\pinlabel {$\partial_4$} at 150 90
\pinlabel {$\gamma_2'=\psi_2$} [t] at 154 62

\pinlabel {(b)} at 315 256
\pinlabel {$\gamma_2$} [bl] at 396 236
\pinlabel {$\boundary_1$} [tr] at 348 182
\pinlabel {$\psi_2 = \gamma'_2$} at 450 28
\pinlabel {$p$} [b] at 373 110
\pinlabel {$\boundary_4$} [b] at 503 54
\pinlabel {$\boundary_3$} [b] at 492 224

\pinlabel {(c)} at 637 256
\pinlabel {$\gamma_2$} [l] at 721 240
\pinlabel {$\boundary_1$} [b] at 685 196
\pinlabel {$\psi_2 = \gamma'_2$} [t] at 811 28
\pinlabel {$p$} [b] at 691 124
\pinlabel {$r$} [t] at 774 122
\pinlabel {$s$} [t] at 811 123
\pinlabel {$x$} [b] at 853 79
\pinlabel {$q$} [t] at 816 221
\pinlabel {$\gamma_3$} at 847 60
\pinlabel {$\boundary_4$} [r] at 756 99
\endlabellist 
\includegraphics[width=.95\textwidth]{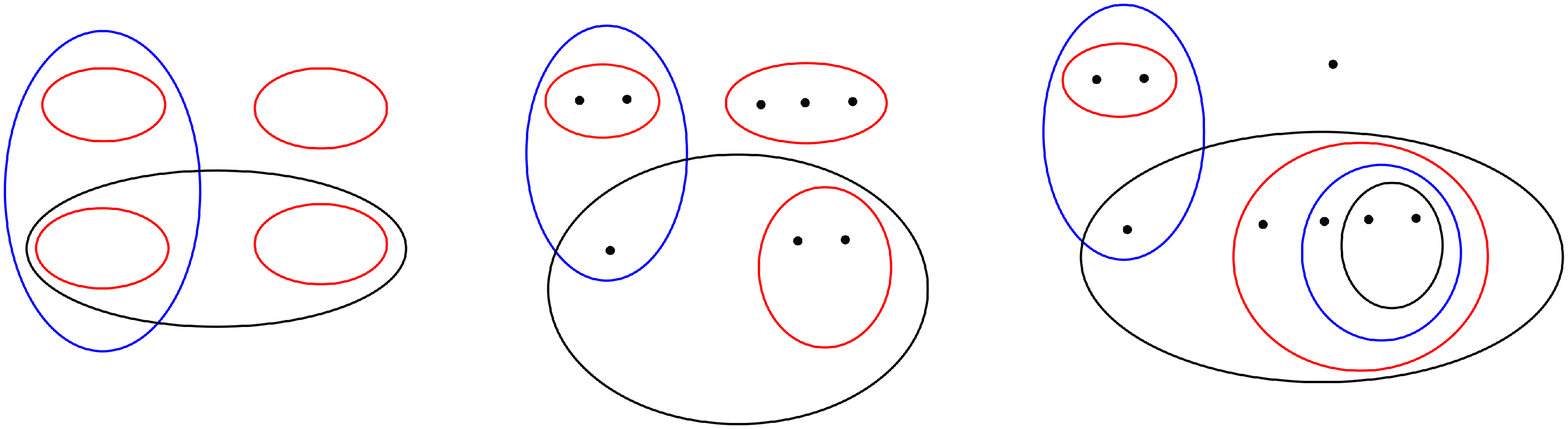}
\caption{Two subcases, depending in the number of punctures bounded by $\partial_3$.} 
\label{fig_case2b_part1}
\end{figure}

%\textbf{Subcase 2a:} 
\textbf{Case 1:} $\partial_3$ bounds three punctures. In particular, $\partial_3=\gamma_3$ bounds a cut disk. 
%Suppose first that $\gamma_3$ moves second in the path $\lambda(ij)$. By Lemma \ref{remark_reverse_lem}, we must have that $\gamma'_3 = \psi_1$, forcing $\gamma_1$ and $\psi_1$ to be disjoint. Lemma \ref{lemma_red_intersect} implies that both $\gamma_1$ and $\psi_1$ have to bound the same number of punctures. Notice that, in this setup (Figure \ref{fig_case2b_part1}(b)), $\gamma_1\in \{\partial_1, \partial_4\}$ bounds two punctures and $\gamma'_3$ cannot bound two punctures. Hence, 

By the previous paragraph, $\gamma_3$ has to be moved in third place and $\gamma_1$ in second. 
%Lemma \ref{remark_reverse_lem} implies that $\gamma'_1 = \psi_1$ and so 
Since $\gamma_1 \mapsto \psi_1$ is an A-move, we know that $|\gamma_1 \cap \psi_1|=2$. This is a contradiction due to the following argument also found in Lemma \ref{lemma_red_intersect}. Denote by $D \subset \Sigma$ the twice punctured disk bounded by $\gamma_1$. We have that $\psi_1\cap D$ consists of parallel arcs separating the punctures. Since $\psi_1$ is a reducing curve, then it bounds discs in $\Sigma$ each containing an even number of punctures. Therefore, $|\psi_1\cap D|$ is even and $|\psi_1\cap \gamma_1|$ is a multiple of 4.  

%But $\gamma_1\in \{\partial_1, \partial_4\}$ bounds two punctures, and we have seen (first paragraph of Proof of Lemma \ref{lem_gamma1_first}) that this condition forces $|\gamma_1 \cap \psi_1|\equiv 0$ modulo 4, a contradiction. Hence, $\partial_3$ cannot bound three punctures. 

%\textbf{Subcase 2b:} 
\textbf{Case 2:} $\partial_3$ bounds one puncture, named $q$. 

After a surface homeomorphism, we can draw the curves as in Figure \ref{fig_case2b_part1}(c). Recall that $\gamma_1\mapsto \psi_1$ is the second A-move in $\lambda(ij)$. It follows that $\gamma_1 \in \{\partial_1, \partial_4, x\}$ and observe that all the possible configurations for the curve $\gamma'_1=\psi_1$ in Figure \ref{fig_case2b_part1}(c) contradict the combinatorial conditions in Remark \ref{remark_combinatorial_property}. Thus, this case cannot occur. 
\end{proof}

\begin{proposition}\label{prop_longer_proposition}
Suppose $\gamma_1$ bounds two punctures and $\psi_1$ bounds four. Then any path $\lambda(ij)$ from $p_{ij}^{i}$ to $p_{ij}^{j}$ must be of distance at least five. 
\end{proposition}

\begin{proof}[Proof of Proposition \ref{prop_longer_proposition}]
By Proposition \ref{lem_new_long_proposition} it is enough to show the distance from $p^i_{ij}$ to $p^j_{ij}$ is not four. By way of contradiction, let $\lambda$ be a geodesic path of length four between such pants decompositions. 
By Lemmas \ref{lem_1} and \ref{lemma_psi_f_distinct}, each $\gamma_n$-curve must move at least once. We have two cases, depending on how many curves of $\{f_1, f_2\}$ are moved. 

\textbf{Case 1:} $\lambda$ moves one curve of $\{f_1, f_2\}$.

Without loss of generality $f_1$ is moved and so $f_2=h_2$ is fixed. In this case, each of $\{\gamma_1, \gamma_2, \gamma_3, f_1\}$ is moved once to one curve among $\{\psi_1, \psi_2, \psi_3, h_1\}$. Denote by $x'$ the image of a loop $x$ under the path $\lambda$; i.e., $x\mapsto x'$ differ by one A-move. 

First observe that, since $h_n$ and $\gamma_1$ are compressing curves for the same tangle, it must happen that if $\gamma_1$ bounds $\{p,q\}$, then they are both on the same side of $h_n$. Thus, $|\gamma_1\cap h_n|\equiv 0$ modulo 4. In particular, $\gamma'_1 \neq h_n$. Similarly $\gamma'_1 \neq \psi_1$. Thus, $\gamma'_1$ bounds a cut disk, say $\gamma'_1=\psi_2$. 
In particular $|\gamma_1 \cap \psi_2|=2$. This is a contradiction to Part 1 of Lemma \ref{lem_experimental_2}. Hence, this case cannot occur. 

% Notice that the 4-holed sphere corresponding to the A-move $\gamma_1 \mapsto \psi_2$ must have two boundaries $\partial_2$ and $\partial_4$ bounding two and four punctures, respectively (see Figure \ref{fig_low5_case1}(a)). This yields to the existence of at least four pairwise disjoint curves bounding an even number of punctures: $\gamma_1$, $\partial_2$, $\partial_4$ and $x$. But the path $\lambda$ only involves curves in $p^i_{ij}\cup p^j_{ij}$ and $\gamma_1\cap \psi_1 \neq \emptyset$ by Lemma \ref{lemma_red_intersect}, so $\{\partial_2, \partial_4, x\}=\{h_1, f_1, f_2=h_2\}$. 
% %Now observe that the A-move $\gamma_1 \mapsto \gamma'_1$ cannot be the first one in $\lambda$ since $h_1\not\in p^i_{ij}$. Moreover, since $\psi_1$ is not in the picture yet, $\gamma_1 \mapsto \gamma'_1$ cannot be last either. 
% By Lemmas \ref{lem_gamma1_first} and \ref{lem_gamma1_last}, the A-move $\gamma_1 \mapsto \psi_2$ cannot be first nor last. 
% Leaving us with two subcases. 

\begin{figure}[h]
\labellist \small\hair 2pt 
\pinlabel{(a)} [tl] at 1 186
\pinlabel {$p$} [t] at 45 145
\pinlabel {$q$} [b] at 45 61
\pinlabel {$v$} [b] at 128 151
\pinlabel {$w$} [b] at 165 151
\pinlabel {$u$} [b] at 112 62
\pinlabel {$t$} [b] at 148 62
\pinlabel {$r$} [b] at 188 62
\pinlabel {$s$} [b] at 224 62
\pinlabel {$\gamma_1$} [bl] at 73 156
\pinlabel {$\boundary_2$} [bl] at 169 167
\pinlabel {$x$} [tl] at 158 47
\pinlabel {$y$} [tl] at 208 52
\pinlabel {$\boundary_4$} [tr] at 107 30
\pinlabel {$\gamma'_1 = \psi_2$} [bl] at 191 103

\pinlabel {(b)} [tl] at 287 186
\pinlabel {$r$} [b] at 353 70
\pinlabel {$\gamma_2$} [bl] at 381 173
\pinlabel {$\boundary_1$} [b] at 353 119
\pinlabel {$\boundary_3$} [b] at 454 171
\pinlabel {$\boundary_4$} [tr] at 419 42
\pinlabel {$x$} [b] at 475 51
\pinlabel {$\psi_1 = \gamma'_2$} [bl] at 515 106
\endlabellist
\centering
\includegraphics[width=.9\textwidth]{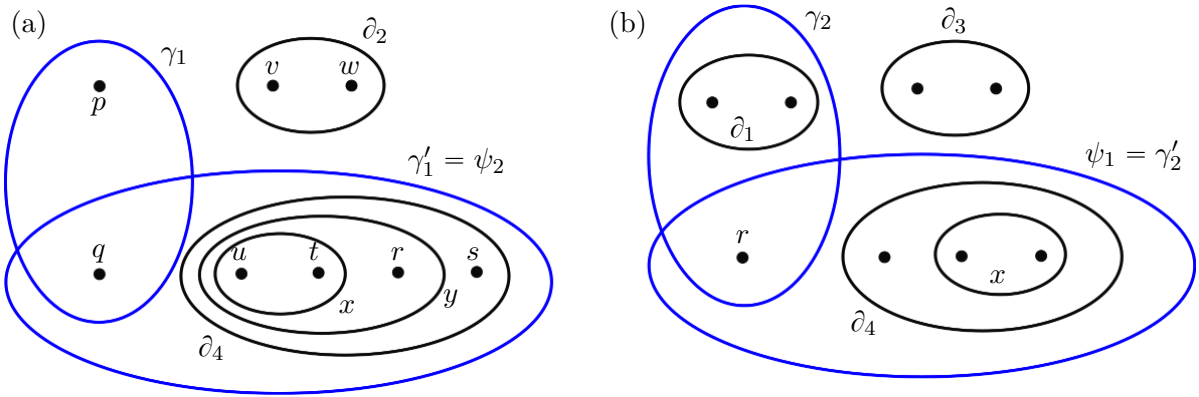}
\caption{How the curves in $\Sigma$ look for specific A-moves.} 
\label{fig_low5_case1}
\end{figure}

\textbf{Case 2:} $\lambda$ fixes $\{f_1, f_2\}$. 

We can write $f_1=h_1$ and $f_2=h_2$. In this case, one of $\{\gamma_1, \gamma_2, \gamma_3\}$ will move twice and the other $\gamma_n$-loops move once along $\lambda$. For the curve $\gamma_j \in \{\gamma_1, \gamma_2, \gamma_3\}$ that moves twice, denote by $\theta$ the curve $\gamma_j'$. We will also refer to $\theta$ as the \textbf{pivotal curve}.

\textbf{Subcase 2a:} $\gamma_1$ moves once along $\lambda$. By Lemma \ref{lemma_red_intersect} $|\gamma_1 \cap \psi_1|\geq 4$ so $\gamma'_1$ must bound a cut disk, say $\gamma'_1=\psi_2$. 
In particular $|\gamma_1 \cap \psi_2|=2$. This is impossible since it contradicts Part 1 of Lemma \ref{lem_experimental_2}.
%As we did in Case 1, when looking at the 4-holed sphere corresponding to the A-move $\gamma_1 \mapsto \psi_2$, we conclude that it has two boundaries $\partial_2$ and $\partial_4$ bounding two and four punctures, respectively. This forces the existence of at least four pairwise disjoint curves bounding an even number of punctures (see Figure \ref{fig_low5_case1}(a)). Thus $\theta$ bounds an even number of punctures and the set $\{\partial_2, \partial_4, x\}$ is equal to $\{f_1=h_1, f_2=h_2, \theta\}$. We know that $h_n$ must bound two punctures since $\psi_1$ bounds four, so $\partial_4=\theta$. Since $f_1$ and $f_2$ are fixed, we may assume $\partial_2=f_2$ and $x=f_1$. 
%
%In particular, the rest of the A-moves in $\lambda$ occur inside the component of $\Sigma-\psi_2$ containing $p$. 
%Using the notation in Figure \ref{fig_low5_case1}(a), the only curves that could move after $\gamma_1 \mapsto \gamma'_1$ are $\partial_4=\theta$ and $y$, both being final moves. 
%If $\theta$ moves before $y$ does, we notice that $\theta'$ must bound two or four punctures. So $\theta'=\psi_1$ will separate $p$ and $q$, contradicting Remark \ref{remark_combinatorial_property}. If $y$ moves before $\theta$ does, then $y'$ either bounds two or three punctures. Since $\psi_1$ bounds four punctures, $y'$ must bound three. This forces $\theta'=\psi_1$. This is again a contradiction since $\theta'$ has to separate $p$ and $q$. Hence, subcase 2a cannot occur. 

\textbf{Subcase 2b:} $\gamma_1$ moves twice along $\lambda$. 
We will first see that $\gamma'_n\neq \psi_1$ for any $n$. In particular, $\theta'=\psi_1$ and the following property holds: at each vertex of $\lambda$, there are at most three pairwise disjoint curves bounding an even number of punctures. 

By Lemma \ref{lemma_red_intersect}, $\gamma'_1\neq \psi_1$. Suppose, without loss of generality, that $\gamma'_2=\psi_1$. 
The 4-holed sphere corresponding to the A-move $\gamma_2 \mapsto \psi_1$ has one boundary component bounding one puncture, $r$, and boundary loops $\partial_1$, $\partial_3$ and $\partial_4$ bounding two, two and three punctures, respectively (see Figure \ref{fig_low5_case1}(b)). Here, there are four pairwise disjoint curves bounding an even number of punctures: $\{\psi_1, \partial_1, \partial_3, x\}$. Since %$f_n=h_n$ are fixed, 
$\gamma_1 \cap \psi_1 \neq \emptyset$ by Lemma \ref{lemma_red_intersect}, we know that $\{f_1, f_2, \theta\}=\{\partial_1, \partial_3, x\}$.
If $\partial_1 = \theta$, then $\gamma_1$ will bound $r$ and one of the two punctures bounded by $\partial_1$. This is impossible since such punctures are on distinct sides of $\psi_1$. Hence $\partial_1 = f_1=h_1$. 

Observe that the two punctures bounded by $\gamma_1$ must be separated by $\theta=\gamma'_1$; if not, then $|\gamma_1 \cap \theta|\equiv 0$ modulo 4 which makes impossible the A-move $\gamma_1\mapsto \theta$. We use this to see that if $\partial_3 = \theta$, then $\gamma_1$ would bound one puncture inside $\partial_3$ with one puncture inside $\partial_4$. These points are in distinct sides of $\psi_1$ (see Figure \ref{fig_low5_case1}(b)) which is a contradiction to Remark \ref{remark_combinatorial_property}. 
Hence, $x=\theta$,  $\partial_3 = f_2$ and $\partial_1=f_1$. Notice that all the incoming A-moves will occur in the side of $\psi_1$ containing $\partial_4$. This forces $p^j_{ij}$ to have at least four curves bounding an even number of punctures, a contradiction to Lemma \ref{lemma_pants_structure}. This concludes that $\gamma'_n\neq \psi_1$, as desired.

By the above, the $\gamma_n$-cut curves move once along $\lambda$ to $\psi_n$ cut curves. Without loss of generality, $\gamma'_n = \psi_n$ for $n=2,3$. \textbf{We will assume that $\gamma_3\mapsto \psi_3$ is not the last A-move in $\lambda$ in the path $\lambda$}; if not, we can relabel the $\gamma_n$ curves. We will focus on the 4-holed sphere corresponding to the A-move $\gamma_3 \mapsto \gamma'_3$ (see Figure \ref{fig_low5_case2b}(a)). We have two cases, depending on the number of punctures bounded by $\partial_2$ and $\partial_3$. 
\begin{figure}[h]
\labellist \small\hair 2pt 
\pinlabel {(a)} at 4 224
\pinlabel {even} at 60 173
\pinlabel {odd} at 60 84
\pinlabel {odd} at 193 173
\pinlabel {even} at 193 84
\pinlabel {$\boundary_1$} [t] at 59 150
\pinlabel {$\boundary_2$} [tl] at 79 62
\pinlabel {$\boundary_3$} [t] at 193 150
\pinlabel {$\boundary_4$} [tr] at 176 65
\pinlabel {$\gamma'_3 = \psi_3$} [t] at 135 37
\pinlabel {$\gamma_3$} [bl] at 102 193

\pinlabel {(b)} at 328 224
\pinlabel {$\boundary_1$} [b] at 382 190
\pinlabel {q} [b] at 365 148
\pinlabel {p} [b] at 400 148
\pinlabel {$\boundary_4$} [tr] at 476 42
\pinlabel {$\gamma_3$} [r] at 335 164
\pinlabel {$\gamma'_3 = \psi_3$} [bl] at 560 146
\pinlabel {$t$} [b] at 385 82
\pinlabel {$x$} [b] at 548 102
\pinlabel {$y$} [b] at 528 43
\pinlabel {$s$} [t] at 469 76
\pinlabel {$u$} [t] at 507 79
\pinlabel {$v$} [t] at 540 82
\pinlabel {$w$} [t] at 560 82
\pinlabel {$r$} [b] at 513 182

\pinlabel {(c)} at 679 224
\pinlabel {$\boundary_1$} [t] at 737 146
\pinlabel {$t$} [b] at 725 170
\pinlabel {$u$} [b] at 754 171
\pinlabel {$r$} [b] at 837 171
\pinlabel {$p$} [b] at 864 171
\pinlabel {$q$} [b] at 894 171
\pinlabel {$x$} [t] at 877 206
\pinlabel {$s$} [b] at 745 79
\pinlabel {$\gamma_3$} [b] at 737 212
\pinlabel {$\gamma'_3 = \psi_3$} [t] at 807 36
\pinlabel {$\boundary_4$} [t] at 873 107
\pinlabel {$\boundary_3$} [br] at 821 195
\endlabellist
\centering
\includegraphics[width=.95\textwidth]{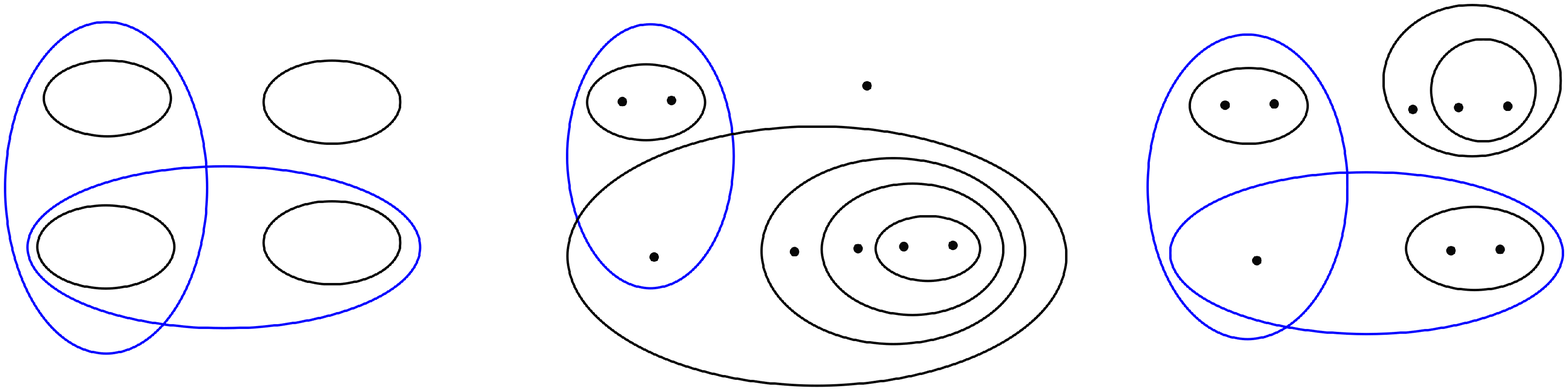}
\caption{The three possibilities occuring in Case 2b.} 
\label{fig_low5_case2b}
\end{figure}

\textbf{Subcase 2b(i):} Both $\partial_2$ and $\partial_3$ bound one puncture each. We adopt the notation in Figure \ref{fig_low5_case2b}(b). In this case, we already have three pairwise disjoint curves bounding an even number of punctures $\{\partial_1, \partial_4, x\}$, so there is a curve $y$ bounding $x$ and one puncture $u$ (see Figure \ref{fig_low5_case2b}(b)). 
Recall that $h_n$ bounds two punctures and $f_n=h_n$ is fixed by $\lambda$. This implies that $\partial_1 = f_1$, $x=f_2$ and $\partial_4\in \{\theta, \psi_1\}$. Now, since $\gamma_3 \mapsto \psi_3$ is not the last A-move in $\lambda$, there are two possible curves which may move next, $y$ and $\theta$. 

Suppose first that $y$ moves before $\theta$ does. (The curve $\theta$ may or may not move). Then $y'$ must be a cut disk and we get $y'=\psi_2$ and $y=\gamma_2$. Using the notation of Figure \ref{fig_low5_case2b}(b), since $\gamma_1$ bounds two punctures and is disjoint from $\gamma_2$ and $\gamma_3$, we obtain that $\gamma_1$ bounds $\{r,s\}$. But $\partial_4$ separates such punctures, so the only option is $\partial_4=\theta$. Now, the fact that $y'$ bounds a cut disk implies that it bounds the two punctures inside $x$ and $s$. The next move $\theta\mapsto \psi_1$ is forced to separate $r$ and $s$, contradicting Remark \ref{remark_combinatorial_property}. 

It remains to study what happens when $\partial_4$ moves before $y$. (The curve $y$ may or may not move). Here, $\partial_4=\theta$. Focusing on Figure \ref{fig_low5_case2b}(b), we observe that $\psi_1=\theta'$ bounds the two punctures inside $x=f_2$, together with $t$ and $u$. By Remark \ref{remark_combinatorial_property}, $\gamma_1$ bounds either $\{r,s\}$ or $\{t,u\}$. The latter is impossible since $\gamma_3$ is disjoint from $\gamma_1$ and $\gamma_3$ separates such punctures. Thus $\gamma_1$ bounds $\{r,s\}$. Since $\psi_3$ separates $r$ and $s$, the A-move $\gamma_1\mapsto \theta$ must appear in $\lambda$ before $\gamma_3\mapsto \psi_3$. Moreover, the move $\gamma_2\mapsto \psi_2=y$ cannot happen between $\gamma_1\mapsto \theta$ and $\gamma_3\mapsto\psi_3$. 
This claim holds because, if $\gamma_2$ moves between $\gamma_1$ and $\gamma_3$, it would force $\gamma_2$ to bound the two punctures inside $x=f_2$ together with $s$, which implies the contradiction $\gamma_1 \cap \gamma_2\neq \emptyset$. We are left with two possibilities, depending on the order of the curves moving: $\left( \gamma_2, \gamma_1, \gamma_3, \theta\right)$ or $\left(\gamma_1, \gamma_3, \theta, \gamma_2\right)$. Figure \ref{fig_low5_case2bi_paths} showcases the two possible paths and what punctures are bounded by each curve. 
\begin{figure}[h]
\centering
\labellist \small\hair 2pt  
\pinlabel{$f_1$} [b] at 45 110
\pinlabel{$r$} [b] at 125 105
\pinlabel{$\psi_3$} [b] at 83 150
\pinlabel{$\gamma_2$} [b] at 10 150
\pinlabel{$\gamma_1$} [b] at 140 80
\pinlabel{$\gamma_2$} [b] at 155 85
\pinlabel{$s$} [b] at 117 70
\pinlabel{$u$} [b] at 133 42
\pinlabel{$f_2$} [b] at 161 35
\pinlabel{$\psi_2$} [b] at 172 21
\pinlabel{$\theta$} [b] at 90 21
\pinlabel{$t$} [b] at 50 60
\pinlabel{$\psi_1$} [b] at 50 10

\pinlabel{$f_1$} [b] at 303 110
\pinlabel{$r$} [b] at 385 105
\pinlabel{$\gamma_1$} [b] at 360 110
\pinlabel{$\psi_3$} [b] at 390 140
\pinlabel{$\theta$} [b] at 400 75
\pinlabel{$\gamma_2$} [b] at 420 93
\pinlabel{$f_2$} [b] at 445 85
\pinlabel{$u$} [b] at 388 43
\pinlabel{$s$} [b] at 359 73
\pinlabel{$\psi_2$} [b] at 320 10
\pinlabel{$\psi_1$} [b] at 290 5
\pinlabel{$\gamma_3$} [b] at 260 40
\pinlabel{$t$} [b] at 306 45

\endlabellist
\includegraphics[width=.8\textwidth]{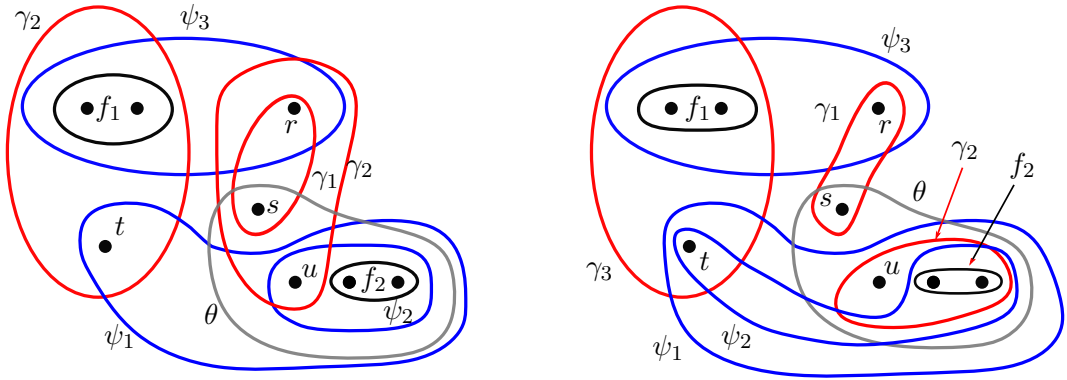}
\caption{Two paths.} 
\label{fig_low5_case2bi_paths}
\end{figure}

We focus on the sub path of $\lambda$ corresponding to the consecutive A-moves $\gamma_1 \mapsto \theta$ followed by $\gamma_3 \mapsto \psi_3$. The second A-move occurs inside a 4-holes sphere with boundaries associated to $t$, $r$, $f_1$ and $\theta$ (see Figure \ref{fig_low5_case2bi_end}(a)). The fact that $\gamma_1$ and $\gamma_3$ are disjoint implies that the condition $|\gamma_3\cap \psi_3|=2$ is equivalent to $|\gamma_1\cap \psi_3|=2$. One can see this claim by noticing that the curves $\gamma_3$ and $\partial \eta(\gamma_1\cup \theta)$ are isotopic in the 4-holed sphere. 
The condition $|\gamma_1 \cap \psi_3|=2$ contradicts the statement of Lemma \ref{lem_experimental_2}. 
% Recall that there exist curves $h'_1$, $h'_2$ in $\Sigma$ such that $p^j_{jk}=\{\psi_1, \psi_2, \psi_3, h'_1, h'_2\}$ forms an efficient defining pair $(p^j_{ij}, p^j_{jk})$. Since $|h'_1\cap h_1|=2$ and $h'_1\cap \psi_3=\emptyset$, we can think of $h'_1$ as a curve inside the disk bounded by $\psi_3$ with punctures $p$, $q$ and $r$. In there, $\gamma_1$ is an arc with both ends in $\psi_3$, separating $r$ from the $p$, $q$, and disjoint from $f_1=h_1$. All this implies that $h'_1$ intersects $\gamma_1$ in two points. To end, we take properly embedded arcs $c$ and $c'$ in $\Sigma$ such that $h'_1 = \partial \eta(c')$ and $\gamma_1 = \partial\eta(c)$. Since $|\gamma_1 \cap h'_1|=2$, we can find arcs such as $c\cap c'=\partial c \cap \partial c'=\{r\}$. The arc $c$ (resp. $c'$) is a shadow for the tangles $T_{ij}$ and $T_{ik}$ (resp. $T_{jk}$) because it bounds two punctures and $\gamma_1$ (resp. $h'_1$) belongs to $\mathcal{D}_{ij}\cap \mathcal{D}_{ik}$ (resp. $\mathcal{D}_{jk}$). Hence, by Lemma \ref{criterion_destab} the bridge trisection is stabilized. 
In other words, subcase 2b(i) is impossible. 
\begin{figure}[h]
\centering
\labellist \small\hair 2pt 
\pinlabel {(a)} at 5 190
\pinlabel {$f_1$} [b] at 68 132
\pinlabel {$q$} [t] at 57 115
\pinlabel {$p$} [t] at 80 115
\pinlabel {$t$} [t] at 66 55
\pinlabel {$\gamma_3$} [tl] at 71 14
\pinlabel {$\gamma_1$} at 130 130
\pinlabel {$r$} [t] at 126 114
\pinlabel {$s$} [b] at 126 57
\pinlabel {$u$} [b] at 147 57
\pinlabel {$f_2$} [t] at 182 42
\pinlabel {$\theta$} [tr] at 118 24
\pinlabel {$\psi_3$} [b] at 115 174

\pinlabel {(b)} at 228 190
\pinlabel {$t$} [b] at 252 119
\pinlabel {$u$} [b] at 274 119
\pinlabel {$\boundary_1$} [b] at 262 146
\pinlabel {$\psi_2 = \boundary'_3$} [b] at 269 170
\pinlabel {$\boundary_q = f_1$} [t] at 327 43
\pinlabel {$\psi_3$} [l] at 369 44
\pinlabel {$\gamma_1$} [bl] at 321 145
\pinlabel {$x = \theta$} [bl] at 353 131
\pinlabel {$r$} [b] at 304 119
\pinlabel {$p$} [b] at 335 117
\pinlabel {$q$} [b] at 357 117
\pinlabel {$s$} [t] at 261 55
\pinlabel {$\gamma_2 = \boundary_3$} [b] at 341 176

\pinlabel {(c)} at 426 190
\pinlabel {$t$} [t] at 447 114
\pinlabel {$u$} [t] at 469 114
\pinlabel {$\gamma_3$} [bl] at 469 143
\pinlabel {$\gamma_1$} [r] at 437 87
\pinlabel {$\theta$} [b] at 457 132
\pinlabel {$r$} [tl] at 507 114
\pinlabel {$p$} [t] at 530 114
\pinlabel {$q$} [t] at 553 114
\pinlabel {$s$} [b] at 447 58
\pinlabel {$f_1 = \boundary_4$} [tr] at 512 30
\pinlabel {$\boundary_3$} [b] at 535 138
\pinlabel {$\psi_3$} [l] at 559 62
\endlabellist
\includegraphics[width=1\textwidth]{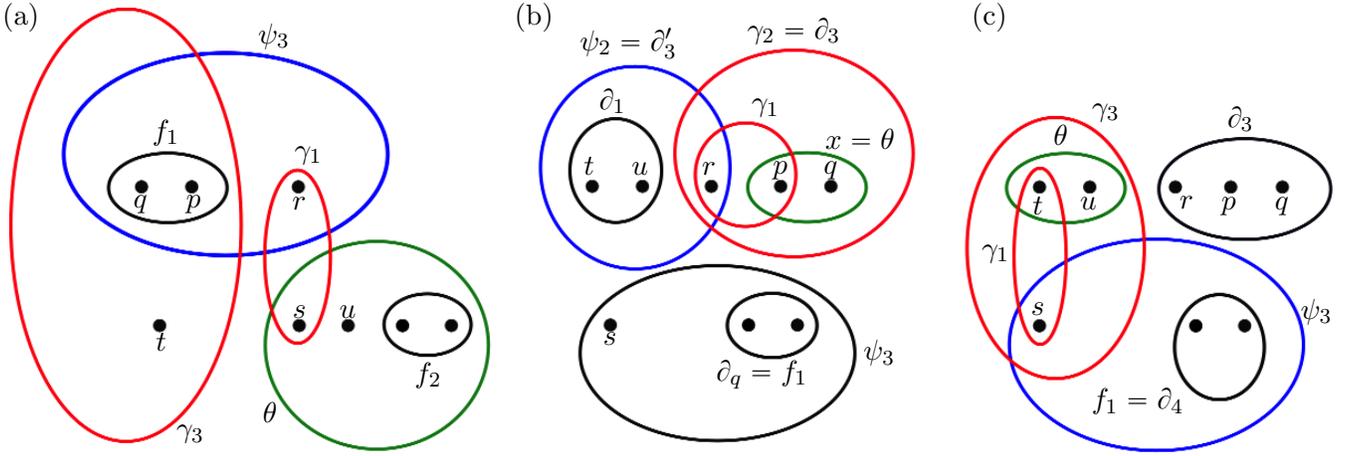}
\caption{Curves interacting in the consecutive A-moves $\gamma_1 \mapsto \theta$, $\gamma_n \mapsto \psi_n$ for a fixed $n$.} 
\label{fig_low5_case2bi_end}
\end{figure}

\textbf{Subcase 2b(ii):} Only one of $\{\partial_2, \partial_3\}$ bounds one puncture. Without loss of generality, $\partial_2$ bounds one puncture and $\partial_3$ three. This forces the setup in Figure \ref{fig_low5_case2b}(c). 
The curves along the path $\lambda$ bounding an even number of punctures are $\gamma_1$, $\psi_1$, $f_1=h_1$, $f_2=h_2$ and (possibly) $\theta$. %But we have seen that $\psi_1\neq \gamma'_n$ for all $n$, so $\partial_4$ must not move after $\gamma_3$. 
But we have seen that $\theta'=\psi_1$ and $\gamma'_1 = \theta$. This implies that $\partial_4 \not \in \{\gamma_1, \theta, \psi_1\}$ since all the A-moves starting at $\partial_4$ will be forced to end at curves bounding two punctures. 
Thus we may assume that $\partial_4=f_1$. Since no curve at this moment bounds four punctures, there should be another A-move after $\gamma_3 \mapsto \psi_3$. Using the notation in Figure \ref{fig_low5_case2b}(c), the curves that might move are $\{\partial_1, \partial_3, x\}$. 

Suppose that $\partial_3$ moves first, then $\partial_3=\gamma_2$ and $\partial'_3=\psi_2$. Since $\psi_2$ bounds three punctures then $\partial'_3$ must enclose $\partial_1$ and the puncture $r$ together. Since $\psi_1$ separates the cut curves $\psi_2$ and $\psi_3$ (Figure \ref{fig_gamma1_first}), it follows that $\partial_1=f_2$ and $\psi_1$ separates $p$ and $q$. Thus, from Remark \ref{remark_combinatorial_property}, we must have $x=\theta$. Without loss of generality, $\gamma_1$ encloses $r$ and $p$ (see Figure \ref{fig_low5_case2b}(c)). % and so $\theta'=\psi_1$ encloses $\partial_1$, together with $r$ and $p$. 
We now focus in the consecutive A-moves $\gamma_1 \mapsto \theta$, $\partial_3=\gamma_2 \mapsto \psi_2$. Observe that $\gamma_2 \mapsto \psi_2$ occurs in a 4-holed sphere with boundaries corresponding to $\psi_3$, $r$, $\partial_1=f_2$ and $\theta$. 
This local setup in depicted in Figure \ref{fig_low5_case2bi_end}(b). %, is the same as in the last paragraph of Subcase 2b(i). In order to see this, replace the roles of $\gamma_2$, $\psi_2$, $\partial_1$, and $\psi_3$ in Figure \ref{fig_low5_case2bi_end}(b) with $\gamma_3$, $\psi_3$, $f_1$, and $t$ in Figure \ref{fig_low5_case2bi_end}, respectively. We can then mimic the argument in such paragraph and conclude that the bridge trisection is stabilized. 
In here, the conditions $\gamma_1\cap \gamma_2=\emptyset$ and $|\gamma_2 \cap \psi_2|=2$ force $|\gamma_1\cap \psi_2|=2$. This contradicts the statement of Lemma \ref{lem_experimental_2}. 

If $x$ moves before $\partial_1$ and $\partial_3$, then $\partial_1=f_2$. In particular, $x=\gamma_1$ and $\partial_3$ must move so that $\theta'=\psi_1$ can bound four punctures. We can then redefine $x$ to be $\gamma'_1=\theta$ and proceed as if $\partial_3$ moves first (paragraph above). We get then a contradiction. 

The last case to check is when $\partial_1$ moves before $\partial_3$ and $x$. In particular $x=f_2$ and $\partial_1\in \{\gamma_1, \theta\}$. 
%Since we assumed that $\gamma_3$ moves before $\gamma_2$, we must have $\partial_3 = \gamma_2$.

First we see that that if $\partial_1=\gamma_1$, then $\partial_3$ will have to move between $\gamma_1\mapsto \theta$ and $\theta\mapsto \psi_1$. This is true because, if $\partial_3 $ doesn't move immediately after, then $(\gamma'_1)'=\psi_1$ would separate $t$ and $u$, contradicting Remark \ref{remark_combinatorial_property}. 
In particular $\partial_3=\gamma_2$ must move between $\gamma_1$ and $\theta$. Moreover, the A-move $\gamma_2 \mapsto \psi_1$ occurs in a 4-holed sphere with boundaries corresponding two $\theta$, $\partial_1=f_2$ and two boundaries bounding one puncture each. If we switch the labels and redefine $\gamma_2$ to be $\gamma_3$, we get the situation of Subcase 2b(i). We can then obtain a contradiction. 

Therefore, we must have $\partial_1 = \theta$. 
Since $\gamma_1$ is disjoint from $\gamma_3$, using the notation in Figure \ref{fig_low5_case2b}(c), we can assume that $\gamma_1$ bounds $t$ and $s$. We obtain the sub path of $\lambda$, depicted in Figure \ref{fig_low5_case2bi_end}(c), given by the consecutive A-moves $\gamma_1 \mapsto \theta$, $\gamma_3 \mapsto \psi_3$. Observe that $\gamma_3\mapsto \psi_3$ occurs in a 4-holed sphere with boundaries corresponding to $s$, $\partial_3$, $f_1$ and $\theta$. 
In here, the conditions $\gamma_3\cap \gamma_1=\emptyset$ and $|\gamma_3\cap \psi_3|=2$ force $|\gamma_1 \cap \psi_3|=2$, contradicting Lemma \ref{lem_experimental_2}. 
%This local setup is the same as in the last paragraph of Subcase 2b(i), just replace $r$ and $\partial_3$ in Figure \ref{fig_low5_case2bi_end}(c) with $s$ and $t$ in Figure \ref{fig_low5_case2bi_end}, respectively. We can then mimic the argument in such paragraph and conclude that the bridge trisection is stabilized. 
Hence, Subcase 2b(ii) cannot occur. We have exhausted all the possibilities, thus concluding the proof of the Proposition. 
\end{proof}
%%%%%%%%%%%%%%%%%%%%%%%%%%%%%%%%%%%%%%%%%%%%%%%%%%%%%%

\begin{proposition}\label{prop_five_22}
Suppose that both $\gamma_1$ and $\psi_1$ bound two punctures each. Then any path $\lambda(ij)$ from $p_{ij}^{i}$ to $p_{ij}^{j}$ must be of distance at least five. 
\end{proposition}

\begin{proof}[Proof of Proposition \ref{prop_five_22}]
This proof follows the same path as Proposition \ref{prop_longer_proposition}. By Proposition \ref{lem_new_long_proposition}, it is enough to show the distance from $p^i_{ij}$ to $p^j_{ij}$ is not four. By way of contradiction, let $\lambda$ be a geodesic path of length four between such pants decompositions. 
By Lemmas \ref{lem_1} and \ref{lemma_psi_f_distinct}, each $\gamma_n$-curve must move at least once. We have two cases, depending on how many curves of $\{f_1, f_2\}$ are moved. 

\textbf{Case 1:} $\lambda$ moves one curve of $\{f_1, f_2\}$.
Without loss of generality, assume $f_2=h_2$ is fixed. Observe that, since $\psi_1$ and $\gamma_1$ bound two punctures and the curves $\psi_1$, $\gamma_1$, $h_1$, $f_1$ are compressing curves for the same tangle, we obtain that $\gamma'_1 \neq h_1, \psi_1$ and $\psi_1 \neq f'_1$. Thus, we can assume that $\gamma'_1 = \psi_2$ and $\gamma'_2 = \psi_1$. By Lemmas \ref{lem_gamma1_first} and \ref{lem_gamma1_last}, the A-moves $\gamma_1 \mapsto \psi_2$ and $\gamma_2 \mapsto \psi_1$ cannot be first nor last in $\lambda$.

\textbf{Subcase 1(a):} $\gamma_1 \mapsto \psi_2$ is second. In particular, $\gamma_2 \mapsto \psi_1$ is third, and there are at most three curves bounding an even number of punctures after the second A-move:  $\{f_1,h_1,f_2=h_2\}$. We focus our attention to the 4-holed sphere corresponding to $\gamma_1 \mapsto \psi_2$. By the previous sentence, we are forced to have an arrangement of curves as in Figure  %\ref{fig_low_22_case1}.
\ref{fig_low5_case1}(a) (compare with Figure \ref{fig_low5_22_case2}).
In particular, $\{x,\partial_2, \partial_4\}=\{f_1, h_1, f_2\}$ and $y=\gamma_2$. Since $\psi_1$ is the next curve to appear, $\psi_1$ must bound $\{r,s\}$. 
This is already a contradiction since Part 2 of Lemma \ref{lem_experimental_2} implies that $\psi_1$ bounds two of the three punctures $\{p,v,w\}$. 
% Now, if $x=f_a$ for some $a\in \{1,2\}$, then $f'_a$ will be forced to intersect $\psi_1$ in two points. This will yield the existence of shadow arcs $c$ and $c'$ with $\psi_1 = \partial \eta(c)$, $f'_a=\partial\eta(c')$, and $c\cap c' =\partial c \cap \partial c'=\{r\}$. Lemma \ref{criterion_destab} implies that the bridge trisection is stabilized. 
% We obtain the same contradiction if $\partial_2$ belongs to $\{h_1, h_2\}$. 
% Thus $x = h_1$, $\partial_2 = f_1$ and $\partial_4 = f_2=h_2$. 
% The last equality is also a contradiction since the dual curve $h'_2$ will be forced to bound an odd number of punctures. 
This subcase is impossible. 
% \begin{figure}[h]
% \centering
% \includegraphics[width=.35\textwidth]{fig_low_22_case1.PNG}
% \caption{ADD words} 
% \label{fig_low_22_case1}
% \end{figure}

\textbf{Subcase 1(b):} $\gamma_1 \mapsto \psi_2$ is third and $\gamma_2 \mapsto \psi_1$ is second in $\lambda$. Recall that the only curves bounding an even number of punctures are $\{\gamma_1, \psi_1, f_1, h_1, f_2=h_2\}$. We need to decide which of the A-moves $\gamma_3 \mapsto h_1$ and $f_1 \mapsto \psi_3$ is first. 
For us to decide, focus on the 4-holed sphere corresponding to the A-move $\gamma_1 \mapsto \psi_2$. Counting $\gamma_1$, there are four or five pairwise disjoint curves bounding an even number of punctures before $\gamma_1$ moved (See Figure \ref{fig_low5_22_case2}). 
But every A-move in $\lambda$ interchanges cut and compressing curves, so the number of even curves after the second A-move will be three or five. %But, the curves must look like in Figure \ref{fig_low5_22_case2}. % \ref{fig_low5_case1}, where $y$ could be replaced by a curve bounding $\{r,s\}$. 
Thus, $\gamma_3$ moves first, $f_1$ at last and the curves look like in Figure \ref{fig_low5_22_case2}(b). 
% As we saw in Subcase 1(a), $\partial_2$ cannot be equal to some $h_a$ because this would give that the bridge trisection is stabilized. 
% We also observe that $\partial_2$ cannot be equal to $f_1$, since this would force $f'_1=\psi_3$ to bound two punctures. 
% Thus $\partial_2$ must be equal to $\psi_1$.
Part 2 of Lemma \ref{lem_experimental_2} implies that $\partial_2 = \psi_1$. 
Since $\gamma_2\mapsto \psi_1$ occurs in second place, we can assume that $\gamma_2$ bounds $\{p,q,v\}$. 

We will focus on $\partial_4$. First observe that if $\partial_4=f_2=h_2$, then the A-moves in distinct sides of $\partial_4$ commute. This would let us to contradict Lemma \ref{lem_gamma1_last} since we could make $\gamma_2 \mapsto \psi_1$ the first A-move. 
Suppose now $\partial_4=f_1$. Since $f_1$ is the last curve to move, we can assume that $f'_1=\psi_3$ bounds $\{q,u,t\}$. Moreover, because $|\gamma_1\cap\psi_2|=|\partial_4 \cap \psi_3|=2$ and $\psi_3$ is disjoint from $x$, $z$, and $\psi_2$, we can see that $\gamma_1$ and $\psi_3$ must intersect in two points. Now, we know that $x= h_a$ for some $a\in \{1,2\}$. 
%This is a contradiction since 
We can use the dual curve $h'_a\in p_{jk}^j$ to find a tuple $(c, c')$ of destabilization shadows as in Lemma \ref{criterion_destab}. Thus, $\partial_4=h_1$ is the remaining option. 

If $\partial_4=h_1$, then we can assume that $\gamma_3$ bounds $\{r,s,w\}$ because $\gamma_3\mapsto h_1$ is the first A-move in $\lambda$. Recall that $\gamma_2$ bounds $\{p,q,v\}$. By thinking in the 4-holed sphere with boundaries $\gamma_3$, $\partial_2$, $z$ and $x$, the conditions $|\partial_4 \cap \gamma_3|= |\partial_2 \cap \gamma_2|=2$ and $\partial_4 \cap \partial_2 = \emptyset$ imply that $\gamma_2$ intersects $\partial_2=\psi_1$ in two points. Now, we know that $z=f_a$ for some $a\in \{1,2\}$. We can use the dual curve $f'_a\in p_{ik}^i$ to find a pair of shadows $(c,c')$ as in Lemma \ref{criterion_destab}. We have concluded Case 1. 

\textbf{Case 2:} $\lambda$ fixes $\{f_1,f_2\}$. 

In this case, one $\gamma_n$-curve moves twice and the rest exactly once. We write $f_a=h_a$ and denote by $\theta$ the pivotal curve. There are two subcases depending on how many times $\gamma_1$ moves.
\begin{figure}[h]
\centering
\labellist \small\hair 2pt
\pinlabel {(a)} at -15 140
\pinlabel {$q$} at 30 40
\pinlabel {$p$} at 20 120
\pinlabel {$\gamma_1$} at -1 120
\pinlabel {$\gamma'_1 = \psi_2$} at 235 10
\pinlabel {$\partial_4$} at 215 115
\pinlabel {$y$} at 185 125
\pinlabel {$\partial_2$} at 155 140
\pinlabel {$x$} at 120 70
\pinlabel {$r$} at 145 60
\pinlabel {$t$} at 100 60
\pinlabel {$u$} at 75 60
\pinlabel {$s$} at 180 60
\pinlabel {$v$} at 105 120
\pinlabel {$w$} at 130 120

\pinlabel {(b)} at 280 140
\pinlabel {$q$} at 315 40
\pinlabel {$p$} at 315 120
\pinlabel {$x$} at 400 75
\pinlabel {$z$} at 450 77
\pinlabel {$r$} at 435 60
\pinlabel {$t$} at 395 60
\pinlabel {$u$} at 370 60
\pinlabel {$s$} at 460 60
\pinlabel {$v$} at 400 120
\pinlabel {$w$} at 425 120
\pinlabel {$\partial_4$} at 510 120
\pinlabel {$\gamma'_1 = \psi_2$} at 530 10
\pinlabel {$\partial_2$} at 460 120
\pinlabel {$\gamma_1$} at 295 120

\endlabellist
\includegraphics[width=.75\textwidth]{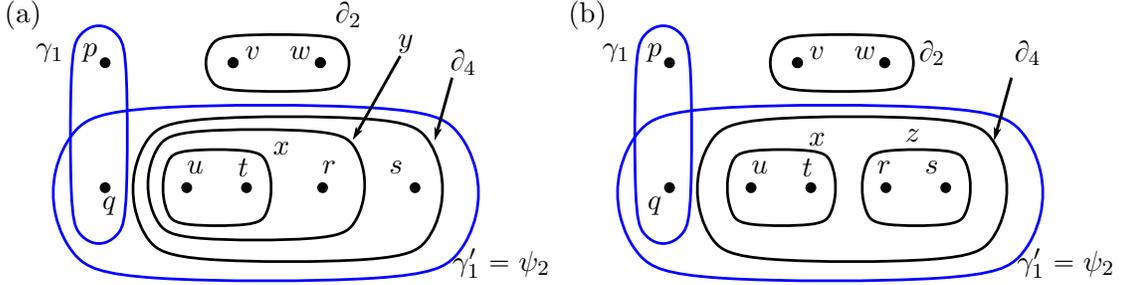}
\caption{When $\gamma_1$ and $\psi_2$ differ by one A-move, there are either (a) three or (b) four curves disjoint from $\gamma_1$ bounding an even number of punctures.} 
\label{fig_low5_22_case2}
\end{figure}

\textbf{Subcase 2a:} $\gamma_1$ moves once along $\lambda$. Recall that $\gamma_1$, $\psi_1$, $f_1=h_1$, and $f_2=h_2$ bound compressing disks in $T_{ij}$ and $\gamma_1$, $\psi_1$ bound two punctures. Thus, $|\gamma_n\cap \alpha|$ and $|\psi_1 \cap \alpha|$ are both divisible by four for all $\alpha \in \{\gamma_1, \psi_1, f_1=h_1, f_2=h_2\}$. This implies that $\gamma'_1$ must bound a cut disk, say $\gamma'_1=\psi_2$. Lemmas \ref{lem_gamma1_first} and \ref{lem_gamma1_last} force $\gamma_1$ to move second or third in $\lambda$. 
We can represent the curves in 4-holed sphere corresponding to $\gamma_1 \mapsto \psi_2$ like in Figure \ref{fig_low5_22_case2}. Observe that, before the A-move of $\gamma_1$, there are either four or five pairwise disjoint curves bounding an even number of punctures. 

We first study $\partial_2$ in Figure \ref{fig_low5_22_case2}. Since $\partial_2$ bounds two punctures, we have $\partial_2 \in \{f_1=h_1, f_2=h_2, \psi_1, \theta\}$. 
%As we did in Subcase 1(a), $\partial_2$ cannot be equal to some $h_a$ since then $\mathcal{T}$ would be stabilized. 
Notice that $\partial_2$ cannot be $\theta$. If that were the case, $\theta'$ would be forced to bound an even number of punctures, say $\{p,v\}$, and $\theta'=\psi_1$. In particular, $\psi_1$ separates $p$ and $q$ which contradicts Remark \ref{remark_combinatorial_property}. Lemma \ref{lem_experimental_2} implies that $\psi_1$ bounds two punctures from $\{p,v,w\}$, thus $\partial_2 = \psi_1$.

\textbf{Subcase 2a(i):} Suppose first that there are five even curves. We use the notation in Figure \ref{fig_low5_22_case2}(b). We have that the sets of curves $\{x,z,\partial_4\}$ and $\{\theta, f_1, f_2\}$ agree. In particular, by Lemma \ref{lem_gamma1_last} $\gamma_2\mapsto\psi_1$ must be the second A-move and so $\gamma_3 \mapsto \theta$ is the first one. 
%Observe that $\partial_2$ cannot be equal to some $h_a$. If so, we can conclude that $h'_a\in p^j_{jk}$ intersects $\gamma_1$ in two points, then find a tuple of shadow arcs $(c,c')$ as in Lemma \ref{criterion_destab} and conclude that $\mathcal{T}$ is stabilized. Thus $\partial_2 \in \{\psi_1, \theta\}$. 
If $\partial_4$ is equal to some $f_a$, then the curves $\theta$ and $\psi_1$ will lie in different sides of $\partial_4$. We could then permute their corresponding A-moves and obtain $\gamma_2 \mapsto\psi_1$ first in $\lambda$, contradicting Lemma \ref{lem_gamma1_last}. Thus we conclude that %$\partial_2=\psi_1$, 
$\partial_4=\theta$, $x=f_1=h_1$, and $z=f_2=h_2$. Here, we can assume that $\gamma_2$ bounds $\{p,q,v\}$ and $\gamma_3$ bounds $\{w,r,s\}$. 
Now, by looking at the 4-holed sphere bounded by $\gamma_2$, $x$, $z$ and $\partial \eta(w)$, we can see that $\gamma_3 \cap \gamma_2 = \emptyset$ and $|\gamma_2 \cap \psi_1|=2$ imply that $|\gamma_3\cap \psi_1|=2$. Then, inside the component of $\Sigma\setminus \gamma_3$ containing $w$, we can use $f'_2$ to find a tuple of shadows $(c,c')$ satisfying the conditions of Lemma \ref{criterion_destab}. Thus, this subcase cannot occur. 

\textbf{Subcase 2a(ii):} Before the A-move $\gamma_1 \mapsto \psi_2$, there are four curves bounding even number of punctures. We can draw the curves in $\Sigma$ as in Figure \ref{fig_low5_22_case2}(a). %Here, $\{x,\partial_2, \partial_4\}\subset \{f_1, f_2, \theta, \psi_1\}$ and only one of $\{\theta, \psi_1\}$ is in $\{x,\partial_2, \partial_4\}$. We now study $\partial_2$. 
%As we did in Subcase 1(a), $\partial_2$ cannot be equal to some $h_a$ since then $\mathcal{T}$ would be stabilized. Thus, $\partial_2 \in \{\theta, \psi_1\}$ and 
Since $\partial_2 = \psi_1$, we can assume $x=f_1=h_1$ and $\partial_4 = f_2= h_2$. 
%Notice that $\partial_2$ cannot be $\theta$. If that were the case, $\theta'$ would be forced to bound an even number of punctures, say $\{p,v\}$, and $\theta'=\psi_1$. In particular, $\psi_1$ separates $p$ and $q$ which contradicts Remark \ref{remark_combinatorial_property}. Hence, $\partial_2 = \psi_1$.
Now, since $\partial_4$ is fixed along $\lambda$, the A-moves occurring in different sides of $\partial_4$ can be permuted. Thus, we can assume that $y=\gamma_2$ and so $\gamma'_2 \in \{\psi_3, \theta\}$.

Suppose now that $\gamma'_2=\psi_3$. Since $\psi_3=\gamma'_2$ is forced to bound $\{u,t,s\}$, we can assume that $h'_1\in p^j_{jk}$ bounds $\{t,s\}$. In particular, $T_{jk}$ connects the punctures $\{t,s\}$. On the other hand, since $\gamma_1$ $f_1$, and $f_2$ bound disks in $T_{ij}$, we know that $T_{ij}$ connects $p$, $u$ and $r$ with $q$, $t$, and $s$, respectively. The fact that $L_j=T_{ij}\cup \T_{jk}$ is a 2-component link and $\psi_1$ is a reducing curve implies that $T_{jk}$ connects the punctures $\{u,r\}$ with $\{p,q\}$. 
Since $\gamma_2$ bounds a cut-disk for $T_{ik}$, we have that $T_{ik}$ must connect $r$ with either $u$ or $t$. In any case, the fact that $L_k=T_{ik}\cup \T_{jk}$ is a 2-component link forces $v$ and $w$ to be connected by $T_{ik}$. 
Since $\psi_1$ bounds a compressing disk in both $T_{ij}$ and $T_{jk}$, we obtain that $v$ and $w$ are connected by the three tangles. This implies the surface $S$ is disconnected, a contradiction.

We are left with $\gamma'_2=\theta$ which forces $\gamma'_3=\psi_1=\partial_2$ and $\theta'=\psi_3$. 
Since $\partial_4=f_2=h_2$ is fixed along $\lambda$, the A-moves on distinct sides of $\partial_4$ commute. Thus, we can take $\lambda$ so that $\gamma_3 \mapsto \psi_1$ is the first A-move. This contradicts the conclusion of Lemma \ref{lem_gamma1_last}. Hence, this subcase cannot occur. 

\textbf{Subcase 2b:} $\gamma_1$ moves twice along $\lambda$. 
By symmetry and Subcase 2a, it is enough to study the case that $\theta'=\psi_1$. We write $\gamma'_2=\psi_2$ and $\gamma'_3=\psi_3$. First observe that, since $\gamma_1$ and $\psi_1$ bound disjoint sets of two punctures (Lemma \ref{lem_combinatorics}), the A-moves $\gamma_1\mapsto \theta$ and $\theta\mapsto \psi_1$ cannot be consecutive in $\lambda$. In other words, at least one cut-curve must move between those moves. We are left with two options (up to symmetry) for the order of the A-moves along $\lambda$: $(\gamma_1, \gamma_3, \gamma_2, \theta)$ and $(\gamma_1, \gamma_3, \theta, \gamma_2)$. 
We focus on the \textbf{second} A-move $\gamma_3 \mapsto \psi_3$. It occurs inside a 4-holed sphere depicted in Figure \ref{fig_low5_case2b}(a). %We have two cases, depicted in Figure \ref{fig_low5_case2b}(b-c). 

\textbf{Subcase 2b(i):} Both $\partial_2$ and $\partial_3$ bound one puncture each. We use the notation in Figure \ref{fig_low5_case2b}(b) and observe that $y=\gamma_2$. 
Since $\gamma_1 \mapsto \theta$ and $\gamma_3 \mapsto \psi_3$ are the first two A-moves in $\lambda$, we know that the sets of curves $\{x,\partial_1, \partial_4\}$ and $\{\theta, f_1=h_1, f_2=h_2\}$ agree. 
%We study which of $\{x,\partial_1, \partial_4\}$ is equal to $\theta$. 
Suppose $\partial_4 = \theta$, then $\gamma_1$ is forced to bound $\{r,s\}$. 
In the 4-holed sphere with boundaries $\partial_1$, $y$, $\partial\eta(r)$ and $\partial \eta(t)$, the conditions $\gamma_3 \cap \gamma_1=\emptyset$ and $|\gamma_3\cap \psi_3|=2$ force $|\gamma_1\cap \psi_3|=2$. Lemma \ref{lem_experimental_2} implies that $\psi_1$ bounds two punctures from $\{q,p,r\}$. This is impossible since $\partial_1 \in \{h_1, h_2\}$ is disjoint from $\psi_1$. 
%This is the exact setup of Figure \ref{fig_low5_case2bi_end}(a). We then proceed as in the last paragraph of Case 2b(i) in the proof or Proposition \ref{prop_longer_proposition} to conclude that $\mathcal{T}$ is stabilized. 
Thus we conclude that $\partial_4= f_2=h_2$. %Since $\gamma_2$ and $\partial_1$ lie in distinct sides of $\partial_4$, we know that $\partial_1 \neq \gamma_1$ since this would give us that the A-moves $\gamma_1\mapsto \theta$ and $\theta\mapsto \psi_1$ can be made consecutive which contradicts the previous paragraph.  

Suppose now that $\partial_1 = f_1 = h_1$. 
Since the A-move $\gamma_1 \mapsto \theta$ occurs inside $\gamma_2$, we can reuse Figure \ref{fig_low5_case2b}(b) and assume that $x=\gamma_1$ and $\theta=\gamma'_1$ bounds $\{u,v\}$. 
After $\gamma'_1\mapsto \theta$, the next A-move has to be $\gamma_2\mapsto \psi_2$. Here, $\psi_2$ and $\psi_1=\theta'$ will bound $\{s,u,v\}$ and $\{s,u\}$, respectively.
Focus on the 4-holed sphere $E$ corresponding to the A-move $\theta\mapsto \psi_1$. Notice that $E$ has boundaries corresponding to $s$, $u$, $v$ and $\psi_2$. Since $|\gamma_2 \cap \psi_2|=2$, the intersection $\gamma_2\cap E$ is an arc with both endpoints on $\psi_2$ that separates $s$ from $\{u,v\}$ (see Figure \ref{fig_low_22_case2bi}(a)). Since $\theta\cap \gamma_2 =\emptyset$, the condition $|\psi_1 \cap \theta|=2$ forces $\psi_1$ to intersect $\gamma_2$ in two points. 
\begin{figure}[h]
\centering
\labellist \small\hair 2pt
\pinlabel {(a)} at -20 200
\pinlabel {$s$} at 53 100
\pinlabel {$u$} at 105 100
\pinlabel {$v$} at 155 100
\pinlabel {$w$} at 207 100
\pinlabel {$\theta$} at 145 130
\pinlabel {$\psi_1$} at 55 185
\pinlabel {$\partial_4=f_1=h_1$} at 180 170
\pinlabel {$\psi_2$} at 200 35
\pinlabel {$\gamma_2$} at 210 140

\pinlabel {(b)} at 270 200
\pinlabel {$q$} at 355 150
\pinlabel {$p$} at 415 150
\pinlabel {$t$} at 350 95
\pinlabel {$s$} at 415 50
\pinlabel {$u$} at 490 50
\pinlabel {$r$} at 515 150
\pinlabel {$v$} at 535 50
\pinlabel {$w$} at 580 50
\pinlabel {$f'_2$} at 305 50
\pinlabel {$c'$} at 495 130
\pinlabel {$\partial_1=f_1$} at 390 180
\pinlabel {$\gamma_3$} at 350 50
\pinlabel {$f'_1$} at 380 120
\pinlabel {$c$} at 390 5
\pinlabel {$\partial_4=f_2$} at 600 110
\pinlabel {$\gamma_2$} at 580 85
\pinlabel {$\psi_1$} at 520 70

\endlabellist
\includegraphics[width=.85\textwidth]{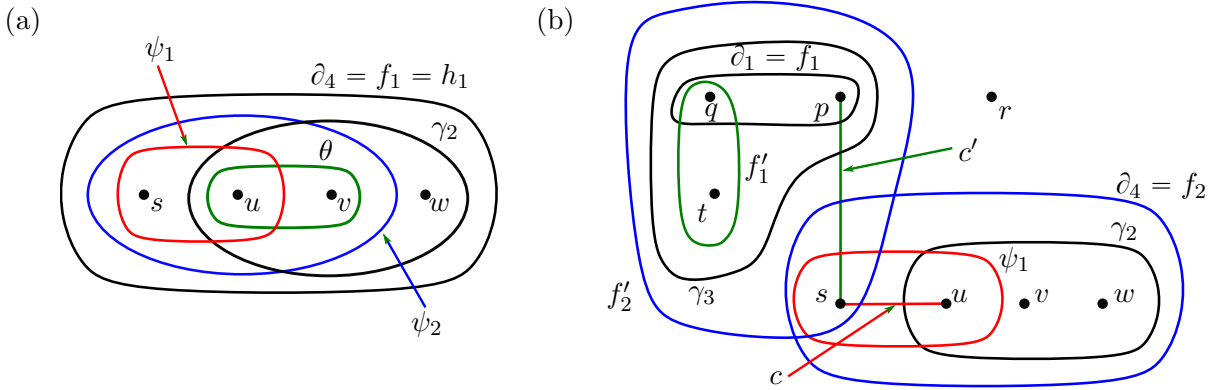}
\caption{A close-up to some curves in Subcase 2b(i).} 
\label{fig_low_22_case2bi}
\end{figure}

To end, we study the curve $f'_2$. For reference, we use the curves and notation from Figure \ref{fig_low_22_case2bi}(b). We now look at the 4-holes sphere $E'$ with boundaries $\gamma_3$, $\gamma_2$, $\partial \eta(r)$, and $\partial \eta(t)$. Since $|\psi_1 \cap \gamma_2|=2$, $\psi_1\cap E'$ is an arc with both endpoints on $\gamma_2$ that separates $s$ from $r$ and $\gamma_3$. Thus, the conditions $\gamma_2\cap f_2=\emptyset$ and $|f'_2\cap f_2|=2$ imply that $\psi_1$ intersects $f'_2$ in two points. 
If $f'_2$ bounds two punctures, we can use the condition $|\psi_1\cap f'_2|=2$ to find a tuple $(c,c')$ of shadows satisfying the condition of Lemma \ref{criterion_destab}, contradicting the fact that $\mc{T}$ is not stabilized.

On the other hand, if $f'_2$ bounds four punctures, we will also find a tuple $(c,c')$ as in Lemma \ref{criterion_destab}. The rest of this paragraph explains how to do this. First observe that $f'_2$ will bound $\gamma_3$ and $s$. Since $f'_1$ is lies inside $\gamma_3$ and intersects $f_1$ in two points, we can assume that $f'_1$ bounds $\{q,t\}$. Both $f'_1$ and $f'_2$ bound compressing disks in $T_{ik}$ so we can find a shadow $c'$ of an arc of $T_{ik}$ connecting $\{p,s\}$ such that $c'$ is disjoint from $f'_1$ and $f'_2$. Inside the disk component of $\Sigma \setminus f'_2$ that contains $\gamma_3$, the condition $|\psi_1\cap f'_2|=2$ implies that $\psi_1$ is an arc with both endpoints in $f'_2$ that separates $s$ from $f'_1$ and $p$. We can slide $c'$ over $f'_1$ and $f'_2$ and assume that $|c'\cap \psi_1|=1$. The last condition allows us to pick an arc $c$ in $\Sigma$ connecting $\{s,u\}$ such that $\partial \eta (c)=\psi_1$ and $c\cap c'=\partial c \cap \partial c'=\{s\}$. Notice that $c'$ is a shadow for arcs in $T_{ij}$ and $T_{jk}$. Hence, the tuple $(c,c')$ satisfies the conditions of Lemma \ref{criterion_destab}. This is a contradiction. 

We are left with $x=f_1 = h_1$ and $\theta=\partial_1$. Since $\partial_4 = f_2=h_2$ is fixed along $\lambda$, A-moves on distinct sides of $\partial_4$ commute. Moreover, this setup is equivalent to the previous case ($\partial_1 = f_1=h_1$): one can reflect Figure \ref{fig_low5_case2b}(b) with respect to $\partial_4$ and the roles of the curves on each side will reverse. Therefore, this case is impossible. 

\textbf{Subcase 2b(ii):} $\partial_2$ and $\partial_3$ enclose one and three punctures, respectively. We use the notation of Figure \ref{fig_low5_case2b}(c). One of the curves $\{\partial_1, x, \partial_4\}$ is equal to $\theta$. 
%If $\partial_4$ participates in future A-moves, $\partial'_4$ will bound two of the three punctures inside of $\psi_3$. Thus $\partial_4 \neq \gamma_1$ since the previous sentence will force $\psi_1$ to separate the two punctures bounded by $\partial_4$, contradicting Remark \ref{remark_combinatorial_property}. 
Observe that, if $\rho$ is a curve such that $\rho\mapsto \partial_4$ is an A-move immediately before $\partial_3 \mapsto\psi_3$, then $\rho$ bounds four punctures. In particular, $\rho\neq \gamma_1$. Thus $\partial_4 \neq \theta$ and so $\partial_4 = f_1 = h_1$. 
Suppose now that $x=\theta$. We can assume that $\gamma_1$ bounds $\{r,p\}$. By Lemma \ref{lem_combinatorics}, the two punctures bounded by $\psi_1$ must be distinct than $\{r,p\}$. Here, notice that $\gamma'_2=\psi_2$ is forced to bound $\{t,u,r\}$ and $\theta=x$ must move after $\gamma_2$. Moreover, $\theta'$ has to bound four punctures, contradicting $\theta'=\psi_1$. 
%On the other hand, suppose that $x=\gamma_1$. By the previous discussion ($x=\theta$), $\gamma_2$ has to move after $\gamma_3$. Since $\psi_2$ bounds a cut-disk, $\gamma'_2 = \psi_2$ bounds $\{t,u,r\}$. The last two moves in $\lambda$ correspond to $\gamma_1$ and $\theta$, and they occur inside the 4-holed sphere with boundaries $\psi_2$, $\psi_3$, $\partial\eta(p)$, and $\partial \eta(q)$. Thus, $\psi_1$ will bound $\{p,q\}$, contradicting Lemma \ref{lem_combinatorics}. 
Hence $x=f_2=h_2$ and $\partial_1 = \theta$.

We are left to discard the case $\partial_1 =\theta$. Since $x$ is fixed along $\lambda$ and $\psi_3$ won't move, we see that two out of the three punctures $\{t,u,r\}$ will be bounded by $\psi_1$. %Thus $\partial_1$ must be $\theta$ and $\gamma_1$ moves before $\gamma_3 \mapsto \psi_3$.
We can assume that $\gamma_1$ bounds $\{t,s\}$. By looking at the 4-holed sphere with boundaries $\gamma_3$, $\partial\eta(t)$, $\partial\eta(s)$ and $\partial\eta(u)$, we see that the conditions $\psi_3 \cap \theta=\emptyset$, $|\gamma_3\cap \psi_3|=2$ and $|\theta\cap \gamma_1|=2$ imply $|\gamma_1 \cap \psi_3|=2$. Now, inside the disk of $\Sigma\setminus \psi_3$ containing $\partial_4 = h_1$, one can see that $h'_1\in p^j_{jk}$ must intersect $\gamma_1$ in two points. Thus, there is a shadow $c'$ for an arc of $T_{jk}$ with $\partial \eta(c')=h'_1$ and $|c'\cap \gamma_1|=1$. By taking $c\subset \Sigma$ with $\partial \eta (c) = \gamma_1$, $c\cap c' = \partial c \cap \partial c'=\{s\}$, we obtain a tuple $(c, c')$ like in Lemma \ref{criterion_destab}. Hence, $\T$ is an stabilization. This finishes the analysis in Case 2. 
\end{proof}

%%%%%%%%%%%%%%%%%%%%%%%%%%%%%%%%%%%%%%%%%%%%%%%%%%%%%%%%%%%%%%%%%%%%
\begin{proposition}\label{prop_five_44}
Suppose that both $\gamma_1$ and $\psi_1$ bound four punctures each. Then any path $\lambda(ij)$ from $p^i_{ij}$ to $p^j_{ij}$ must have length at least five. 
\end{proposition}

\begin{proof}[Proof of Proposition \ref{prop_five_44}]
By Proposition \ref{lem_new_long_proposition}, it is enough to show the distance from $p^i_{ij}$ to $p^j_{ij}$ is not four. By way of contradiction, let $\lambda$ be a geodesic path of length four between such pants decompositions. By Lemmas \ref{lem_1} and \ref{lemma_psi_f_distinct}, each $\gamma_n$-curve must move at least once. 

Notice that if two pants decompositions differ by the A-move $\gamma_1 \mapsto \psi_1$, then each boundary loop of the 4-holed sphere corresponding to this A-move must bound two punctures. This is true because the curves $\gamma_1$ and $\psi_1$ bound compressing disks for the same tangle $T_{ij}$. In particular, %if $\lambda$ has length less than four, 
we know that there are at most five curves bounding an even number of punctures that are involved in $\lambda$, say $\{\gamma_1, \psi_1, h_1, f_1, f_2=h_2\}$ or $\{\gamma_1, \psi_1, \theta, f_1=h_1, f_2=h_2\}$, where $\theta$ is the pivotal curve. Thus it cannot contain the edge $\gamma_1 \mapsto \psi_1$. 

\textbf{Case 1:} $\lambda$ moves one curve of $\{f_1, f_2\}$. Say $f_2=h_2$ is fixed. Notice that $f_1$ bounds two punctures since $\gamma_1$ bounds four. Also, $f_1$ and $\psi_1$ bound compressing disks for the same tangle $T_{ij}$, so the two punctures bounded by $f_1$ must be on the same side of $\psi_1$. Thus, $|f_1\cap \psi_1|$ is divisible by four. This implies that $f'_1 \neq \psi_1$. Similarly, $\gamma'_1\neq h_1$. We can then assume that $\gamma_2 \mapsto \psi_1$ and $\gamma_1\mapsto \psi_2$ are A-moves along $\lambda$. Moreover, by Lemmas \ref{lem_gamma1_first} and \ref{lem_gamma1_last} such A-moves must be in either second or third place. But $\gamma_1 \cap \psi_1\neq \emptyset$ so $\gamma_1\mapsto \psi_2$ must be second and $\gamma_2 \mapsto \psi_1$ is third. 

We now study the 4-holed sphere where the A-move $\gamma_1\mapsto \psi_2$ occurs. We can assume that the curves look like in Figure \ref{fig_low_44}(a). In particular $\partial_1=\gamma_2$ and the sets of curves $\{x, \partial_3, \partial_4\}$ and $\{f_1, h_1, f_2=h_2\}$ agree. 
Since the next A-move is $\gamma_2 \mapsto \psi_1$ we obtain that $\psi_1$ bounds $x$ and $\partial_3$. From Figure \ref{fig_gamma1_first} we know that the reducing curve $\gamma_1$ (resp. $\psi_1$) must separate $f_1$ and $f_2$ (resp. $h_1$ and $h_2$). This implies that $x=f_1$, $\partial_3 = f_2=h_2$ and $\partial_4 = h_1$. 
\begin{figure}[h]
\centering
\labellist \small\hair 2pt
\pinlabel {(a)} at -20 130
\pinlabel {$\gamma_1$} at -5 70
\pinlabel {$\psi_2$} at 220 65
\pinlabel {$\gamma_2$} at 30 73
\pinlabel {$x$} at 75 105
\pinlabel {$\partial_3$} at 145 115

\pinlabel {(b)} at 240 130
\pinlabel {$\theta$} at 255 115
\pinlabel {$\gamma_1$} at 420 25
\pinlabel {$\partial_4$} at 300 25
\pinlabel {$\partial_3$} at 300 120

\pinlabel {(c)} at 440 130
\pinlabel {$\theta$} at 570 125
\pinlabel {$\partial_1$} at 550 100
\pinlabel {$x$} at 515 87
\pinlabel {$\gamma_1$} at 660 65
\pinlabel {$\partial_4$} at 565 15
\pinlabel {$z$} at 600 22

\tiny\pinlabel {$f=h$} at 340 20
\tiny\pinlabel {$f=h$} at 340 120

\endlabellist
\includegraphics[width=.9\textwidth]{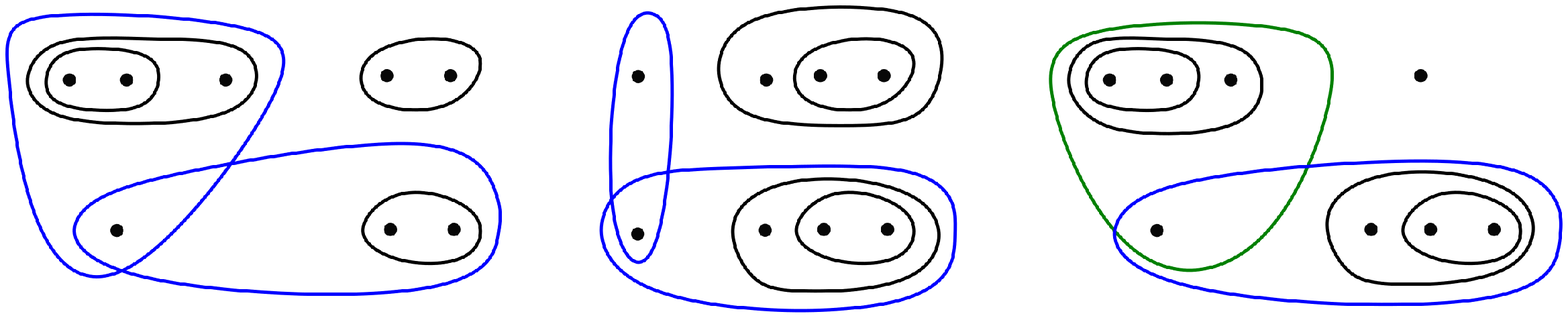}
\caption{Curve arrangements for specific A-moves.} 
\label{fig_low_44}
\end{figure}

To end this case, we will analyze the possible shadows of the tangles $T_{ij}$, $T_{ik}$ and $T_{jk}$. Figure \ref{fig_low_44_ending}(a) contains the labels of the punctures and the new shadows described throughout this paragraph. 
Notice that $h'_1$ bounds two punctures, say $\{s,t\}$. By looking at the 4-holed sphere with boundaries $\psi_2$, $s$, $t$ and $u$, we can conclude that $h'_1$ must intersect $\gamma_1$ in two points. In particular, there is a shadow $c$ of an arc in $T_{jk}$ connecting $\{s,t\}$ such that $\partial \eta (c)=h'_1$. Since $|h'_1\cap \gamma_1|=2$, we see that $c$ intersects $\gamma_1$ once. 
Now focus in the disk component of $\Sigma \setminus \gamma_1$ containing $\gamma_2$. Since $f_1$ and $\gamma_1$ bound compressing disks for $T_{ij}$, there are shadows $a_1$, $a_2$ for arcs of $T_{ij}$ that are disjoint from $f_1\cup \gamma_1$ satisfying $\partial\eta (a_1)=f_1$ and $a_2$ connects $\{r,s\}$. Notice that $f_1$ and $h'_1$ are in opposite sides of $\gamma_2$, so $a_1 \cap c =\emptyset$. Moreover, we can think of $a_2$ as an arc in a 4-holed sphere with boundaries $x=f_1$, $\gamma_1$, $\partial \eta(s)$ and $\partial \eta(r)$, where $a_2$ and $c$ are arcs connecting $\{r,s\}$ and $\{s, \gamma_1\}$, respectively. We can slide $a_2$ over $f_1$ and $\gamma_1$ and still obtain a shadow arc for $T_{ij}$. Thus, we can slide $a_2$ inside this 4-holed sphere and choose $a_2$ to have interior disjoint from $c$; i.e., $a_2 \cap c = \partial a_2 \cap \partial c=\{s\}$. 
To end, we observe that $f'_1$ bounds two punctures and is inside $\gamma_2$. We can assume that $f'_1$ bounds $
\{q, r\}$. Since $f'_1$ and $\gamma_1$ bound compressing disks for $T_{ik}$, we can find shadows $b_1$, $b_2$ for arcs in $T_{ik}$ disjoint from $f'_1$ and $\gamma_1$ satisfying $\partial \eta (b_1) = f'_1$ and $b_2$ connects $\{p,s\}$. Since $|f_1 \cap f'_1|=2$, we can choose $b_1$ so that $b_1 \cap a_1=\partial b_1 \cap \partial a_1=\{q\}$. As we did with $a_2$, we can slide $b_2$ over $f'_1$ and $\gamma_1$ until $b_2$ has interior disjoint from $c$. 
We can further slide $a_2$ and $b_2$ and see that $a_1\cup b_1 \cup a_2 \cup b_2$ can be chosen to be a simple closed curve (ignoring the punctures). %First see that $a_2$ (resp. $b_2$) can be slid over $f_1$ (resp. $f'_1$) until it has interior disjoint from $b_1$ (resp. $a_1$). Then we slide both $a_2$ and $b_2$ over $\gamma_1$ until they have interior disjoint from $c$. This forces $a_2$ and $b_2$ to have empty interior; the circling around the puncture $\{s\}$ can be undone. Thus $a_1\cup b_1 \cup a_2 \cup b_2$ is a simple clodes curve. 
The tuple $(\alpha, \beta, \gamma)=(\{a_1, a_2\}, \{b_1, b_2\}, c)$ satisfies the conditions of Lemma \ref{gen_criterion_destab}, concluding that $\mc{T}$ is an stabilization. 

\textbf{Case 2:} $\lambda$ fixes $\{f_1, f_2\}$. 
Suppose first that $\gamma'_1 = \psi_2$. From Figure \ref{fig_low_44}(a), we note that before the A-move $\gamma_1 \mapsto \psi_2$ there are four curves bounding even number of punctures say $\{\gamma_1, x,\partial_3, \partial_4\}$. Since $\gamma_1\cap \psi_1 \neq \emptyset$, $f_1=h_1$, and $f_2 = h_2$, the mentioned A-move is impossible. Thus $\gamma'_1 \neq \psi_2, \psi_3$. Similarly, we see that $\psi_1 \neq \gamma'_2, \gamma'_3$. We have already established that $\gamma'_1$ cannot be equal to $\psi_1$. Thus, the only option is $\gamma'_1 = \theta$ and $\theta'=\psi_1$. In particular $\gamma'_2 = \psi_2$ and $\gamma'_3 = \psi_3$. 

We now study how many punctures $\theta$ bounds. First note that $\theta$ cannot bound three punctures. This holds because, before the A-move $\gamma_1\mapsto \theta$, there would be three other curves bounding an even number of punctures (set $\psi_2=\theta$ in Figure \ref{fig_low_44}(a)). This is impossible since only five curves can bound even number of punctures $\{\gamma_1, \psi_1, \theta, f_1=h_1, f_2=h_2\}$, and $\psi_1$ and $\theta$ intersect $\gamma_1$. If $\theta$ bounds two punctures, the curves in $\Sigma$ will look as in Figure \ref{fig_low_44}(b). If $\theta$ moves immediately after $\gamma_1$, then three out of the four punctures bounded by $\gamma_1$ will be in the same side of $\psi_1=\theta$, contradicting Remark \ref{remark_combinatorial_property}. If cut-curve moves before $\theta$, we can assume is $\gamma_2=\partial_3$. Since $\gamma'_2$ bounds a cut disk, it is forced to bound $\theta$ together with one other puncture. This implies that $\theta'=\psi_1$ will bound two punctures, a contradiction. 
\begin{figure}[h]
\centering
\labellist \small\hair 2pt 
\pinlabel {(a)} at -15 180 
\pinlabel {$p$} at 30 130
\pinlabel {$q$} at 85 130
\pinlabel {$r$} at 140 130
\pinlabel {$s$} at 80 25
\pinlabel {$u$} at 230 35
\pinlabel {$t$} at 205 35
\pinlabel {$a_1$} at 55 113
\pinlabel {$b_1$} at 120 113
\pinlabel {$f_1$} at 55 148
\pinlabel {$f'_1$} at 120 148
\pinlabel {$\gamma_2$} at 20 162
\pinlabel {$\gamma_1$} at 25 50
\pinlabel {$\psi_2$} at 270 60
\pinlabel {$\partial_3$} at 255 133
\pinlabel {$h_1$} at 180 50
\pinlabel {$c$} at 140 25
\pinlabel {$b_2$} at 67 75
\pinlabel {$a_2$} at 105 75
 
\pinlabel {(b)} at 300 180 
\pinlabel {$p$} at 360 130 
\pinlabel {$q$} at 390 130 
\pinlabel {$r$} at 445 130 
\pinlabel {$s$} at 500 130 
\pinlabel {$t$} at 380 35 
\pinlabel {$u$} at 475 33 
\pinlabel {$v$} at 530 33 
\pinlabel {$w$} at 590 33 
\pinlabel {$a_1$} at 565 33
\pinlabel {$b_1$} at 500 33
\pinlabel {$f_2$} at 600 60
\pinlabel {$f'_2$} at 480 15
\pinlabel {$a_2$} at 450 33
\pinlabel {$b_2$} at 500 73
\pinlabel {$\gamma_1$} at 610 80
\pinlabel {$\theta$} at 480 150
\pinlabel {$\gamma_2$} at 450 160
\pinlabel {$\gamma_2'$} at 330 80
\pinlabel {$h_1$} at 372 154 
\pinlabel {$c$} at 377 80

\endlabellist
\includegraphics[width=.9\textwidth]{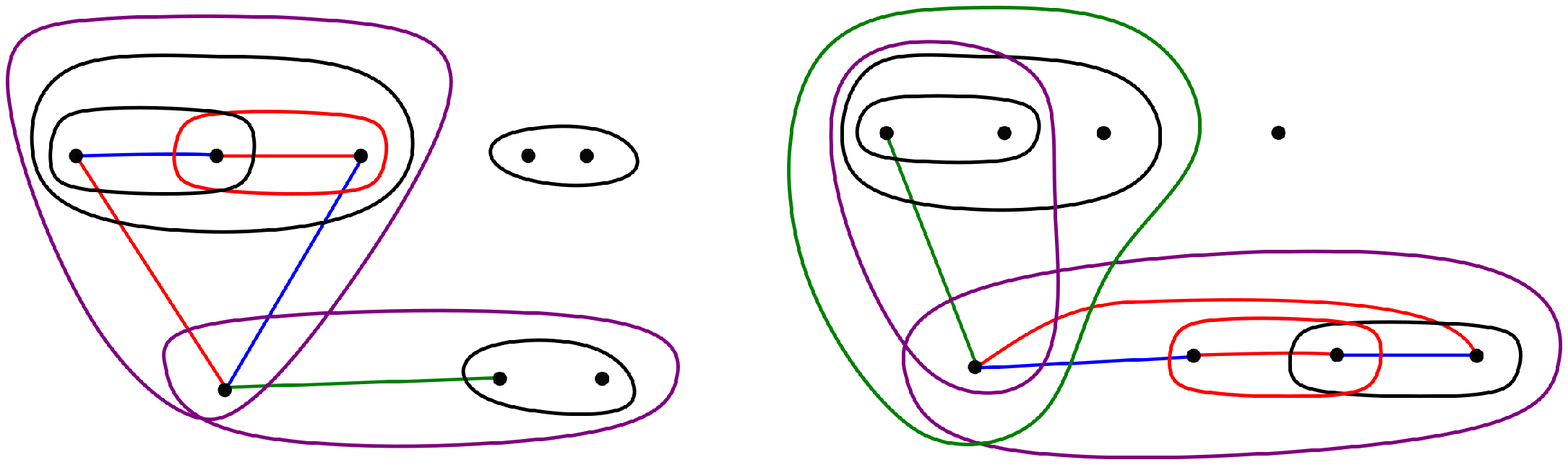}
\caption{Shadows.} 
\label{fig_low_44_ending}
\end{figure}

The only remaining option is if $\theta$ bounds four punctures. Since only $\{f_1=h_1,f_2=h_2\}$ are curves disjoint from $\gamma_1$ that bound an even number of punctures, we can draw the curves in $\Sigma$ before the A-move $\gamma_1 \mapsto \theta$ as in Figure \ref{fig_low_44}(c). Moreover, we can assume $x=f_1=h_1$ and  $z=f_2=h_2$. 
Recall that $f_1=h_1$ and $f_2=h_2$ lie in different sides of both $\gamma_1$ and $\psi_1$ (see Figure \ref{fig_gamma1_first}). Thus, by Remark \ref{remark_combinatorial_property}, since $\gamma_1$ bounds $t$, $u$ and $f_2=h_2$, we conclude that $\psi_1$ bounds $r$, $s$ and $f_2=h_2$. But $\partial_1$ bounds $h_1$ and $r$ which are on distinct sides of $\psi_1$. 
Thus $\partial_1 \not \in \{ \psi_2, \psi_3\}$. Similarly, $\partial_4 \not \in \{\psi_2, \psi_3\}$. We can then assume that $\partial_1 = \gamma_2$ and $\partial_4 = \gamma_3$. 
Since $\theta$ separates $\{r,t\}$ from $\{s,u\}$, we see that $\gamma_2$ moves before $\theta$. Also, $\gamma'_2 = \psi_2$ will bound $t$ and $f_1=h_1$. The A-move $\gamma_2 \mapsto \psi_2$ occurs inside a 4-holed sphere with boundaries $f_1=h_1$, $\partial \eta(r)$, $\partial \eta (t)$ and $\theta$. Here, $\gamma_1$ is an arc with both endpoints in $\theta$ that separates $t$ from $f_1=h_1$ and $\partial \eta(r)$. Thus, since $\gamma_2 \cap \gamma_1=\emptyset$, the condition $|\gamma_2 \cap \psi_2|=2$ is equivalent to $|\psi_2 \cap \gamma_1|=2$. 
Now, inside $\psi_2$, we can assume that the curve $h'_1$ bounds $\{p,t\}$. Again, the condition $|\gamma_1 \cap \psi_2|=2$ implies that $|h'_1 \cap \gamma_1|=2$. In particular, there is a shadow $c$ of an arc in $T_{jk}$ connecting $\{p,t\}$ such that $\partial \eta(c) = h'_1$. The condition $|h'_1 \cap \gamma_1|=2$ implies that $c$ intersects $\gamma_1$ once. 
%In the rest of this paragraph we will study the shadows of arcs in $T_{ij}$, $T_{ik}$. 
Focus on the disk component of $\Sigma \setminus \gamma_1$. Here, the arc $c$ is an arc with endpoints in $\gamma_1$ and $\{t\}$. We can repeat the argument in Case 1 and find shadows $a_1$, $a_2$ for arcs in $T_{ji}$ and $b_1$, $b_2$ for arcs in $T_{ik}$ as in Figure \ref{fig_low_44_ending}(b). One of the key properties we obtain is that $a_1 \cup b_1 \cup a_2 \cup b_2$ is a simple closed curve (ignoring the punctures) disjoint from $\gamma_1$ and that intersects $c$ in the puncture $\{t\}$. Then the tuple $(\alpha, \beta, \gamma)=(\{a_1, a_2\}, \{b_1, b_2\}, c)$ satisfies the conditions of Lemma \ref{gen_criterion_destab}, concluding that $\mc{T}$ is an stabilization. 
\end{proof}

%%%%%%%%%%%%%%%%%%%%%%%%%%%%%%%%%%%%%%%%%%%%%%%%%%%%%%%%%%%%%%%%
\begin{theorem}\label{thm_lower_bound}
Let $\mc{T}$ be a $(4,2)$-bridge trisection for a knotted connected surface $S$ in $S^4$. Then $$L(\mc{T})\geq 15.$$ %Moreover if $\mc{T}$ is of type 1 or 2 from Lemma \ref{lem_combinatorics}, then $\L(\mc{T})\geq 15$. 
\end{theorem}

%%%%%%%%%%%%%%%%%
\begin{proof}[Proof of Theorem \ref{thm_lower_bound}]
We first observe that $\mc{T}$ is unstabilized and irreducible. If $\mc{T}$ was stabilized, then $\b(S) \leq 3$. By \cite[Theorem 1.8]{meier2017bridge}, $S$ is unknotted, contradicting our assumption. If $\mc{T}$ is reducible, then by \cite{blair2020kirby}, it is either the distant sum or connected sum of two other trisections. In the former case, this would imply that $F$ is disconnected, a contradiction. In the latter case, the two other trisections have bridge numbers $b_1, b_2\geq 2$ and $b_1 + b_2 - 1 = 4$. Thus, $b_1, b_2 \leq 3$. Again by \cite[Theorem 1.8]{meier2017bridge}, this means both surfaces being trisected are unknotted and so $S$, being their connected sum, is also unknotted. 

Let $(p_{ij}^i, p_{ik}^i)$ for $\{i,j,k\} = \{1,2,3\}$ be choices of efficient pairs so that
\[
\mc{L}(\mc{T}) = d(p_{12}^1, p_{12}^2) + d(p_{13}^1, p_{13}^3)+ d(p_{23}^2, p_{23}^3)
\]
% By Proposition \ref{lem_new_long_proposition}, each $d(p_{ij}^i, p_{ij}^j) \geq 4$, and so $\mc{L}(\mc{T}) \geq 12$. 

% In order to improve the bound, we analyze a bit more. Proposition \ref{prop_longer_proposition} will give us that if $\gamma_1$ bonds two punctures and $\psi_1$ bounds four, then the distance between $p^i_{ij}$ and $p^j_{ij}$ is at least five. Proposition \ref{prop_five_22} gives us the same lower bound when $\gamma_1$ and $\psi_1$ bound two punctures each. 
By Lemma \ref{lem_combinatorics}, the reducing curves of $p^i_{ij}$ and $p^j_{ij}$ either (1) bound two and four punctures each, (2) both bound two punctures, or (3) both bound four punctures. Propositions \ref{prop_longer_proposition}, \ref{prop_five_22}, \ref{prop_five_44} state that $d(p^i_{ij}, p^j_{ij})\geq 5$ on each case. Hence
%Thus, (4,2)-bridge trisections of type 1 and 2 (See Lemma \ref{lem_combinatorics}) have 
$\Lcal(\mc{T})$ is at least $5+5+5=15$
\end{proof}

\begin{corollary}\label{cor_lower_2b}
Let $K\neq U$ be a 2-bridge knot in $S^3$. %The (MZ)-bridge trisection $\mc{T}_{MZ}$ for the (twist) spun $\mc{S}(K)$ satisfies $\mc{L}(\mc{T}_{MZ})\geq 15$. 
The spun $\mc{S}(K)$ satisfies $$\mc{L}(S(K))\geq 15.$$ 
\end{corollary}
\begin{proof}
From Theorem \ref{thm_spun_are_42}, if $\mc{T}$ is a minimal $(b, c_1, c_2, c_2)$-bridge trisection of $S(K)$ then $b=4$ and $c_1=c_2=c_3=2$. By Theorem \ref{thm_lower_bound}, $\mc{L}(\mc{T})\geq 15$. 
\end{proof}

%%%%%%%%%%%%%%%%%%%%%%%%%%%%%%%%%%%%%%%%%%%%%%%%%
\section{Upper bounds for $\Lcal$-invariant of spun knots}

The goal of this section is to build an upper bound for $\Lcal(S(K))$ in terms of the bridge splitting for $K$. Through out this section, $K$ will denote a knot in $b$-bridge position $K=T^+_K\cup T^-_K$ and $\mathcal{T}_{MZ}$ is the $(3b-2,b)$-bridge trisection for the spun of $K$ from Section \ref{section_spun_knots}. 

\begin{example}[$\Lcal$-invariant of spun trefoil]\label{example_L_trefoil}
When $K$ is the trefoil knot, the triplane diagrams from Section \ref{section_spun_knots} give us the links $L_i=T_{ij}\cup \T_{ik}$ in Figure \ref{fig_spun_trefoil_1}. In the same figure, we find particular choices for efficient defining pairs $(p^i_{ij}, p^i_{ik})$ for the link $L_i$ which have bounded distance $d\left(p^i_{ij}, p^j_{ij}\right)\leq 5$ (Figure \ref{fig_spun_trefoil_2}). Thus, $\Lcal(S(K))\leq 15$. One can observe that such paths resemble a particular path in the four punctured sphere (Figure \ref{fig_spun_trefoil_2}(d)). The main idea of this section is to formalize the resemblance and use it to build a general upper bound in Theorem \ref{thm_upper_bbridges}.
\begin{figure}[h!]
\centering
\labellist \small\hair 2pt 
\pinlabel {$T_{12}$} at 40 290
\pinlabel {$\T_{13}$} at 40 220
\pinlabel {$p^1_{12}$} at 00 170
\pinlabel {$p^1_{13}$} at 00 80

\pinlabel {$T_{13}$} at 230 290
\pinlabel {$\T_{23}$} at 230 220
\pinlabel {$p^3_{13}$} at 220 170
\pinlabel {$p^3_{23}$} at 220 80

\pinlabel {$T_{23}$} at 455 290
\pinlabel {$\T_{12}$} at 450 220
\pinlabel {$p^2_{23}$} at 425 170
\pinlabel {$p^2_{12}$} at 425 80

\endlabellist
\includegraphics[width=.8\textwidth]{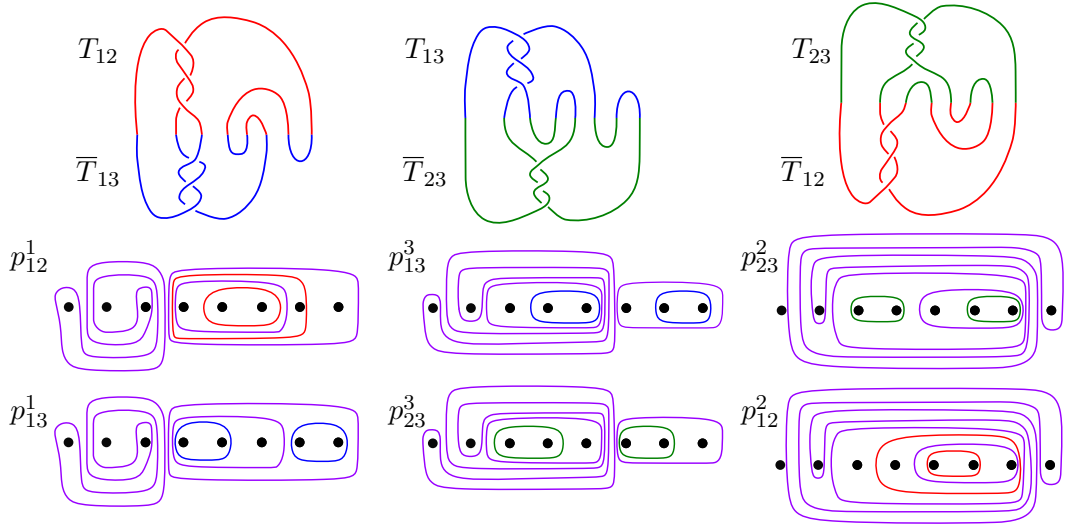}
\caption{Bridge positions and efficient defining pairs for the links $L_i=T_{ij}\cup \T_{ik}$.}
\label{fig_spun_trefoil_1}
\end{figure}
\begin{figure}[h!]
\centering
\labellist \small\hair 2pt 
\pinlabel {(a)} at -5 425
\pinlabel {$p^1_{12}$} [b] at 180 410
\pinlabel {$p^2_{12}$} [b] at 180 1

\pinlabel {(b)} at 220 425
\pinlabel {$p^3_{13}$} [b] at 400 410
\pinlabel {$p^1_{13}$} [b] at 400 20

\pinlabel {(c)} at 460 425
\pinlabel {$p^3_{23}$} [b] at 600 410
\pinlabel {$p^2_{23}$} [b] at 620 1

\pinlabel {(d)} at 680 425
\pinlabel {A} [b] at 855 300

\endlabellist
\includegraphics[width=.95\textwidth]{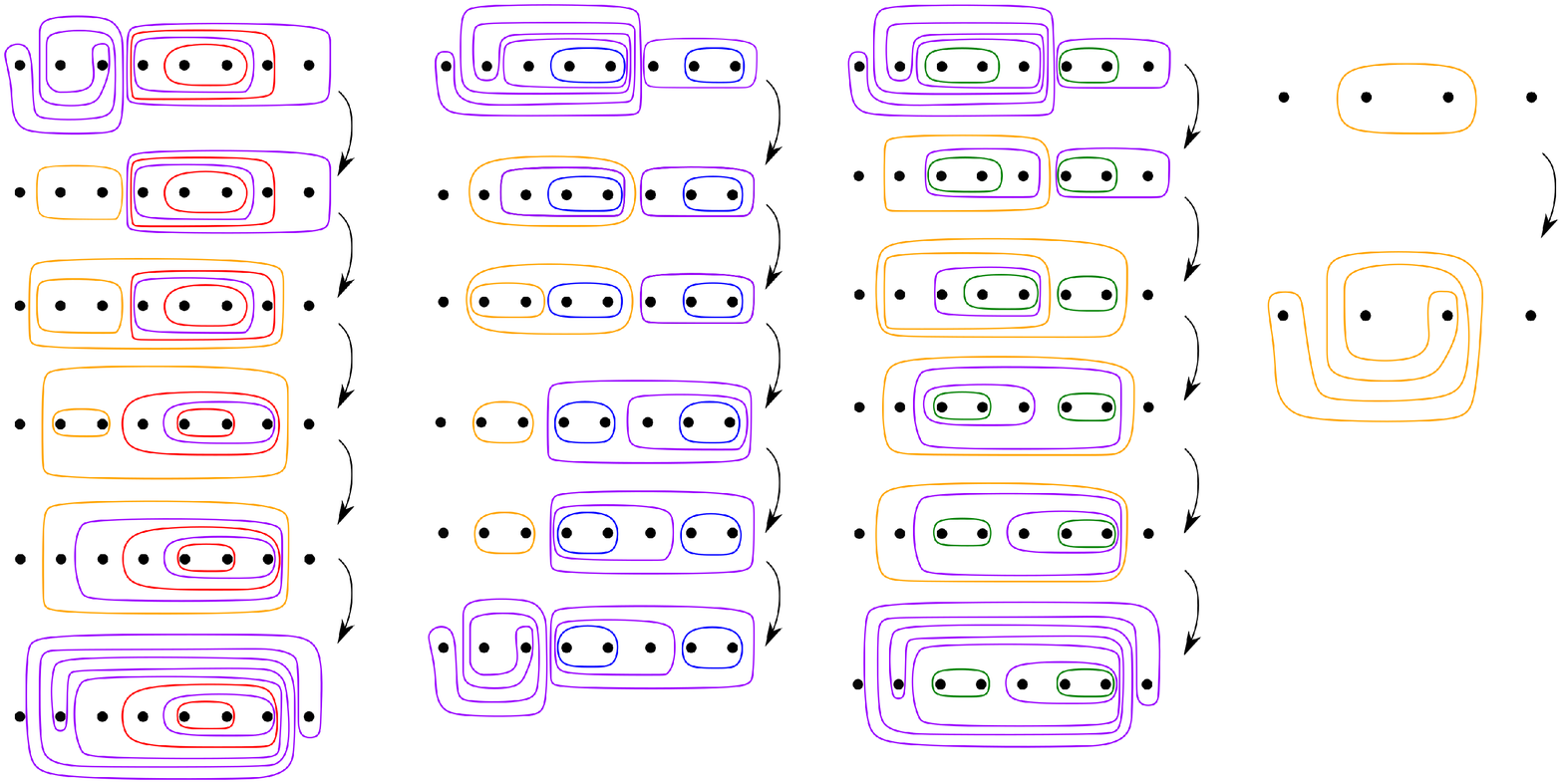}
\caption{Three paths of length five between $p^i_{ij}$ and $p^j_{ij}$.}
\label{fig_spun_trefoil_2}
\end{figure}
\end{example}

Recall that a link $L=L_+ \cup L_-$ in bridge position is perturbed if there is a pair of bridge disks (one on each side) intersecting once in one puncture. This notion is equivalent to the existence of a pair of compressing disks (one per tangle) with boundaries $f_+$ and $f_-$ such that: (1) each $f_\pm$  bounds two punctures, (2) $f_+$ and $f_-$ bound one common puncture, and (3) $|f_+\cap f_-|=2$. Observe that if $c_\pm$ is the shadow for the bridge disk in the perturbation, then $f_\pm = \partial \eta c_\pm$. 

A \textbf{perturbation system} is a finite collection of perturbation pairs $\{(c^n_-, c^n_+)\}_n$ with pairwise disjoint interiors such that $\bigcup_n (c^n_+\cup c^n_-)$ contains no circles in the bridge surface. In other words, it is a collection of perturbations that can be undone at the same time. Figure \ref{fig_three_links} contains examples of perturbation systems. As submanifolds of the bridge surface, the loops $\partial \eta \left( \bigcup_n (c^n_+ \cup c^n_-) \right)$ bound disks c-disks for $L$ in both sides. We will refer to these curves (resp. spheres) in the bridge surface (resp. $S^3$) as \textbf{sensor} curves (resp. spheres) since they allow us to think of $L$ as a link with lower bridge number. 
\begin{figure}[h!]
\centering
\labellist \small\hair 2pt 
\pinlabel {$12$} at 5 240 
\pinlabel {$\overline{13}$} at 5 10
\pinlabel {$\beta$} at 180 185
\pinlabel {$\overline\beta$} at 180 40

\pinlabel {$13$} at 690 240 
\pinlabel {$\overline{23}$} at 690 10
\pinlabel {$\beta$} at 550 185
\pinlabel {$\overline\beta$} at 550 40

\pinlabel {$23$} at 1070 240 
\pinlabel {$\overline{12}$} at 1070 10
\pinlabel {$\beta$} at 930 185
\pinlabel {$\overline\beta$} at 930 40

\endlabellist
\includegraphics[width=.9\textwidth]{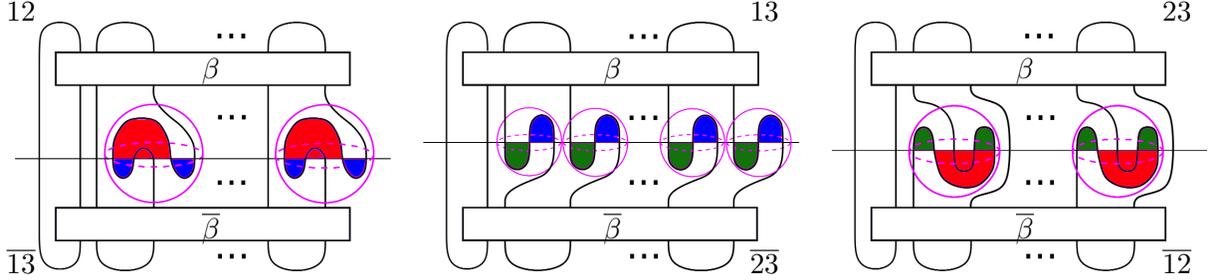}
\caption{Bridge presentations for the links $L_{\eps, \bar\delta} = T_{\eps} \cup \T_\delta$.}
\label{fig_three_links}
\end{figure}

For the $b$-bridge links in Figure \ref{fig_three_links}, the perturbation systems will determine two  simplicial maps between pants complexes $\Pcal(\Sigma_{2b}) \ra \Pcal(\Sigma_{6b-4})$. The main idea of the upper bound for $\Lcal(\mc{T}_{MZ})$ is to induce paths in $\Pcal(\Sigma_{6b-4})$ using information from the splitting of the knot $K$.

Fix $(\eps, \delta, \rho)$ to be a cyclic permutation of the labels $(12, 13, 23)$. Focus on the link $L_{\eps, \bar{\delta}}= T_\eps \cup \T_{\bar{\delta}}$ and the perturbation system in Figure \ref{fig_three_links}. If we shrink the sensor spheres to a point by collapsing the 3-ball containing the perturbation disks, we obtain a link isotopic to $L_{\eps, \bar{\delta}}$ in $b$-bridge position. At the level of the bridge surfaces, this collapsing induces a continuous map between the punctured spheres $g_{\eps, \bar\delta}: \Sigma_{6b-4} \ra \Sigma_{2b}$. Given a pants decomposition $p\in \Pcal(\Sigma_{2b})$, define the following sets of curves $G^\pm_{\eps, \bar \delta}(p) = g^{-1}_{\eps, \bar \delta} (p) \cup \mu^{\pm}_{\eps, \bar \delta} \cup \phi_{\eps, \bar \delta}$, where $\mu^{\pm}_{\eps, \bar\delta}$ and $\phi_{\eps, \bar\delta}$ are collections of curves described in Figure \ref{fig_the_map_G}. By construction, both $G^\pm_{\eps, \bar\delta}(p)$ are a pants decompositions of $\Sigma_{6b-4}$. Furthermore, the functions $\{G^\pm_{\eps, \bar\delta}\}_{(\eps, \bar\delta)}$ satisfy several properties described in the following lemma.
\begin{figure}[h!]
\centering
\labellist \small\hair 2pt 
\pinlabel {$K$:} at -20 237  
\pinlabel {$\left(12,\overline{13}\right)$:} at -40 190  
\pinlabel {$\left(13,\overline{23}\right)$:} at -40 115  
\pinlabel {$\left(23,\overline{12}\right)$:} at -40 42  

\pinlabel {$1$} at 12 247  
\pinlabel {$2$} at 39 247  
\pinlabel {$3$} at 70 247  
\pinlabel {$4$} at 104 247  

\pinlabel {$5$} at 102 190 
\pinlabel {$5$} at 190 190 
\pinlabel {$5$} at 257 190 
\pinlabel {$3$} at 102 115 
\pinlabel {$3$} at 190 115 
\pinlabel {$3$} at 67 115 
\pinlabel {$3$} at 155 115 
\pinlabel {$3$} at 257 115 
\pinlabel {$5$} at 70 42 
\pinlabel {$5$} at 160 42
\pinlabel {$5$} at 257 42 

\pinlabel {$\mu^+_{12,\overline{13}}$} at 455 245 
\pinlabel {$\phi_{12,\overline{13}}$} at 365 245
\pinlabel {$\mu^-_{12,\overline{13}}$} at 500 140

\pinlabel {$\mu^-$} at 355 115
\pinlabel {$\mu^+$} at 425 115
\pinlabel {$\phi_{13,\overline{23}}$} at 320 85

\pinlabel {$\mu^-_{23,\overline{12}}$} at 500 95
\pinlabel {$\mu^+_{23,\overline{12}}$} at 425 -05
\pinlabel {$\phi_{23,\overline{12}}$} at 545 100

\tiny\pinlabel {$2b-1$} at 160 247 
\pinlabel {$2b-2$} at 200 247

\endlabellist
\includegraphics[width=.8\textwidth]{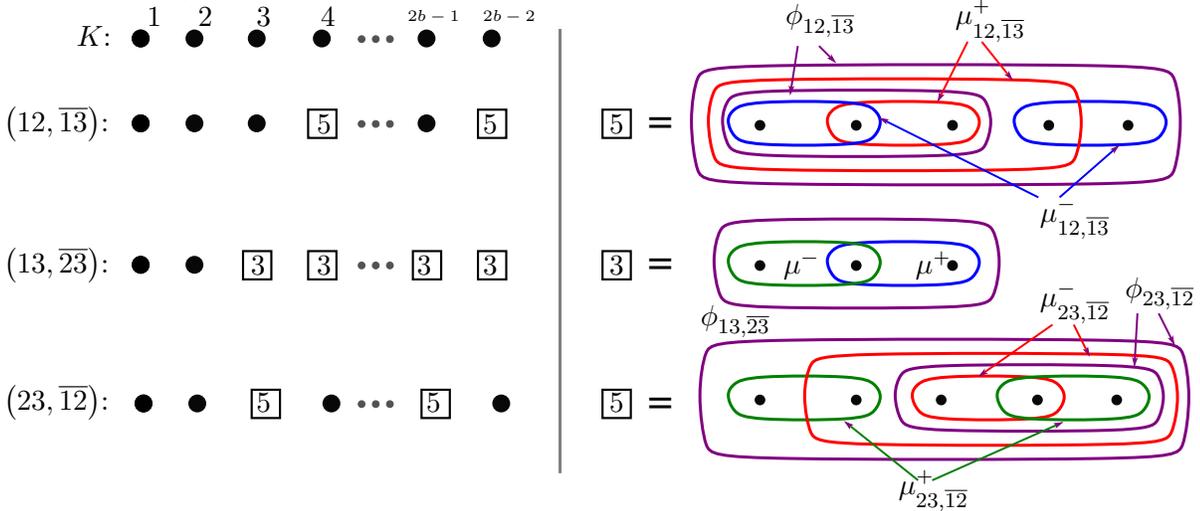}
\caption{Curves that complete $G^\pm_{\eps, \bar\delta}$, we removed the indices in the right side of the figure.}
\label{fig_the_map_G}
\end{figure}
\begin{lemma}\label{lem_props_G_map}
Let $(\eps, \delta, \rho)$ be a cyclic permutation of $(12,13,23)$ and let $p_0$ and $p_1$ be any two pants decompositions of $\Sigma_{2b}$. The following holds: 
\begin{enumerate}
\item $G^\pm_{\eps, \bar\delta}: \Pcal(\Sigma_{2b}) \ra \Pcal(\Sigma_{6b-4})$ is a 1-simplicial map; in other words, if $\lambda\subset \Pcal(\Sigma_{2b})$ is a path from $p_0$ to $p_1$, then $G^\pm_{\eps, \bar\delta}(\lambda)$ is a path connecting $G^\pm_{\eps, \bar\delta}(p_0)$ and $G^\pm_{\eps, \bar\delta}(p_1)$. 
\item If every loop in $p_0$ bounds a c-disk in $T^+_K$, then the tuple $\left(G^+_{\eps, \bar\delta}(p_0), G^-_{\eps, \bar\delta}(p_0)\right)$ is an efficient pair for the link $T_{\eps} \cup \T_{\delta}$.
\item If every loop in $p_1$ bounds a compressing disk for $T^-_K$, then the distance in $\Pcal(\Sigma_{6b-4})$ between $G^+_{\eps,\bar\delta}(p_1)$ and $G^-_{\rho, \bar\eps}(p_1)$ is $2(b-1)$. 
\end{enumerate}
\end{lemma}

\begin{proof}
Part 1 follows from the definition of $G^\pm_{\eps, \bar\delta}$. 
In order to prove Part 2, we first observe that $G^+_{\eps, \bar\delta}(p_0)$ and $G^-_{\eps, \bar\delta}(p_0)$ are pants decompositions with looks bounding c-disks in $T_\eps$ and $T_\delta$, respectively. The loops in $\mu^\pm_{\eps, \bar\delta}$ arise from perturbation pairs and the ones in $\phi_{\eps, \bar\delta}$ from sensor loops (see Figure \ref{fig_three_links}). Thus they bound c-disks. The extra assumption in $p_0$ implies that $g^{-1}_{\eps, \bar\delta}(p_0)$ also bounds c-disks. 
Next, one can see from Figure \ref{fig_the_map_G} that the loops in $\mu^+_{\eps, \bar\delta}$ and $\mu^-_{\eps, \bar\delta}$ can be paired so that they intersect in two points and are disjoint from the rest. Thus, there is a path in $\Pcal(\Sigma_{6g-4})$ of length $2b-2$. Lemma \ref{lem_efficient_pairs} concludes that $\left(G^+_{\eps, \bar\delta}(p_0), G^-_{\eps, \bar\delta}(p_0)\right)$ is an efficient pair. 

We will now discuss Part 3. Label the punctures in the bridge sphere for $K$ as in the left side of Figure \ref{fig_the_map_G}. In particular, since every loop in $p_1$ bonds a compressing disk for $T^-_K$, we get that the pairs of punctures $\{2n-1, 2n\}$ belong to the same component of $\Sigma_{2b}\setminus p_1$ for $n=1, \dots, b$. We denote such collection of loops by $B\subset p_1$. After an isotopy of the bridge surface for $K$, which changes the surface by a homeomorphism fixing the punctures, we can assume that the loops in $B$ look as in Figure \ref{fig_outer_path}. 
Observe that this isotopy of $K$ does not affect the class of bridge trisection $\mathcal{T}_{MZ}$; more precisely, it changes the triplane diagrams by a pure braid. We can then consider the pants decompositions $G^+_{\eps, \bar \delta}(p_1)$ and $G^-_{\rho, \bar \eps}(p_1)$ and see that the loops in $g^{-1}_{\eps, \bar\delta}(p_1)$ and $g^{-1}_{\rho, \bar\eps}(p_1)$ agree. We also observe that the loops in $\mu^+_{\eps, \bar\delta}$ and $\mu^-_{\rho, \bar\delta}$ are the same since their corresponding bridge disks agree (see Figure \ref{fig_three_links}). To end, we can perform the length two path of A-moves described by Figure \ref{fig_outer_path} near each loop in $B$ ($b-1$ times), and find a path in $\Pcal(\Sigma_{6b-4})$ replacing the loops $\phi_{\eps, \bar\delta}$ by the loops $\phi_{\rho, \bar\eps}$. Thus  the distance in $\Pcal(\Sigma_{6b-4})$ between $G^+_{\eps,\bar\delta}(p_1)$ and $G^-_{\rho, \bar\eps}(p_1)$ is at most $2(b-1)$. Since the sets of curves $\phi_{\eps, \bar\delta}$ and $\phi_{\rho, \bar\eps}$ have no common curve, we conclude that this path is minimal length. 
\end{proof}
\begin{figure}[h!]
\centering
\labellist \small\hair 2pt 
\pinlabel {$\Sigma_{2b}$} at 130 340 
\pinlabel {$1$} at 173 335
\pinlabel {$2$} at 187 335
\pinlabel {$3$} at 231 335
\pinlabel {$4$} at 243 335
%\pinlabel {$2b-1$} at 555 335 
\pinlabel {$2b$} at 328 336
\pinlabel {$B$} at 155 355
\pinlabel {$B$} at 225 355
\pinlabel {$B$} at 305 355

\pinlabel {$\left( 12, \overline{13}\right)^+$} at -7 275
\pinlabel {$B$} at 120 275
\pinlabel {$B$} at 120 5
\pinlabel {$\left( 23, \overline{12}\right)^-$} at -7 5

\pinlabel {$\left( 23, \overline{12}\right)^+$} at 350 275
\pinlabel {$B$} at 210 275
\pinlabel {$B$} at 210 5
\pinlabel {$\left( 13, \overline{23}\right)^-$} at 350 5

\pinlabel {$\left( 13, \overline{23}\right)^+$} at 545 275
\pinlabel {$B$} at 410 275
\pinlabel {$B$} at 410 5
\pinlabel {$\left( 12, \overline{13}\right)^-$} at 545 5

\endlabellist
\includegraphics[width=.7\textwidth]{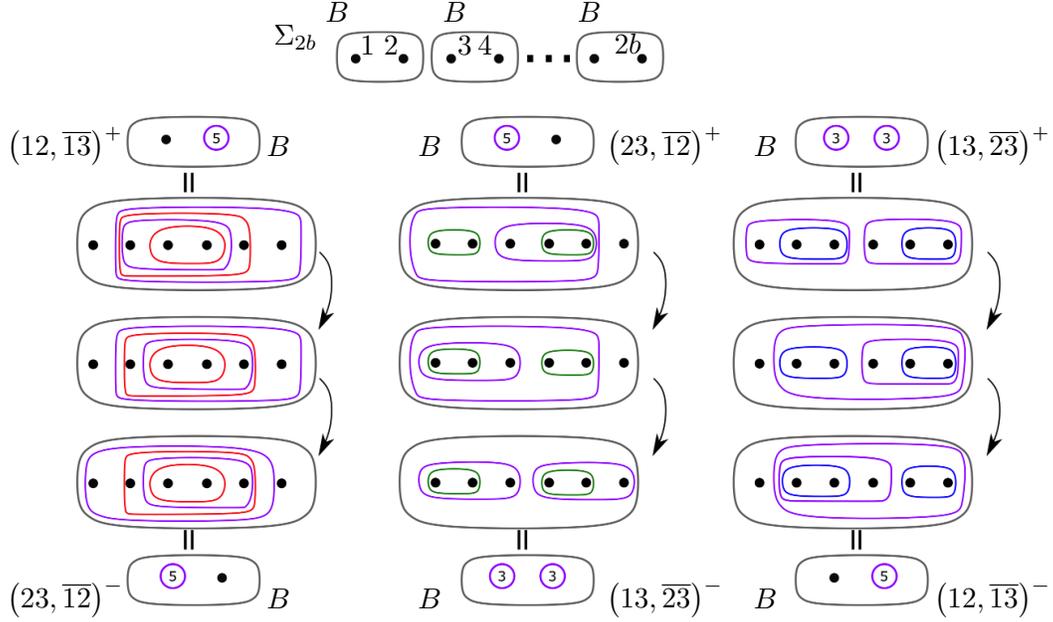}
\caption{If we perform the sequence of A-moves inside each component of $B$, we obtain paths of length $2(b-1)$ connecting $\phi_{\eps, \bar\delta} \mapsto \phi_{\rho, \bar\eps}$.}
\label{fig_outer_path}
\end{figure}

Motivated by Lemma \ref{lem_props_G_map}, for a trivial $N$-tangle $T$, we define $\Pcal_{comp}(T)$ and $\Pcal_{c}(T)$ to be the sets of pants decompositions $p\in \Pcal(\Sigma_{2N})$ such that all loops in $p$ bound compressing disks and c-disks, respectively. The upper bound in the following Theorem can be summarized in Figure \ref{fig_outer_path}. 

\begin{theorem}\label{thm_upper_bbridges}
Let $K=T^+_K \cup T^-_K$ be a knot in $b$-bridge position and let $\mathcal{T}_{MZ}$ be the $(3b-2,b)$-bridge trisection for the spun 2-knot $S(K)\subset S^4$. Let $d\geq 0$ be the distance in $\Pcal(\Sigma_{2b})$ between the sets $\Pcal_c(T^+_K)$ and $\Pcal_{comp}(T^-_K)$. Then 
\[ \Lcal(\mathcal{T}_{MZ})\leq 6(d+b-1).\]
\end{theorem}

\begin{proof}
Let $p_0 \in \Pcal_{c}(T^+_K)$ and $p_1 \in \Pcal_{comp}(T^-_K)$ be pants decompositions realizing the distance $d$, and let $\lambda\subset \Pcal(\Sigma_{6b-4})$ be a geodesic path connecting them. In particular, $p_0$ and $p_1$ satisfy the conclusions of Lemma \ref{lem_props_G_map} for any cyclic permutation $(\eps, \delta, \rho)$ of $(12,13,23)$. Now, consider the loop in $\Pcal(\Sigma_{6b-4})$ described in Figure \ref{fig_upper_bound}. By Lemma \ref{lem_props_G_map}, this loop satisfies the conditions in the definition of $\Lcal(\mathcal{T}_{MZ})$. Since each $G^\pm_{\eps, \bar\delta}(\lambda)$ is a path of length $d$, we can conclude the desired inequality. 
\begin{figure}[h]
\centering
\labellist \small\hair 2pt 

\pinlabel {${T_{12}}$} at 105 195
\pinlabel {{\color{white}$T_{13}$}} at 215 195
\pinlabel {{\color{white}$T_{23}$}} at 160 105

\pinlabel {$G^+_{12,\overline{13}}(p_0)$} at 135 240
\pinlabel {$G^+_{12,\overline{13}}(p_1)$} at 110 305

\pinlabel {$G^-_{12,\overline{13}}(p_0)$} at 255 230
\pinlabel {$G^-_{12,\overline{13}}(p_1)$} at 230 305

\pinlabel {$G^+_{13,\overline{23}}(p_0)$} at 275 180
\pinlabel {$G^+_{13,\overline{23}}(p_1)$} at 330 128

\pinlabel {$G^-_{13,\overline{23}}(p_0)$} at 220 100
\pinlabel {$G^-_{13,\overline{23}}(p_1)$} at 265 28

\pinlabel {$G^+_{23,\overline{12}}(p_0)$} at 100 100
\pinlabel {$G^+_{23,\overline{12}}(p_1)$} at 58 28

\pinlabel {$G^-_{23,\overline{12}}(p_0)$} at 50 180
\pinlabel {$G^-_{23,\overline{12}}(p_1)$} at 28 128

\pinlabel {
\begin{rotate}{55}
$2b-2$
\end{rotate}
}
at -5 240

\pinlabel {
\begin{rotate}{-60}
$2b-2$
\end{rotate}
}
at 320 250

\pinlabel {$2b-2$} at 155 15

\endlabellist
\includegraphics[width=.6\textwidth]{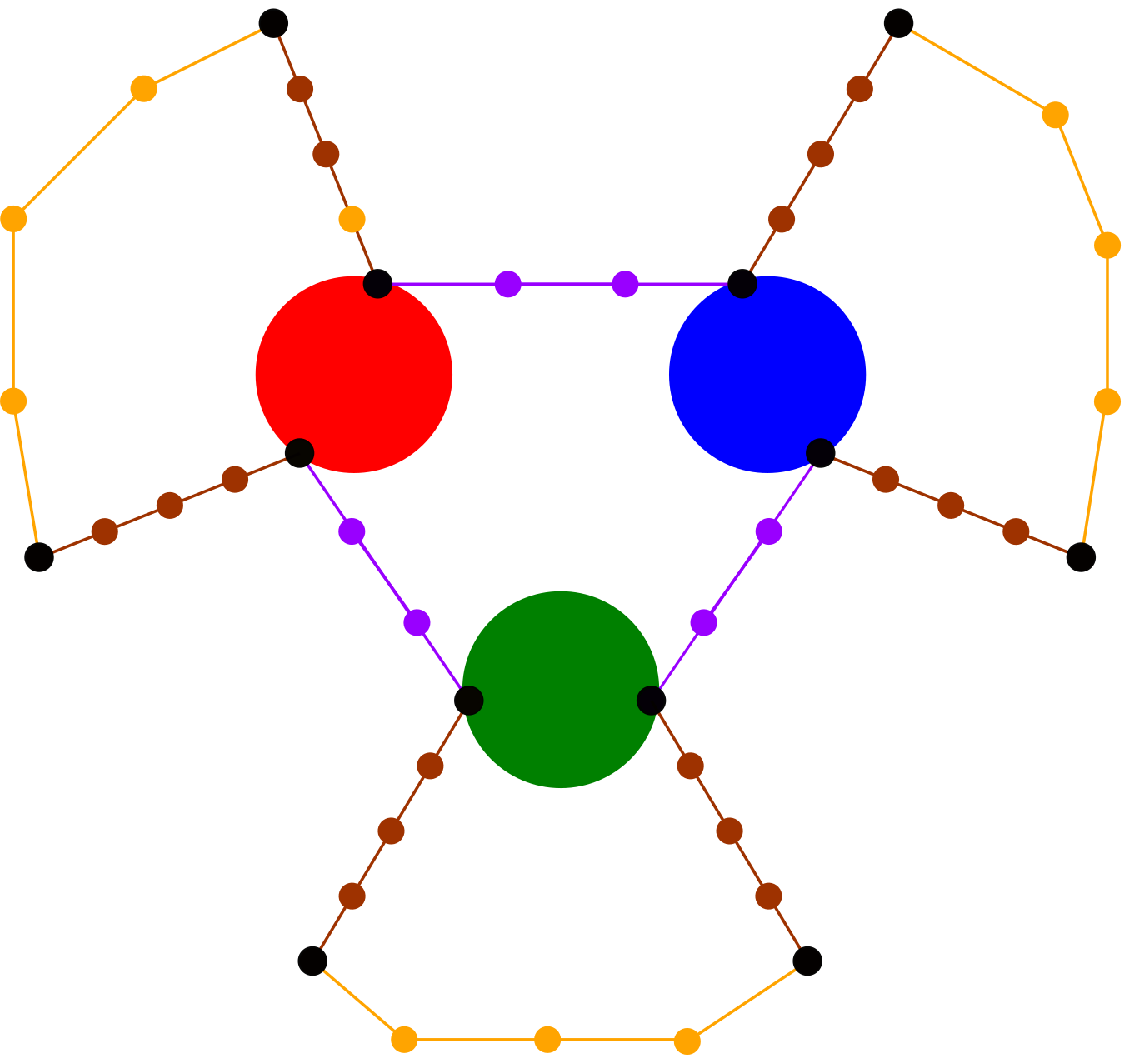}
\caption{Upper bound for $\Lcal(\mathcal{T})$.}
\label{fig_upper_bound}
\end{figure}
\end{proof}

\begin{remark}\label{remark_general_ub}
From %the proof of Theorem \ref{thm_upper_bbridges} and
Figure \ref{fig_upper_bound}, we can derive a more general upper bound for $\Lcal(\mathcal{T}_{MZ})$ as follows: 
If $p_0, p_1\in \Pcal(\Sigma_{2b})$ are pants decompositions with $p_0\in \Pcal_c(T^+_K)$, then 
\[ \Lcal(\mathcal{T}_{MZ})\leq 6d(p_0, p_1) + 
d(G^-_{12,\bar{13}}(p_1), G^+_{13,\bar{23}}(p_1)) +
d(G^-_{13,\bar{23}}(p_1), G^+_{23,\bar{12}}(p_1)) +
d(G^-_{23,\bar{12}}(p_1), G^+_{12,\bar{13}}(p_1)).\]
\end{remark}
The following Corollary studies the distance between $G^+_{\eps,\bar\delta}(p_1)$ and $G^-_{\rho, \bar\eps}(p_1)$ for families of pants decompositions other than $\Pcal_{comp}(T^-_K)$. 
We use Conway's notation \cite{2_bridge_Kauffman, 2_bridge_mulazzani, conway} to describe 2-bridge links. %The link with Conway number $0$ (resp. $\infty$) is the 2-bridge link such that the curve in the top of Figure \ref{fig_outer_path} (resp. Figure \ref{fig_outer_path_2}) bounds a compressing disk on both sides.
The curve in the top of Figure \ref{fig_outer_path} (resp. Figure \ref{fig_outer_path_2}) bounds a compressing disk on both sides of the 2-bridge link with Conway number $0$ (resp. $\infty$). 
The distance below can be computed using continued fraction expansions of $p/q$ \cite{hatcherfarey}. 

\begin{corollary}\label{cor_upper_2bridges}
Let $K\subset S^3$ be a 2-bridge knot with Conway number $p/q$. We have 
\[ \Lcal(\mc{T}_{MZ}) \leq \min \big\{ 6d(p/q,0) + 6,  6d(p/q,\infty) + 9\big\}.\] 
\end{corollary}
\begin{proof}
For 2-bridge knots, the only curve bounding a compressing disk in $T^-_K$ (resp. $T^+_K$) is the loop of slope $0$ (resp. $p/q$) in the 4-punctured bridge sphere. Furthermore, there are no cut disks for $T^+_K$ since $b$ is small. The first inequality $\Lcal(\mc{T}_{MZ}) \leq 6d(p/q,0) + 6$ follows from Theorem \ref{thm_upper_bbridges}. 

In order to prove the second inequality, we consider $p_1\subset \Pcal(\Sigma_4)$ corresponding to the curve $B\subset \Sigma_4$ with slope $\infty$ in Figure \ref{fig_outer_path_2}. In the same figure, we observe that the distance between the pants decompositions $G^+_{\eps,\bar\delta}(p_1)$ and $G^-_{\rho, \bar\eps}(p_1)$ is bounded by three. By Remark \ref{remark_general_ub}, we conclude $\Lcal(S(K)) \leq 6d(p/q,\infty) + 3\cdot 3$, as desired. 
\begin{figure}[h!]
\centering
\labellist \small\hair 2pt 

\pinlabel {$\Sigma_4$} at 190 390
\pinlabel {$1$} at 250 390
\pinlabel {$2$} at 315 390
\pinlabel {$3$} at 335 390
\pinlabel {$4$} at 380 390
\pinlabel {$B$} at 280 370

\pinlabel {$\left(12,\overline{13}\right)^+$} at 190 340
\pinlabel {$\left(23,\overline{12}\right)^-$} at 190 25

\pinlabel {$\left(23,\overline{12}\right)^+$} at 435 340
\pinlabel {$\left(13,\overline{23}\right)^-$} at 435 25

\pinlabel {$\left(13,\overline{23}\right)^+$} at 665 340
\pinlabel {$\left(12,\overline{13}\right)^-$} at 665 25

\endlabellist
\includegraphics[width=.8\textwidth]{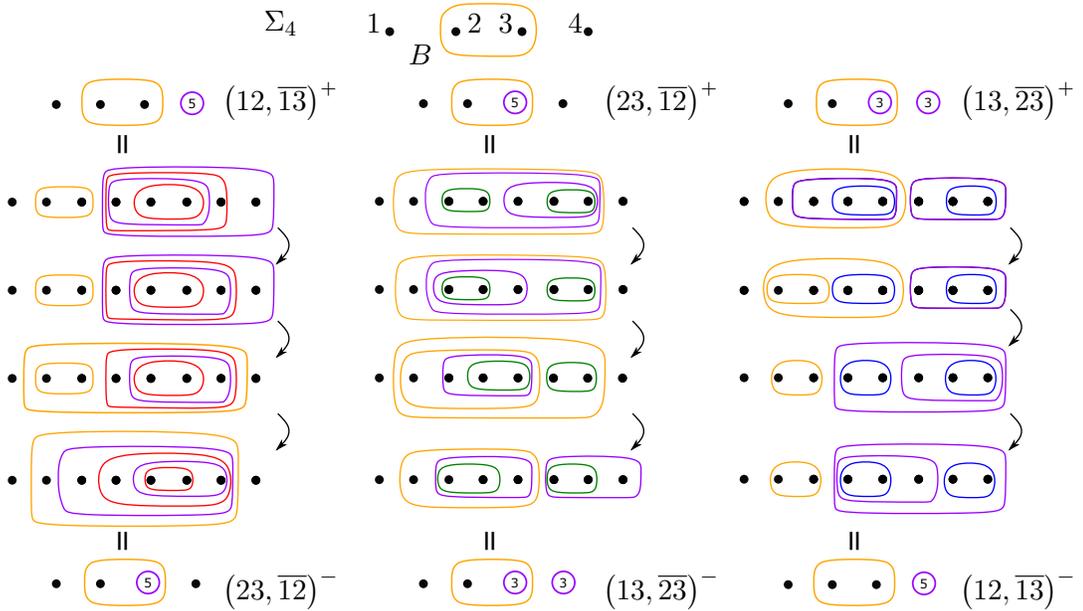}
\caption{Paths of length three between $G^+_{\eps,\bar\delta}(p_1)$ and $G^-_{\rho, \bar\eps}(p_1)$.}
\label{fig_outer_path_2}
\end{figure}
\end{proof}

\bibliographystyle{plain}
{ \small
\bibliography{L_invariant_spun_knots}

\begin{thebibliography}{10}

\bibitem{artin1925isotopie}
Emil Artin.
\newblock Zur isotopie zweidimensionaler fl{\"a}chen imr4.
\newblock {\em Abhandlungen aus dem Mathematischen Seminar der Universit{\"a}t
  Hamburg}, 4:174--177, 1925.

\bibitem{bachman2005distance}
David Bachman and Saul Schleimer.
\newblock Distance and bridge position.
\newblock {\em Pacific journal of mathematics}, 219(2):221--235, 2005.

\bibitem{blair2020kirby}
Ryan Blair, Marion Campisi, Scott~A Taylor, and Maggy Tomova.
\newblock Kirby-thompson distance for trisections of knotted surfaces.
\newblock {\em Journal of the London Mathematical Society}, 2021.

\bibitem{boileau1989pi}
Michel Boileau and Bruno Zimmermann.
\newblock The $\pi$-orbifold group of a link.
\newblock {\em Mathematische Zeitschrift}, 200(2):187--208, 1989.

\bibitem{conway}
J.~H. Conway.
\newblock An enumeration of knots and links, and some of their algebraic
  properties.
\newblock In {\em Computational {P}roblems in {A}bstract {A}lgebra ({P}roc.
  {C}onf., {O}xford, 1967)}, pages 329--358. Pergamon, Oxford, 1970.

\bibitem{gay2016trisecting}
David Gay and Robion Kirby.
\newblock Trisecting 4--manifolds.
\newblock {\em Geometry \& Topology}, 20(6):3097--3132, 2016.

\bibitem{gay2018doubly}
David Gay and Jeffrey Meier.
\newblock Doubly pointed trisection diagrams and surgery on 2-knots, 2018.

\bibitem{hatcherfarey}
Allen Hatcher.
\newblock Topology of numbers.
\newblock {\em Unpublished manuscript, in preparation}, 2002.

\bibitem{2_bridge_Kauffman}
Louis~H. Kauffman and Sofia Lambropoulou.
\newblock On the classification of rational tangles.
\newblock {\em Adv. in Appl. Math.}, 33(2):199--237, 2004.

\bibitem{kirby2018new}
Robion Kirby and Abigail Thompson.
\newblock A new invariant of 4-manifolds.
\newblock {\em Proceedings of the National Academy of Sciences},
  115(43):10857--10860, 2018.

\bibitem{PLC_thom}
Peter Lambert-Cole.
\newblock Bridge trisections in {$\Bbb{CP}^2$} and the {T}hom conjecture.
\newblock {\em Geom. Topol.}, 24(3):1571--1614, 2020.

\bibitem{weinstein_trisections}
Peter Lambert-Cole, Jeffrey Meier, and Laura Starkston.
\newblock Symplectic 4-manifolds admit {W}einstein trisections.
\newblock {\em J. Topol.}, 14(2):641--673, 2021.

\bibitem{lee2017reduction}
Jung~Hoon Lee.
\newblock Reduction of bridge positions along bridge disks.
\newblock {\em Topology and its Applications}, 223:50--59, 2017.

\bibitem{meier2017bridge}
Jeffrey Meier and Alexander Zupan.
\newblock Bridge trisections of knotted surfaces in {$S^4$}.
\newblock {\em Transactions of the American Mathematical Society},
  369(10):7343--7386, 2017.

\bibitem{meier2018bridge}
Jeffrey Meier and Alexander Zupan.
\newblock Bridge trisections of knotted surfaces in 4-manifolds.
\newblock {\em Proceedings of the National Academy of Sciences},
  115(43):10880--10886, 2018.

\bibitem{2_bridge_mulazzani}
Michele Mulazzani and Andrei Vesnin.
\newblock The many faces of cyclic branched coverings of 2-bridge knots and
  links.
\newblock {\em Atti Sem. Mat. Fis. Univ. Modena}, 49(suppl.):177--215, 2021.
\newblock Dedicated to the memory of Professor M. Pezzana (Italian).

\bibitem{zupan2013bridge}
Alexander Zupan.
\newblock Bridge and pants complexities of knots.
\newblock {\em Journal of the London Mathematical Society}, 87(1):43--68, 2013.

\end{thebibliography}
}

%\begin{flushright}
$\quad$ \\
Jos\'e Rom\'an Aranda, Binghamton University\\
email: \texttt{jaranda@binghamton.edu}\\ 
$\quad$ \\
Puttipong Pongtanapaisan, University of Saskatchewan\\
email: \texttt{puttipong@usask.ca}\\
$\quad$ \\
Scott A. Taylor, Colby College \\ 
email: \texttt{sataylor@colby.edu}\\
$\quad$ \\
Suixin (Cindy) Zhang, Colby College\\
email: \texttt{szhang22@colby.edu}

%\end{flushright}

\end{document}